\documentclass[11pt]{amsart}  

\usepackage{amsbsy,amsmath,amsthm,amssymb,color,verbatim}
\usepackage[normalem]{ulem}
\usepackage{mathtools, mathdots}
\usepackage{soul}
\usepackage{cancel,shuffle}
\usepackage{stackrel}
\usepackage{fdsymbol}
\usepackage{multicol}
\usepackage[all]{xy}
\usepackage{xargs}

\usepackage{float}
\usepackage{bbm}
\usepackage{fullpage,cancel}
\usepackage[breaklinks=true]{hyperref}
\usepackage{cleveref}
\hypersetup{
    colorlinks=true,
    linkcolor=blue!90!black,  
    citecolor=green!75!black, 
    urlcolor=teal             
}
\usepackage{float}
\usepackage{wrapfig}
\usepackage[dvipsnames]{xcolor}
\usepackage{graphicx}
\usepackage{subcaption}
\usepackage{multicol}
\usepackage{multirow}
\usepackage{longtable}
\usepackage{makecell}
\usepackage{supertabular}
\usepackage{tikz}
\definecolor{darkgreen}{cmyk}{.9,0,.9,.2}
\definecolor{midgray}{gray}{0.60}
\definecolor{lightgray}{gray}{0.90}
\definecolor{lmgray}{gray}{0.70}
\usepackage{colortbl}
\usepackage{standalone}
\usetikzlibrary{arrows.meta,calc,decorations.markings,math,arrows.meta}

\tikzset{->-/.style={decoration={
  markings,
  mark=at position #1 with {\arrow{>}}},postaction={decorate}}}
\tikzset{-<-/.style={decoration={
  markings,
  mark=at position #1 with {\arrow{<}}},postaction={decorate}}}

\usepackage{algorithm}
\usepackage{algpseudocode}
\usepackage{stackengine,wasysym}
\usepackage{lipsum}
\newcommand\m[1]{\begin{pmatrix}#1\end{pmatrix}}
\newcommand{\mathdash}{\relbar\mkern-9mu\relbar}

\makeatletter
\newcommand*{\da@rightarrow}{\mathchar"0\hexnumber@\symAMSa 4B }
\newcommand*{\da@leftarrow}{\mathchar"0\hexnumber@\symAMSa 4C }
\newcommand*{\xdashrightarrow}[2][]{%
  \mathrel{%
    \mathpalette{\da@xarrow{#1}{#2}{}\da@rightarrow{\,}{}}{}%
  }%
}
\newcommand{\xdashleftarrow}[2][]{%
  \mathrel{%
    \mathpalette{\da@xarrow{#1}{#2}\da@leftarrow{}{}{\,}}{}%
  }%
}
\newcommand*{\da@xarrow}[7]{%
  \sbox0{$\ifx#7\scriptstyle\scriptscriptstyle\else\scriptstyle\fi#5#1#6\m@th$}%
  \sbox2{$\ifx#7\scriptstyle\scriptscriptstyle\else\scriptstyle\fi#5#2#6\m@th$}%
  \sbox4{$#7\dabar@\m@th$}%
  \dimen@=\wd0 %
  \ifdim\wd2 >\dimen@
    \dimen@=\wd2 %
  \fi
  \count@=2 %
  \def\da@bars{\dabar@\dabar@}%
  \@whiledim\count@\wd4<\dimen@\do{%
    \advance\count@\@ne
    \expandafter\def\expandafter\da@bars\expandafter{%
      \da@bars
      \dabar@ 
    }%
  }%
  \mathrel{#3}%
  \mathrel{%
    \mathop{\da@bars}\limits
    \ifx\\#1\\%
    \else
      _{\copy0}%
    \fi
    \ifx\\#2\\%
    \else
      ^{\copy2}%
    \fi
  }%
  \mathrel{#4}%
}
\makeatother

\newtheorem{thm}{Theorem}[section] 
\newtheorem*{thm*}{Preview of Main Theorem}
\newtheorem{prop}[thm]{Proposition}
\newtheorem{lem}[thm]{Lemma} 
\newtheorem{cor}[thm]{Corollary}

\theoremstyle{definition}                 
\newtheorem{defn}[thm]{Definition}

\theoremstyle{remark}                      
\newtheorem{exmp}[thm]{Example}  
\newtheorem{rem}[thm]{Remark}

\numberwithin{equation}{section}

\def\multiset#1#2{\ensuremath{\left(\kern-.3em\left(\genfrac{}{}{0pt}{}{#1}{#2}\right)\kern-.3em\right)}}

\def\multiset#1#2{\ensuremath{\left(\kern-.3em\left(\genfrac{}{}{0pt}{}{#1}{#2}\right)\kern-.3em\right)}}

\usepackage[backend=biber,
style=alphabetic,
maxalphanames=4,
sorting=nyt,
maxnames=99]{biblatex}
\appto{\bibsetup}{\sloppy}

\title{Matchings and Clusters on Plabic Fences}
\author{Jo\~ao Pedro Carvalho}
\address[J. P. Carvalho]{Department of Mathematics, University of Michigan, Ann Arbor, MI 48109}
\email{\textcolor{blue}{\href{mailto:jpcarv@umich.edu}{jpcarv@umich.edu}}}
\author{Yucong Lei}
\address[Y. Lei]{Department of Mathematics, University of Michigan, Ann Arbor, MI 48109}
\email{\textcolor{blue}{\href{mailto:leiyc@umich.edu}{leiyc@umich.edu}}}

\begin{document}
\maketitle
\begin{abstract}
Fix two positive braid words $\beta_+,\beta_-$, and let $\text{Conf}(\beta_+,\beta_-)$ be the corresponding (type A) double Bott-Samelson variety. Let $\beta$ be a double braid word containing $\beta_+,\beta_-$ as the top, bottom words. We consider the open cluster torus $T(C_\beta)$ associated to a triangulation $C_\beta$ in $\text{Conf}(\beta_+,\beta_-)$ from \cite{Shen_Weng_2021}, and we identify these with weighted plabic fences, where the usual local moves on planar bipartite graphs naturally correspond to change of torus coordinates in double Bott-Samelson variety. Moreover, $T(C_\beta)$ can be parametrized explicitly by certain matrix products, and also has monomial coordinates given by the cluster variables, which are matrix minors. We interpret these minors as a generalized notion of perfect matchings on weighted plabic fences, which may not be reduced plabic graphs. Using this interpretation, we show that the cluster variables are given by ``generalized minimal matchings", extending the minimal matchings from ~\cite{Muller_Speyer_2017}. Lastly, we derive a Chamber Ansatz formula for double Bott-Samelson varieties via face alternating products from dimer theory. In general, plabic fences are not reduced plabic graphs, yet we are able to extend and apply the standard tools such as local moves on plabic graphs, trips, and minimal matchings to them.
\end{abstract}

\setcounter{tocdepth}{1}




\color{black}

\section{Introduction}
Plabic graphs are a combinatorial device first introduced by Postnikov~\cite{postnikov_2006} to study total positivity in the Grassmannian. Postnikov defined a \textbf{boundary measurement map}, using dimer partition functions on \textbf{reduced} plabic graphs to parameterize points inside positroid cells. Later, Muller and Speyer~\cite{Muller_Speyer_2017} associated cluster variables of the positroid cell to the faces of the corresponding reduced plabic graph with monomials in edge weights, constructed by \textbf{minimal matchings}. If the plabic graphs are not reduced, Postnikov's parameterization fails to be injective, and the notion of \textbf{downstream wedges} in Muller--Speyer fails to be well-defined. 

On the other hand, Fomin et al.~\cite{Morsifications} introduced a special class of plabic graphs called \textbf{plabic fences} to study positive braids by planar graph combinatorics, which was further used by Galashin--Lam~\cite{galashin_lam_2024} to study plabic links. Plabic fences can be non-reduced plabic graphs, and while the usual plabic theory of move-equivalences and alternating strand diagrams play important roles in the problem of their interest, there has been little study as to what \textbf{edge weights} or \textbf{matchings} on plabic fences should correspond. 

Double Bott-Samelson varieties were first introduced in \cite{Elek_lu}. Given two positive braid words $\beta_+,\beta_-,$ the \textbf{decorated double Bott-Samelson variety} $Conf(\beta_+,\beta_-)$ is a configuration space of flags with conditions dictated by $\beta_+$ and $\beta_-.$ When both $\beta_+$ and $\beta_-$ are lifts of reduced words in $S_n,$ these recover the \textbf{double Bruhat cells} of \cite{DoubleBruhatCells}.Their coordinate ring has a cluster structure described by Shen and Weng \cite{Shen_Weng_2021}. These varieties are a type of braid variety and have connections to Legendrian knot theory (see \cite{Braid_variety_cluster,grid_plabic,augmentations_clusters} and references therein). Their relationship to plabic fences has been explored, but it has not included weights and matchings (see, e.g. \cite{grid_plabic,filling_seed}). 

In this paper, we generalize the minimal matching construction from Muller--Speyer to the setting of plabic fences, which may not be reduced. Our main result is Theorem~\ref{thm:matching_equals_minor}, previewed below.

\begin{thm*}
Given a double word $\beta$ which is an interlacing of $\beta_+$ and $\beta_-$, we can associate a plabic fence $G_\beta$ (see Section~\ref{sec: plabic_graphs}). There is a choice of edge weights for $G_\beta$ that parameterize points in a torus in $\text{Conf}(\beta_+,\beta_-)$ (see Section~\ref{sec: DBS}). Moreover, with respect to this choice of edge weights, our generalized version of ~\cite{Muller_Speyer_2017}'s minimal matchings (see Section~\ref{sec: gen_matching}) compute the generalized minors associated to these points (see Figure~\ref{fig:generalized_downstream_ex} for an illustration of generalized downstream regions).
\end{thm*}

\begin{figure}[h!]
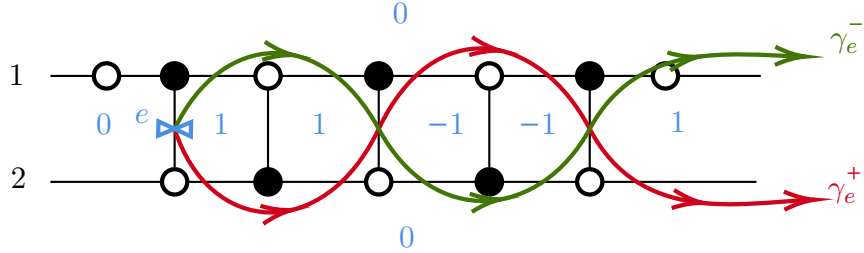

    \centering
    \includestandalone[mode=tex, width=0.7\textwidth]{figures/generalized_downstream_ex}
    \caption{An example of generalized downstream regions with respect to the edge $e$; in the reduced case, downstream faces are assigned the weight $1$, and all other faces have weight $0$; in the non-reduced case, faces can have negative integer weights.}
    \label{fig:generalized_downstream_ex}
    \end{figure}

This provides a graph-theoretical way of computing cluster variables in the coordinate ring of $Conf(\beta_+,\beta_-)$ which generalizes a version of \cite[Theorem 4.11]{DoubleBruhatCells}.

In Section~\ref{sec: cluster_algebra}, we provide the necessary background in Cluster Algebras. In Section~\ref{sec: plabic_graphs}, we provide the background on the plabic theory, including the standard minimal matching construction from \cite{Muller_Speyer_2017}. We also explain some analogous preliminaries for plabic fences. In Section~\ref{sec: DBS}, we give the background on double Bott-Samelson varieties, including the parametrizations from \cite{Shen_Weng_2021} and how they can be adapted into the plabic fence framework. Finally, in Section~\ref{sec: gen_matching} we construct a generalization of Muller-Speyer's minimal matching to plabic fences with a canonical weighting and prove Theorem~\ref{thm:matching_equals_minor}. In Section~\ref{sec: chamber_ansatz}, we derive an inverse formula to express the edge weights in terms of the generalized minors, in the flavor of ~\cite{galashin_lam_2024}.

\section*{Acknowledgements}
 The authors would like to thank David Speyer for his helpful discussions and editing suggestions. The authors also thank Amanda Schwartz for her comments on a previous version of this paper. Y.L. would also like to thank Greg Muller for suggesting possible constructions of generalized minimal matchings on non-reduced graphs. Y.L. was supported by NSF grant DMS-2246570. J.P.C. was partially supported by NSF grant DMS-2348501.
\section{Cluster Algebras}\label{sec: cluster_algebra}

We provide a brief introduction to cluster algebras limited to the scope of this paper. Details can be found in \cite{Fomin_zelevinzky,lam_cohomology_2022}. We start with the data of a \textbf{quiver} $Q$: a directed graph with vertices indexed by some finite set $I$ where multiple edges between vertices are allowed, but not oriented two-cycles or loops. We partition $I=I^{uf}\sqcup I^f$. The vertices indexed by $I^{uf}$ are called mutable, the ones indexed by $I^f$ are frozen. 

\begin{defn}
    Given a quiver $Q$ and a vertex $k\in I^{uf},$ the \textbf{mutation} of $Q$ in the direction $k$ is a new quiver $\mu_k(Q)$ obtained from $Q$ through the following procedure:
    \begin{enumerate}
        \item For each directed pair $i\to k\to j,$ add an edge $i\to j;$
        \item Reverse directions for all edges incident to $k;$
        \item Remove any oriented $2-$cycles that may have appeared through this process.
    \end{enumerate}
\end{defn}

Let $\mathcal{F}$ be a field containing $\mathbb{C}.$ A seed $\Sigma$ is a pair $(Q,\mathbf{x})$ consisting of a quiver $Q$ together with a set of algebraically independent variables in $\mathcal{F}$ also indexed by $I$. These are called cluster variables and are defined to be mutable or frozen matching the vertex with the same index.

\begin{defn}
    Given a seed $\Sigma=(Q,\mathbf{x})$ and $k\in I^{uf},$ we define the mutated seed $\mu_k(\Sigma)=(\mu_k(Q),\mu_k(\mathbf{x})),$ where $\mu_k(x_i)=x_i$ for any $i\neq k\in I$ and 
    \[\mu_k(x_k)\cdot x_k = \prod_{i\to k}x_i + \prod_{k\to i}x_i.\]

    The \textbf{cluster algebra} $\mathcal{A}(\Sigma)$ is defined to be the subalgebra of $\mathcal{F}$ generated by all cluster variables appearing in any seed obtainable from $\Sigma$ by mutation, together with $\{x_i^{-1}\}_{i\in I^f}.$
\end{defn}

For any two seeds obtainable from each other via mutations, the cluster algebras they define are isomorphic. An affine algebraic variety is called a cluster variety if it is $Spec(\mathcal{A}(\Sigma))$ for some seed $\Sigma.$ Decorated double Bott-Samelson varieties are cluster varieties \cite{Shen_Weng_2021}.

\section{Plabic Graphs}\label{sec: plabic_graphs}
We first recall the general theory of plabic graphs, which will serve as a background for understanding our construction on plabic fences later.


\begin{defn}
    A \textbf{plabic graph} $G$ is a planar bipartite graph embedded in a disk, whose internal vertices are colored black and white, and each boundary vertex is adjacent to exactly one white internal vertex. Moreover, the boundary vertices are labeled by $1,2,3,...,n$ in clockwise order. See Figure \ref{fig:plabic_graph_and_trips} for examples. We denote the set of internal vertices as $V(G)$, the set of edges of $G$ as $E(G)$, and the set of faces of $G$ as $F(G)$.
\end{defn}

\begin{figure}[h!]
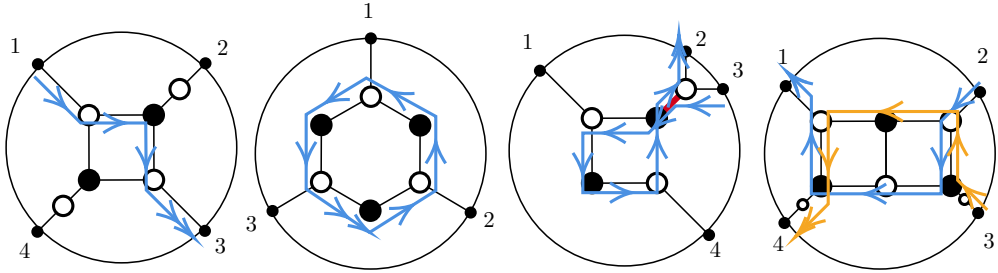

    \centering
    \includestandalone[mode=tex, width=0.8\textwidth]{figures/trips_on_plabic}
    \caption{The leftmost image is a reduced plabic graph with $n = 4, k=2$; the trip starting from boundary $1$ is shown in blue. The rest of the graphs are non-reduced: they have a closed loop, an essential self-intersection (at the internal edge colored in red), and a bad double crossing (at the leftmost and rightmost vertical edges), respectively.}
    \label{fig:plabic_graph_and_trips}
    \end{figure}
    
Given a plabic graph $G$, one can assign weights to its edges, which gives a function $w: E(G)\rightarrow \mathbb{C}^*$. We say two weights $w,w'$ are \textbf{gauge equivalent}, if there exists a function $\lambda: V(G)\rightarrow \mathbb{C}^*$, such that for any edge $e = (uv)\in E(G)$, we have $w'(e) = \lambda(u)\lambda(v)w(e)$.

\begin{defn}
    A \textbf{partial matching} $M$ of the plabic graph $G$ is a subset of edges, where each internal vertex belongs to exactly one such edge, and each boundary vertex belongs to at most one such edge. Let $k$ be the number of white internal vertices minus the number of black internal vertices. Then, the set of boundary vertices covered by edges in $M$ is a $k$-element subset of $[n]=\{1,\cdots,n\}$, denoted by $\partial M$.
\end{defn}

Given a partial matching $M$ and an edge weight function $w$, define the weight of the partial matching to be $w(M) := \prod_{e \in M} w(e)$.
The $w(M)$ are monomial functions on $(\mathbb{C}^{\ast})^E/(\mathbb{C}^{\ast})^{V-1}$, where $E$ is the size of $E(G)$, and $V$ is the size of $V(G)$, since any gauge transformation changes the $w(M)$ simultaneously by a global constant.
For reduced graphs, \cite{Muller_Speyer_2017} defines a family of \textbf{minimal matchings} such that $\{ w(M) : M \ \text{minimal} \}$ are coordinates  on $(\mathbb{C}^{\ast})^E/(\mathbb{C}^{\ast})^{V-1}$.

Given a plabic graph $G$, one defines a \textbf{boundary measurement map} $\mathbb{D}$ from the torus $(\mathbb{C}^{\ast})^E/(\mathbb{C}^{\ast})^{V-1}$ to the Grassmannian $G(k,n)$, by summing over all partial matchings $M$ with $\partial M = I$ for each $k$-element subset  $I$ \cite{dimers,postnikov_2006,Talaska}. 
If $G$ is \textbf{reduced}, this map is injective.

One also considers the following moves on plabic graphs with weights. The following two moves, along with gauge transformations, preserve the weights of the partial matchings $w(M)$, up to a global constant. 
\begin{itemize}
    \item ($d_1$) when we have a square face, whose vertices have degree at least $3$, and $ac+bd \neq 0$, we can perform an \textbf{urban renewal move} (see Figure~\ref{fig:moves_for_plabic_graphs} below left)
    \item ($d_2$) when we have a vertex of degree $2$ whose neighbors are distinct vertices, we can \textbf{contract} the edges at this degree two vertex (see Figure~\ref{fig:moves_for_plabic_graphs} below right),
\end{itemize}
and their inverse moves. 
\begin{figure}[h!]
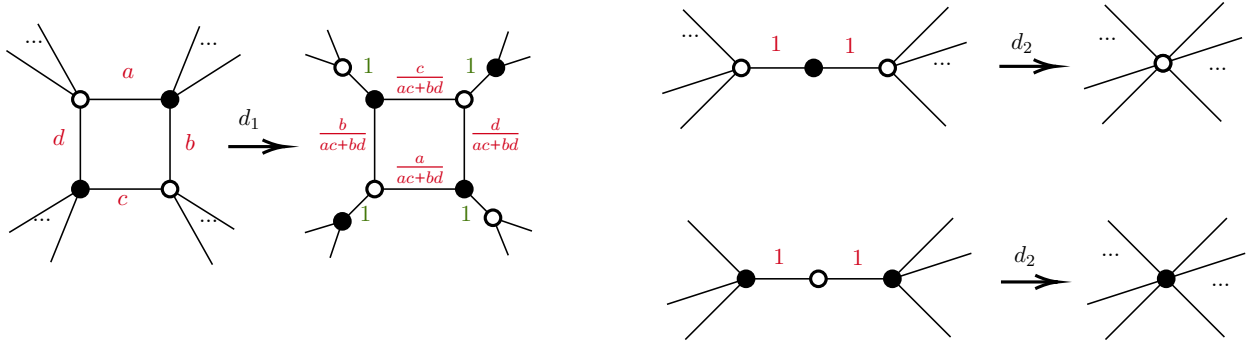

    \centering
    \includestandalone[mode=tex, width=1\textwidth]{figures/plabic_moves}
    \caption{Moves for weighted plabic graphs.}
    \label{fig:moves_for_plabic_graphs}
    \end{figure}

We now introduce another useful statistic one tracks on plabic graphs called \textbf{trips}. In \cite{postnikov_2006}, trips are used to find a decorated permutation on plabic graphs, which tells the positroid cell that the boundary measurement maps to. In addition, trips can be used to detect if a plabic graph is reduced.

\begin{defn}
    A \textbf{trip} is an oriented walk along edges of $G$ turning maximally left at white vertices and maximally right on black vertices. Such oriented walks can either be a closed loop, or can extend both ways to the boundary. 
\end{defn}

\cite[Theorem 13.2]{postnikov_2006} shows that a plabic graph $G$ without internal vertices of degree 1 is \textbf{reduced} if and only if:
\begin{enumerate}
    \item[(i)] There are no closed loops;
    \item[(ii)] (\textbf{No essential self-intersections}) No trip traverses the same edge more than once;
    \item[(iii)] (\textbf{No bad double-crossing}) If two trips both traverse edges $e_1, e_2$, then it must be that one trip visits $e_1$ before $e_2$, and the other trip visits $e_2$ before $e_1$.
\end{enumerate}

\begin{figure}[h!]
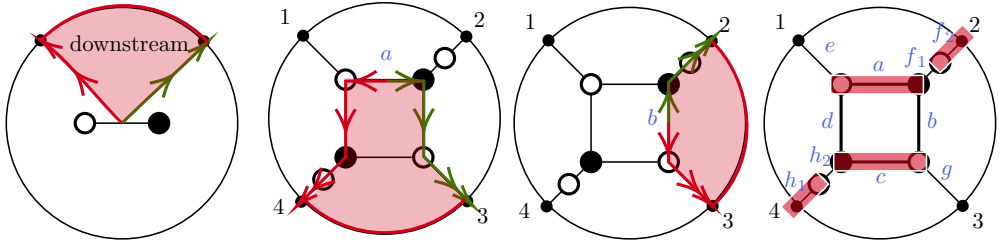

    \centering
    \includestandalone[mode=tex, width=0.8\textwidth]{figures/downstream_wedges}
    \label{fig:minimal_matchings}
    \caption{The leftmost graph sketches the downstream of an edge in reduced plabic graph, where the trip to the white endpoint is in red, and the trip to the black endpoint is in green. The 2nd and 3rd graph sketches the downstream region of the edge $a$ and $b$, respectively. Note that the square face in the middle is in the downstream wedge of $a$ but not in that of $b$. The rightmost graph is the minimal matching associated to the square face, $M_0 = \{a,c,f_2,h_1\}$, and indeed $a\in M_0, b\notin M_0$.}
    \end{figure}

For a reduced plabic graph, trips can also be used to find minimal matchings. These are partial matchings that are the minimal element in the \textbf{up-flip lattice} of all partial matchings with the same boundary condition.  The monomials $\{ w(M) : M \ \text{minimal} \}$ are twisted \cite{marscoTwist, Muller_Speyer_2017} Pl\"ucker coordinates of $\mathbb{D}(w)$.
Twisted Pl\"ucker coordinates are cluster variables, forming a cluster~\cite{scott_2003, galashin_lam_2024, demazure_weave_plabic}.

Given an edge $e$, there are two trips starting from its middle point, one traveling towards the white endpoint of $e$, and the other traveling towards the black endpoint of $e$. Since $G$ is reduced, these two trips will not meet again at a different edge than $e$, otherwise there is a bad double crossing, and they will travel to the boundary. Thus, the union of these two trips divides the disk into two regions, and the region on the left side of the trip traveling towards the white endpoint is defined to be the \textbf{downstream wedge} of $e$. Muller--Speyer showed that an edge $e$ is in the minimal matching associated to the face $f$ if and only if the face $f$ is in the \textbf{downstream wedge} of the edge $e$.

Notice that this construction relies on the assumption that $G$ is reduced. For example, if $G$ has a bad double crossing, then the two trips starting from the first intersection of the double crossing will cross again before reaching the boundary, hence dividing the disk into more than two regions. In Section 5, we will define generalized minimal matchings for plabic fences, an important family of non-reduced plabic graphs related to Double Bott-Samelson varieties. 
We will show that $\{ w(M) : M \ \text{generalized minimal} \}$ form coordinates on $(\mathbb{C}^{\ast})^E/(\mathbb{C}^{\ast})^{V-1}$ and compute cluster variables in a cluster~\cite{Shen_Weng_2021}.
In forthcoming work, the second author will define generalized minimal matchings for any plabic graph whose trips are not loops and show that they are cluster variables.

We now introduce the analogous theory of plabic fences.
\begin{defn}\label{def: double_word}
    A \textbf{double word} $\beta$ is a sequence of nonzero integers $\beta = (i_1,i_2,...,i_N)$ with $|i_k|\in [r] = \{1,2,...,r\}, \forall k=1,...,N$. Its \textbf{top-word} $\beta^+$ is the subsequence of the positive entries in $\beta$, and its \textbf{bottom-word} is the subsequence of the negative entries in $\beta$. 
\end{defn}
\begin{defn}\label{def: fence_quiver}
    Given a double word $\beta = (i_1,i_2,...,i_N)\in ([r]\coprod(-[r]))^N$, define its \textbf{associated plabic fence} $G_\beta$ as a planar bipartite graph embedded in a disk, whose $E(G_\beta)$ contains exactly:
    \begin{enumerate}
        \item[(i)] horizontal edges lying on $(r+1)$ horizontal lines extending to the boundary of the disk, indexed by $1,2,...,r+1$ from top to bottom, where boundary vertices are all black.
        \item[(ii)] a sequence of vertical column edges $e_j, j=1,...,N$, ordered from left to right, such that if $i_j\in [r]$, then $e_j$ is an edge connecting a white vertex on the $i_j$-th row to a black vertex on the $(i_{j}+1)$-th row, and when $i_j\in (-[r])$, then $e_j$ is an edge connecting a black vertex on the $-i_j$-th row to a white vertex on the $(-i_{j}+1)$-th row. 
    \end{enumerate}
    If a face lies in between the $i$-th and $(i+1)$-th lines, we say it is \textbf{in the $i$-th row}, for $i\in [r]$.
\end{defn}

\begin{exmp}\label{exem: weighted_fence}
    See Figure~\ref{fig:weighted_fence} for an example of a plabic fence where $r = 2$, and $\beta = (-1,1,2,-2,-1,2)$. Edges without labels always have weights $1$, and the letters on edges are placeholders that take values in $\mathbb{C}^*$.
    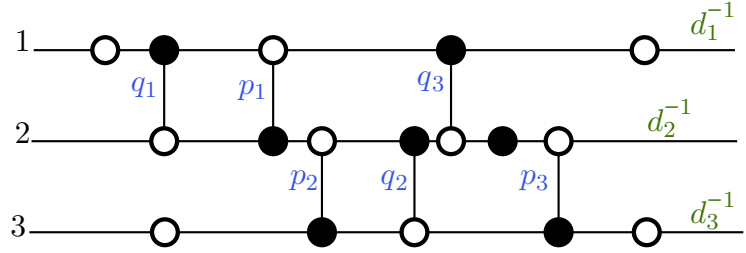
\begin{figure}[h!]
    \centering
    \tikzset{every picture/.style={line width=0.75pt}} 

\begin{tikzpicture}[x=0.75pt,y=0.75pt,yscale=-1,xscale=1]

\draw    (22.34,88) -- (83,88) -- (133.66,88) -- (216.34,88) -- (348.34,88) ;
\draw    (21.34,130.32) -- (83,130.32) -- (156.34,130.32) -- (199.34,130.32) -- (240.34,130.32) -- (266.34,130.32) -- (349.34,130.32) ;
\draw    (20.34,172.64) -- (156.34,172.64) -- (307.34,172.64) -- (352.22,172.64) ;
\draw  [fill={rgb, 255:red, 255; green, 255; blue, 255 }  ,fill opacity=1 ][line width=1.5]  (77.34,172.64) .. controls (77.34,169.33) and (80.02,166.64) .. (83.34,166.64) .. controls (86.65,166.64) and (89.34,169.33) .. (89.34,172.64) .. controls (89.34,175.95) and (86.65,178.64) .. (83.34,178.64) .. controls (80.02,178.64) and (77.34,175.95) .. (77.34,172.64) -- cycle ;
\draw  [fill={rgb, 255:red, 0; green, 0; blue, 0 }  ,fill opacity=1 ][line width=1.5]  (77,88) .. controls (77,84.69) and (79.69,82) .. (83,82) .. controls (86.31,82) and (89,84.69) .. (89,88) .. controls (89,91.31) and (86.31,94) .. (83,94) .. controls (79.69,94) and (77,91.31) .. (77,88) -- cycle ;
\draw    (83,88) -- (83,130.32) ;
\draw  [fill={rgb, 255:red, 255; green, 255; blue, 255 }  ,fill opacity=1 ][line width=1.5]  (77,130.32) .. controls (77,127.01) and (79.69,124.32) .. (83,124.32) .. controls (86.31,124.32) and (89,127.01) .. (89,130.32) .. controls (89,133.63) and (86.31,136.32) .. (83,136.32) .. controls (79.69,136.32) and (77,133.63) .. (77,130.32) -- cycle ;
\draw  [fill={rgb, 255:red, 0; green, 0; blue, 0 }  ,fill opacity=1 ][line width=1.5]  (127.66,130.32) .. controls (127.66,127.01) and (130.34,124.32) .. (133.66,124.32) .. controls (136.97,124.32) and (139.66,127.01) .. (139.66,130.32) .. controls (139.66,133.63) and (136.97,136.32) .. (133.66,136.32) .. controls (130.34,136.32) and (127.66,133.63) .. (127.66,130.32) -- cycle ;
\draw    (133.66,88) -- (133.66,130.32) ;
\draw  [fill={rgb, 255:red, 255; green, 255; blue, 255 }  ,fill opacity=1 ][line width=1.5]  (127.66,88) .. controls (127.66,84.69) and (130.34,82) .. (133.66,82) .. controls (136.97,82) and (139.66,84.69) .. (139.66,88) .. controls (139.66,91.31) and (136.97,94) .. (133.66,94) .. controls (130.34,94) and (127.66,91.31) .. (127.66,88) -- cycle ;
\draw  [fill={rgb, 255:red, 0; green, 0; blue, 0 }  ,fill opacity=1 ][line width=1.5]  (150.34,172.64) .. controls (150.34,169.33) and (153.02,166.64) .. (156.34,166.64) .. controls (159.65,166.64) and (162.34,169.33) .. (162.34,172.64) .. controls (162.34,175.95) and (159.65,178.64) .. (156.34,178.64) .. controls (153.02,178.64) and (150.34,175.95) .. (150.34,172.64) -- cycle ;
\draw    (156.34,130.32) -- (156.34,172.64) ;
\draw  [fill={rgb, 255:red, 255; green, 255; blue, 255 }  ,fill opacity=1 ][line width=1.5]  (150.34,130.32) .. controls (150.34,127.01) and (153.02,124.32) .. (156.34,124.32) .. controls (159.65,124.32) and (162.34,127.01) .. (162.34,130.32) .. controls (162.34,133.63) and (159.65,136.32) .. (156.34,136.32) .. controls (153.02,136.32) and (150.34,133.63) .. (150.34,130.32) -- cycle ;
\draw  [fill={rgb, 255:red, 0; green, 0; blue, 0 }  ,fill opacity=1 ][line width=1.5]  (193.34,130.32) .. controls (193.34,127.01) and (196.02,124.32) .. (199.34,124.32) .. controls (202.65,124.32) and (205.34,127.01) .. (205.34,130.32) .. controls (205.34,133.63) and (202.65,136.32) .. (199.34,136.32) .. controls (196.02,136.32) and (193.34,133.63) .. (193.34,130.32) -- cycle ;
\draw    (199.34,130.32) -- (199.34,172.64) ;
\draw  [fill={rgb, 255:red, 255; green, 255; blue, 255 }  ,fill opacity=1 ][line width=1.5]  (193.34,172.64) .. controls (193.34,169.33) and (196.02,166.64) .. (199.34,166.64) .. controls (202.65,166.64) and (205.34,169.33) .. (205.34,172.64) .. controls (205.34,175.95) and (202.65,178.64) .. (199.34,178.64) .. controls (196.02,178.64) and (193.34,175.95) .. (193.34,172.64) -- cycle ;
\draw  [fill={rgb, 255:red, 0; green, 0; blue, 0 }  ,fill opacity=1 ][line width=1.5]  (210.34,88) .. controls (210.34,84.69) and (213.02,82) .. (216.34,82) .. controls (219.65,82) and (222.34,84.69) .. (222.34,88) .. controls (222.34,91.31) and (219.65,94) .. (216.34,94) .. controls (213.02,94) and (210.34,91.31) .. (210.34,88) -- cycle ;
\draw    (216.34,88) -- (216.34,130.32) ;
\draw  [fill={rgb, 255:red, 255; green, 255; blue, 255 }  ,fill opacity=1 ][line width=1.5]  (210.34,130.32) .. controls (210.34,127.01) and (213.02,124.32) .. (216.34,124.32) .. controls (219.65,124.32) and (222.34,127.01) .. (222.34,130.32) .. controls (222.34,133.63) and (219.65,136.32) .. (216.34,136.32) .. controls (213.02,136.32) and (210.34,133.63) .. (210.34,130.32) -- cycle ;
\draw  [fill={rgb, 255:red, 0; green, 0; blue, 0 }  ,fill opacity=1 ][line width=1.5]  (260.34,172.64) .. controls (260.34,169.33) and (263.02,166.64) .. (266.34,166.64) .. controls (269.65,166.64) and (272.34,169.33) .. (272.34,172.64) .. controls (272.34,175.95) and (269.65,178.64) .. (266.34,178.64) .. controls (263.02,178.64) and (260.34,175.95) .. (260.34,172.64) -- cycle ;
\draw    (266.34,130.32) -- (266.34,172.64) ;
\draw  [fill={rgb, 255:red, 255; green, 255; blue, 255 }  ,fill opacity=1 ][line width=1.5]  (260.34,130.32) .. controls (260.34,127.01) and (263.02,124.32) .. (266.34,124.32) .. controls (269.65,124.32) and (272.34,127.01) .. (272.34,130.32) .. controls (272.34,133.63) and (269.65,136.32) .. (266.34,136.32) .. controls (263.02,136.32) and (260.34,133.63) .. (260.34,130.32) -- cycle ;
\draw  [fill={rgb, 255:red, 0; green, 0; blue, 0 }  ,fill opacity=1 ][line width=1.5]  (234.34,130.32) .. controls (234.34,127.01) and (237.02,124.32) .. (240.34,124.32) .. controls (243.65,124.32) and (246.34,127.01) .. (246.34,130.32) .. controls (246.34,133.63) and (243.65,136.32) .. (240.34,136.32) .. controls (237.02,136.32) and (234.34,133.63) .. (234.34,130.32) -- cycle ;
\draw  [fill={rgb, 255:red, 255; green, 255; blue, 255 }  ,fill opacity=1 ][line width=1.5]  (49.6,88) .. controls (49.6,84.69) and (52.29,82) .. (55.6,82) .. controls (58.91,82) and (61.6,84.69) .. (61.6,88) .. controls (61.6,91.31) and (58.91,94) .. (55.6,94) .. controls (52.29,94) and (49.6,91.31) .. (49.6,88) -- cycle ;
\draw  [fill={rgb, 255:red, 255; green, 255; blue, 255 }  ,fill opacity=1 ][line width=1.5]  (300.34,88) .. controls (300.34,84.69) and (303.02,82) .. (306.34,82) .. controls (309.65,82) and (312.34,84.69) .. (312.34,88) .. controls (312.34,91.31) and (309.65,94) .. (306.34,94) .. controls (303.02,94) and (300.34,91.31) .. (300.34,88) -- cycle ;
\draw  [fill={rgb, 255:red, 255; green, 255; blue, 255 }  ,fill opacity=1 ][line width=1.5]  (301.34,172.64) .. controls (301.34,169.33) and (304.02,166.64) .. (307.34,166.64) .. controls (310.65,166.64) and (313.34,169.33) .. (313.34,172.64) .. controls (313.34,175.95) and (310.65,178.64) .. (307.34,178.64) .. controls (304.02,178.64) and (301.34,175.95) .. (301.34,172.64) -- cycle ;

\draw (66,99.4) node [anchor=north west][inner sep=0.75pt]    {$\textcolor[rgb]{0.22,0.33,0.88}{q_{1}}$};
\draw (182,142.4) node [anchor=north west][inner sep=0.75pt]    {$\textcolor[rgb]{0.22,0.33,0.88}{q}\textcolor[rgb]{0.22,0.33,0.88}{_{2}}$};
\draw (199,96.4) node [anchor=north west][inner sep=0.75pt]    {$\textcolor[rgb]{0.22,0.33,0.88}{q}\textcolor[rgb]{0.22,0.33,0.88}{_{3}}$};
\draw (116,100.4) node [anchor=north west][inner sep=0.75pt]    {$\textcolor[rgb]{0.22,0.33,0.88}{p_{1}}$};
\draw (140,142.4) node [anchor=north west][inner sep=0.75pt]    {$\textcolor[rgb]{0.22,0.33,0.88}{p_{2}}$};
\draw (247,142.4) node [anchor=north west][inner sep=0.75pt]    {$\textcolor[rgb]{0.22,0.33,0.88}{p}\textcolor[rgb]{0.22,0.33,0.88}{_{3}}$};
\draw (326,64.4) node [anchor=north west][inner sep=0.75pt]    {$\textcolor[rgb]{0.25,0.46,0.02}{d^{-1}_{1}}$};
\draw (306,109.4) node [anchor=north west][inner sep=0.75pt]    {$\textcolor[rgb]{0.25,0.46,0.02}{d^{-1}_2}$};
\draw (326,152.4) node [anchor=north west][inner sep=0.75pt]    {$\textcolor[rgb]{0.25,0.46,0.02}{d^{-1}_3}$};
\draw (12,77.4) node [anchor=north west][inner sep=0.75pt]    {$1$};
\draw (12,120.4) node [anchor=north west][inner sep=0.75pt]    {$2$};
\draw (10,163.4) node [anchor=north west][inner sep=0.75pt]    {$3$};

\end{tikzpicture}
    \caption{A plabic fence with right canonical weights; the variables represent arbitrary values in $\mathbb{C}^*$, and all edges without variables are assigned weight $1$.}
    \label{fig:weighted_fence}
    \end{figure}
\end{exmp}

\begin{rem}\label{rem: right_canonical_weights}
    By gauge transformations, one can always normalize the edge weights of a plabic fence $G_\beta$ to have weight $1$ on all horizontal edges except for the rightmost horizontal edges on each row that is the same format as in Figure~\ref{fig:weighted_fence}. We say such a weighted plabic fence is \textbf{right canonically weighted}.
\end{rem}
We now define a set of moves on right-canonically weighted plabic fences. Note that the moves we define here can be generated by a sequence of moves on general plabic graphs. We only present the composition of moves that take a plabic fence to another plabic fence.

\begin{itemize}
    \item \textbf{Non-consecutive Column Swap:} Exchange the horizontal position of two vertical edges $e_1,e_2$ in-between columns $i,i+1$, $j,j+1$, respectively, such that $|i-j|\geq 2$.
    \item \textbf{Consecutive Column Swap:} Let $e_1,e_2$ be vertical edges in-between columns $i-1,i$, $i,i+1$, respectively, such that one is top-white-bottom-black and the other is top-black-bottom-white. Then we can exchange their horizontal positions. See Figure~\ref{fig:consecutive_col_swap}.
    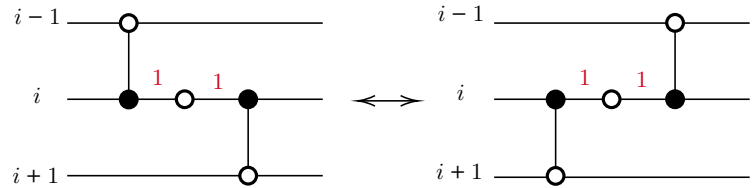
\begin{figure}[h!]
    \centering
    \tikzset{every picture/.style={line width=0.75pt}} 

\begin{tikzpicture}[x=0.75pt,y=0.75pt,yscale=-1,xscale=1]

\draw    (109,62) -- (147.74,62) -- (270,62) ;
\draw    (109,110) -- (183,110) -- (223,110) -- (270,110) ;
\draw    (109,158.15) -- (223,158.15) -- (270,158.15) ;
\draw  [fill={rgb, 255:red, 0; green, 0; blue, 0 }  ,fill opacity=1 ][line width=1.5]  (142.5,110.15) .. controls (142.5,107.25) and (144.84,104.9) .. (147.74,104.9) .. controls (150.64,104.9) and (152.99,107.25) .. (152.99,110.15) .. controls (152.99,113.05) and (150.64,115.4) .. (147.74,115.4) .. controls (144.84,115.4) and (142.5,113.05) .. (142.5,110.15) -- cycle ;
\draw    (147.74,62) -- (147.74,110.15) ;
\draw    (223,110) -- (223,158.15) ;
\draw  [fill={rgb, 255:red, 255; green, 255; blue, 255 }  ,fill opacity=1 ][line width=1.5]  (142.5,62) .. controls (142.5,59.1) and (144.84,56.75) .. (147.74,56.75) .. controls (150.64,56.75) and (152.99,59.1) .. (152.99,62) .. controls (152.99,64.9) and (150.64,67.25) .. (147.74,67.25) .. controls (144.84,67.25) and (142.5,64.9) .. (142.5,62) -- cycle ;
\draw  [fill={rgb, 255:red, 255; green, 255; blue, 255 }  ,fill opacity=1 ][line width=1.5]  (177.75,110) .. controls (177.75,107.1) and (180.1,104.75) .. (183,104.75) .. controls (185.9,104.75) and (188.25,107.1) .. (188.25,110) .. controls (188.25,112.9) and (185.9,115.25) .. (183,115.25) .. controls (180.1,115.25) and (177.75,112.9) .. (177.75,110) -- cycle ;
\draw  [fill={rgb, 255:red, 255; green, 255; blue, 255 }  ,fill opacity=1 ][line width=1.5]  (217.75,158.15) .. controls (217.75,155.25) and (220.1,152.9) .. (223,152.9) .. controls (225.9,152.9) and (228.25,155.25) .. (228.25,158.15) .. controls (228.25,161.05) and (225.9,163.4) .. (223,163.4) .. controls (220.1,163.4) and (217.75,161.05) .. (217.75,158.15) -- cycle ;
\draw  [fill={rgb, 255:red, 0; green, 0; blue, 0 }  ,fill opacity=1 ][line width=1.5]  (217.75,110) .. controls (217.75,107.1) and (220.1,104.75) .. (223,104.75) .. controls (225.9,104.75) and (228.25,107.1) .. (228.25,110) .. controls (228.25,112.9) and (225.9,115.25) .. (223,115.25) .. controls (220.1,115.25) and (217.75,112.9) .. (217.75,110) -- cycle ;
\draw    (378,62) -- (416.74,62) -- (539,62) ;
\draw    (378,110) -- (452,110) -- (492,110) -- (539,110) ;
\draw    (378,159.15) -- (492,159.15) -- (539,159.15) ;
\draw  [fill={rgb, 255:red, 0; green, 0; blue, 0 }  ,fill opacity=1 ][line width=1.5]  (411.5,110.15) .. controls (411.5,107.25) and (413.84,104.9) .. (416.74,104.9) .. controls (419.64,104.9) and (421.99,107.25) .. (421.99,110.15) .. controls (421.99,113.05) and (419.64,115.4) .. (416.74,115.4) .. controls (413.84,115.4) and (411.5,113.05) .. (411.5,110.15) -- cycle ;
\draw    (492,61.85) -- (492,110) ;
\draw    (416.74,110.15) -- (416.74,158.3) ;
\draw  [fill={rgb, 255:red, 255; green, 255; blue, 255 }  ,fill opacity=1 ][line width=1.5]  (486.75,61.85) .. controls (486.75,58.95) and (489.1,56.6) .. (492,56.6) .. controls (494.9,56.6) and (497.25,58.95) .. (497.25,61.85) .. controls (497.25,64.75) and (494.9,67.1) .. (492,67.1) .. controls (489.1,67.1) and (486.75,64.75) .. (486.75,61.85) -- cycle ;
\draw  [fill={rgb, 255:red, 255; green, 255; blue, 255 }  ,fill opacity=1 ][line width=1.5]  (446.75,110) .. controls (446.75,107.1) and (449.1,104.75) .. (452,104.75) .. controls (454.9,104.75) and (457.25,107.1) .. (457.25,110) .. controls (457.25,112.9) and (454.9,115.25) .. (452,115.25) .. controls (449.1,115.25) and (446.75,112.9) .. (446.75,110) -- cycle ;
\draw  [fill={rgb, 255:red, 255; green, 255; blue, 255 }  ,fill opacity=1 ][line width=1.5]  (411.5,158.3) .. controls (411.5,155.4) and (413.84,153.05) .. (416.74,153.05) .. controls (419.64,153.05) and (421.99,155.4) .. (421.99,158.3) .. controls (421.99,161.19) and (419.64,163.54) .. (416.74,163.54) .. controls (413.84,163.54) and (411.5,161.19) .. (411.5,158.3) -- cycle ;
\draw  [fill={rgb, 255:red, 0; green, 0; blue, 0 }  ,fill opacity=1 ][line width=1.5]  (486.75,110) .. controls (486.75,107.1) and (489.1,104.75) .. (492,104.75) .. controls (494.9,104.75) and (497.25,107.1) .. (497.25,110) .. controls (497.25,112.9) and (494.9,115.25) .. (492,115.25) .. controls (489.1,115.25) and (486.75,112.9) .. (486.75,110) -- cycle ;
\draw    (294,111) -- (328,111) ;
\draw [shift={(330,111)}, rotate = 180] [color={rgb, 255:red, 0; green, 0; blue, 0 }  ][line width=0.75]    (10.93,-3.29) .. controls (6.95,-1.4) and (3.31,-0.3) .. (0,0) .. controls (3.31,0.3) and (6.95,1.4) .. (10.93,3.29)   ;
\draw [shift={(292,111)}, rotate = 0] [color={rgb, 255:red, 0; green, 0; blue, 0 }  ][line width=0.75]    (10.93,-3.29) .. controls (6.95,-1.4) and (3.31,-0.3) .. (0,0) .. controls (3.31,0.3) and (6.95,1.4) .. (10.93,3.29)   ;

\draw (161,90.4) node [anchor=north west][inner sep=0.75pt]    {$\textcolor[rgb]{0.82,0.01,0.11}{1}$};
\draw (199,92.4) node [anchor=north west][inner sep=0.75pt]    {$\textcolor[rgb]{0.82,0.01,0.11}{1}$};
\draw (430,90.4) node [anchor=north west][inner sep=0.75pt]    {$\textcolor[rgb]{0.82,0.01,0.11}{1}$};
\draw (466,91.4) node [anchor=north west][inner sep=0.75pt]    {$\textcolor[rgb]{0.82,0.01,0.11}{1}$};
\draw (74,51.4) node [anchor=north west][inner sep=0.75pt]    {$i-1$};
\draw (73,151.4) node [anchor=north west][inner sep=0.75pt]    {$i+1$};
\draw (86,101.4) node [anchor=north west][inner sep=0.75pt]    {$i$};
\draw (341,49.4) node [anchor=north west][inner sep=0.75pt]    {$i-1$};
\draw (340,149.4) node [anchor=north west][inner sep=0.75pt]    {$i+1$};
\draw (353,99.4) node [anchor=north west][inner sep=0.75pt]    {$i$};

\end{tikzpicture}
    \caption{Contracting and expanding on a degree two vertex allows for a column swap.}
    \label{fig:consecutive_col_swap}
    \end{figure}
    \item \textbf{Same-Row Column Swap}: If there is a square face in-between the $i$th and $(i+1)$th horizontal rows, then we can swap the two vertical edges by an urban renewal followed by gauge transformations to re-normalize to a right-canonical weight. See Figure~\ref{fig:same_row_swap}. Note this also may change the weights of columns adjacent to this row.

    \begin{figure}[h!]
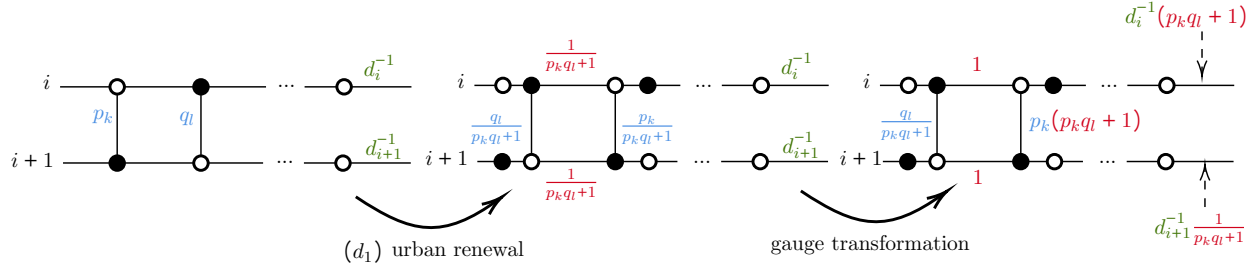

    \centering
    \includestandalone[mode=tex, width=1\textwidth]{figures/swap_col_same_row}
    \caption{Swapping opposite type vertical edges on the same row. This is possible when $p_kq_l+1\neq 0.$ Note the gauge transformations may affect the weights in the columns after the square move in the same and adjacent rows (see Figure~\ref{fig:col_swap_change_of_params}).}
    \label{fig:same_row_swap}
    \end{figure}
    \item \textbf{Braid Move}: We can change a plabic fence $G_\beta$ with $\beta = (...,i,i+1,i,...)$ (or similarly $\beta = (...,-i,-(i+1),-i,...)$) to $G_{\beta'}$ with $\beta' = (...,i+1,i,i+1,...)$ (or respectively, $\beta' = (...,-(i+1),-i,-(i+1))$) by the following local move in Figure~\ref{fig:braid_move}. 
    \begin{figure}[h!]
    \centering
    \tikzset{every picture/.style={line width=0.75pt}} 

\begin{tikzpicture}[x=0.75pt,y=0.75pt,yscale=-1,xscale=1]

\draw    (68,85) -- (94.29,85) -- (206,85) ;
\draw    (94.29,85) -- (94.29,134.67) ;
\draw    (68,134.67) -- (94.29,134.67) -- (170,134.67) -- (207,134.67) ;
\draw    (170,85) -- (170,134.67) ;
\draw    (68,184.33) -- (206,184.33) ;
\draw  [fill={rgb, 255:red, 255; green, 255; blue, 255 }  ,fill opacity=1 ][line width=1.5]  (88.36,85) .. controls (88.36,81.73) and (91.01,79.08) .. (94.29,79.08) .. controls (97.56,79.08) and (100.21,81.73) .. (100.21,85) .. controls (100.21,88.27) and (97.56,90.92) .. (94.29,90.92) .. controls (91.01,90.92) and (88.36,88.27) .. (88.36,85) -- cycle ;
\draw  [fill={rgb, 255:red, 255; green, 255; blue, 255 }  ,fill opacity=1 ][line width=1.5]  (164.08,85) .. controls (164.08,81.73) and (166.73,79.08) .. (170,79.08) .. controls (173.27,79.08) and (175.92,81.73) .. (175.92,85) .. controls (175.92,88.27) and (173.27,90.92) .. (170,90.92) .. controls (166.73,90.92) and (164.08,88.27) .. (164.08,85) -- cycle ;
\draw  [fill={rgb, 255:red, 0; green, 0; blue, 0 }  ,fill opacity=1 ][line width=1.5]  (88.36,134.67) .. controls (88.36,131.4) and (91.01,128.74) .. (94.29,128.74) .. controls (97.56,128.74) and (100.21,131.4) .. (100.21,134.67) .. controls (100.21,137.94) and (97.56,140.59) .. (94.29,140.59) .. controls (91.01,140.59) and (88.36,137.94) .. (88.36,134.67) -- cycle ;
\draw    (132.14,134.67) -- (132.14,184.33) ;
\draw  [fill={rgb, 255:red, 255; green, 255; blue, 255 }  ,fill opacity=1 ][line width=1.5]  (126.22,134.67) .. controls (126.22,131.4) and (128.87,128.74) .. (132.14,128.74) .. controls (135.41,128.74) and (138.07,131.4) .. (138.07,134.67) .. controls (138.07,137.94) and (135.41,140.59) .. (132.14,140.59) .. controls (128.87,140.59) and (126.22,137.94) .. (126.22,134.67) -- cycle ;
\draw  [fill={rgb, 255:red, 0; green, 0; blue, 0 }  ,fill opacity=1 ][line width=1.5]  (164.08,134.67) .. controls (164.08,131.4) and (166.73,128.74) .. (170,128.74) .. controls (173.27,128.74) and (175.92,131.4) .. (175.92,134.67) .. controls (175.92,137.94) and (173.27,140.59) .. (170,140.59) .. controls (166.73,140.59) and (164.08,137.94) .. (164.08,134.67) -- cycle ;
\draw  [fill={rgb, 255:red, 0; green, 0; blue, 0 }  ,fill opacity=1 ][line width=1.5]  (126.22,184.33) .. controls (126.22,181.06) and (128.87,178.41) .. (132.14,178.41) .. controls (135.41,178.41) and (138.07,181.06) .. (138.07,184.33) .. controls (138.07,187.6) and (135.41,190.26) .. (132.14,190.26) .. controls (128.87,190.26) and (126.22,187.6) .. (126.22,184.33) -- cycle ;
\draw    (300.83,85) -- (341.29,85) -- (470.83,85) ;
\draw    (341.29,134.67) -- (341.29,184.33) ;
\draw    (300.83,134.67) -- (341.29,134.67) -- (417,134.67) -- (471.83,134.67) ;
\draw    (417,134.67) -- (417,184.33) ;
\draw    (315,184.33) -- (300.83,184.33) -- (473.83,184.33) ;
\draw  [fill={rgb, 255:red, 255; green, 255; blue, 255 }  ,fill opacity=1 ][line width=1.5]  (335.36,134.67) .. controls (335.36,131.4) and (338.01,128.74) .. (341.29,128.74) .. controls (344.56,128.74) and (347.21,131.4) .. (347.21,134.67) .. controls (347.21,137.94) and (344.56,140.59) .. (341.29,140.59) .. controls (338.01,140.59) and (335.36,137.94) .. (335.36,134.67) -- cycle ;
\draw  [fill={rgb, 255:red, 255; green, 255; blue, 255 }  ,fill opacity=1 ][line width=1.5]  (411.08,134.67) .. controls (411.08,131.4) and (413.73,128.74) .. (417,128.74) .. controls (420.27,128.74) and (422.92,131.4) .. (422.92,134.67) .. controls (422.92,137.94) and (420.27,140.59) .. (417,140.59) .. controls (413.73,140.59) and (411.08,137.94) .. (411.08,134.67) -- cycle ;
\draw  [fill={rgb, 255:red, 0; green, 0; blue, 0 }  ,fill opacity=1 ][line width=1.5]  (335.36,184.33) .. controls (335.36,181.06) and (338.01,178.41) .. (341.29,178.41) .. controls (344.56,178.41) and (347.21,181.06) .. (347.21,184.33) .. controls (347.21,187.6) and (344.56,190.26) .. (341.29,190.26) .. controls (338.01,190.26) and (335.36,187.6) .. (335.36,184.33) -- cycle ;
\draw    (379.14,85) -- (379.14,134.67) ;
\draw  [fill={rgb, 255:red, 255; green, 255; blue, 255 }  ,fill opacity=1 ][line width=1.5]  (373.22,85) .. controls (373.22,81.73) and (375.87,79.08) .. (379.14,79.08) .. controls (382.41,79.08) and (385.07,81.73) .. (385.07,85) .. controls (385.07,88.27) and (382.41,90.92) .. (379.14,90.92) .. controls (375.87,90.92) and (373.22,88.27) .. (373.22,85) -- cycle ;
\draw  [fill={rgb, 255:red, 0; green, 0; blue, 0 }  ,fill opacity=1 ][line width=1.5]  (411.08,184.33) .. controls (411.08,181.06) and (413.73,178.41) .. (417,178.41) .. controls (420.27,178.41) and (422.92,181.06) .. (422.92,184.33) .. controls (422.92,187.6) and (420.27,190.26) .. (417,190.26) .. controls (413.73,190.26) and (411.08,187.6) .. (411.08,184.33) -- cycle ;
\draw  [fill={rgb, 255:red, 0; green, 0; blue, 0 }  ,fill opacity=1 ][line width=1.5]  (373.22,134.67) .. controls (373.22,131.4) and (375.87,128.74) .. (379.14,128.74) .. controls (382.41,128.74) and (385.07,131.4) .. (385.07,134.67) .. controls (385.07,137.94) and (382.41,140.59) .. (379.14,140.59) .. controls (375.87,140.59) and (373.22,137.94) .. (373.22,134.67) -- cycle ;
\draw  [fill={rgb, 255:red, 0; green, 0; blue, 0 }  ,fill opacity=1 ][line width=1.5]  (315.13,134.67) .. controls (315.13,131.4) and (317.79,128.74) .. (321.06,128.74) .. controls (324.33,128.74) and (326.98,131.4) .. (326.98,134.67) .. controls (326.98,137.94) and (324.33,140.59) .. (321.06,140.59) .. controls (317.79,140.59) and (315.13,137.94) .. (315.13,134.67) -- cycle ;
\draw  [fill={rgb, 255:red, 0; green, 0; blue, 0 }  ,fill opacity=1 ][line width=1.5]  (438.49,134.67) .. controls (438.49,131.4) and (441.14,128.74) .. (444.41,128.74) .. controls (447.68,128.74) and (450.34,131.4) .. (450.34,134.67) .. controls (450.34,137.94) and (447.68,140.59) .. (444.41,140.59) .. controls (441.14,140.59) and (438.49,137.94) .. (438.49,134.67) -- cycle ;
\draw [line width=1.5]    (216,135) -- (249,135) ;
\draw [shift={(252,135)}, rotate = 180] [color={rgb, 255:red, 0; green, 0; blue, 0 }  ][line width=1.5]    (14.21,-4.28) .. controls (9.04,-1.82) and (4.3,-0.39) .. (0,0) .. controls (4.3,0.39) and (9.04,1.82) .. (14.21,4.28)   ;
\draw  [fill={rgb, 255:red, 255; green, 255; blue, 255 }  ,fill opacity=1 ][line width=1.5]  (374.48,184.33) .. controls (374.48,181.06) and (377.13,178.41) .. (380.4,178.41) .. controls (383.68,178.41) and (386.33,181.06) .. (386.33,184.33) .. controls (386.33,187.6) and (383.68,190.26) .. (380.4,190.26) .. controls (377.13,190.26) and (374.48,187.6) .. (374.48,184.33) -- cycle ;

\draw (97,99.4) node [anchor=north west][inner sep=0.75pt]    {$\textcolor[rgb]{0.29,0.56,0.89}{p_{k}}$};
\draw (141,152.4) node [anchor=north west][inner sep=0.75pt]    {$\textcolor[rgb]{0.29,0.56,0.89}{p_{k+1}}$};
\draw (175,99.4) node [anchor=north west][inner sep=0.75pt]    {$\textcolor[rgb]{0.29,0.56,0.89}{p_{k+2}}$};
\draw (383,102.32) node [anchor=north west][inner sep=0.75pt]  [font=\footnotesize]  {$\textcolor[rgb]{0.29,0.56,0.89}{p}\textcolor[rgb]{0.29,0.56,0.89}{_{k} +p_{k+2}}$};
\draw (343,148) node [anchor=north west][inner sep=0.75pt]  [font=\footnotesize]  {$\textcolor[rgb]{0.29,0.56,0.89}{\frac{p_{k+1} p_{k+2}}{p_{k} +p_{k+2}}}$};
\draw (422,148) node [anchor=north west][inner sep=0.75pt]  [font=\footnotesize]  {$\textcolor[rgb]{0.29,0.56,0.89}{\frac{p_{k} p_{k+1}}{p_{k} +p_{k+2}}}$};
\draw (47,78.4) node [anchor=north west][inner sep=0.75pt]    {$i$};
\draw (31,126.4) node [anchor=north west][inner sep=0.75pt]    {$i+1$};
\draw (32,173.4) node [anchor=north west][inner sep=0.75pt]    {$i+2$};
\draw (282,75.4) node [anchor=north west][inner sep=0.75pt]    {$i$};
\draw (266,123.4) node [anchor=north west][inner sep=0.75pt]    {$i+1$};
\draw (267,170.4) node [anchor=north west][inner sep=0.75pt]    {$i+2$};

\end{tikzpicture}
    \caption{Braid move for white-over-black columns; the braid move for black-over-white columns is defined similarly}
    \label{fig:braid_move}
    \end{figure}
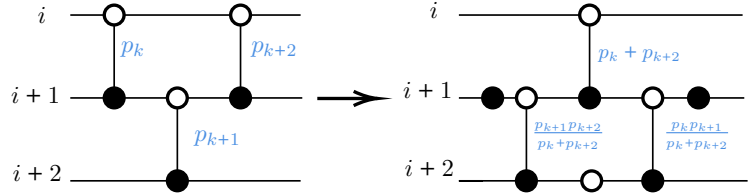
    
\end{itemize}

\begin{rem}
    One can figure out the new column weights by a combination of urban renewal and gauge transformations. See Figure~\ref{fig:braid_move_1}. 
    \begin{figure}[h!]
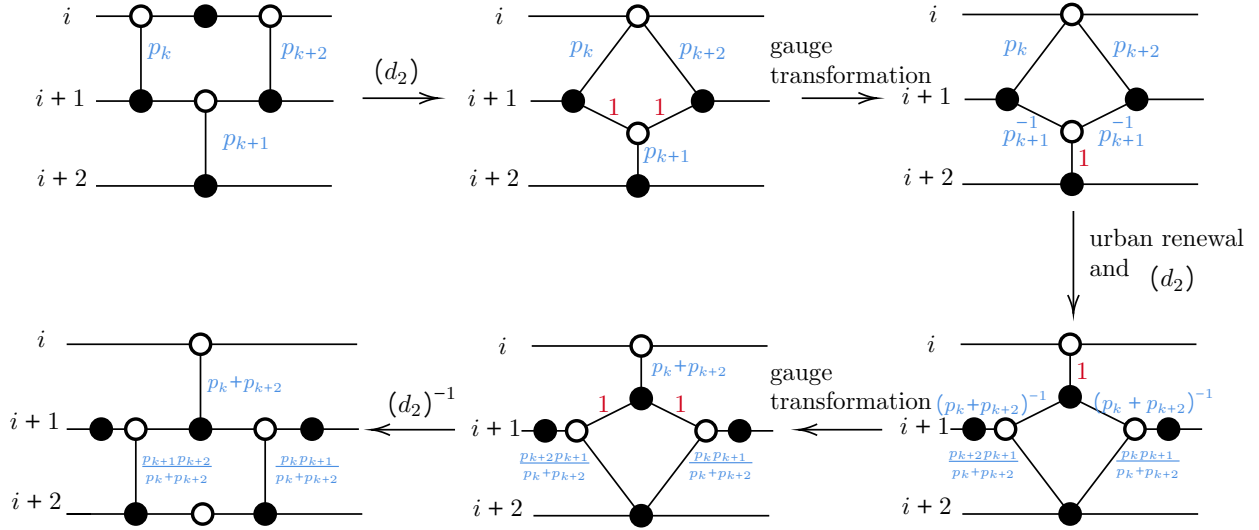

    \centering
    \includestandalone[mode=tex, width=1\textwidth]{figures/braid_move_1}
    \caption{Local transformation to generate braid move}
    \label{fig:braid_move_1}
    \end{figure}
\end{rem}

To each plabic fence $G_\beta$, we can associate a quiver $Q_\beta$ in the following way (see \cite[Definition 7.2.1]{fomin2025introductionclusteralgebraschapter}):
\begin{enumerate}
    \item Put a vertex in each unbounded face except for the top and bottom face, and mark them as \textbf{frozen}, and put a vertex in each bounded face and mark them as \textbf{mutable}.
    \item For each edge $e$ of the plabic fence that does not border the top or bottom faces, we draw an arrow connecting $e$'s two neighboring faces so that the white endpoint of $e$ is on the left side of the arrow, and the black endpoint of $e$ is on the right side of the arrow.
    \item We remove any arrows between two frozen faces, and remove any oriented two cycles. 
\end{enumerate}
\begin{rem}
    With respect to the quiver $Q_\beta$, the \textbf{Same-Row Column Swaps} and the \textbf{Braid Moves} correspond to mutations on the quiver, while the other moves preserve the quiver.    
\end{rem}

We label the vertices of the quiver $Q_\beta$ (i.e. the faces of the plabic fence $G_\beta$) in the same way as in ~\cite{Shen_Weng_2021}. We let $\displaystyle\binom{i}{a}$ denote the $a$-th face in the row in between the $i$-th and $(i+1)$-th horizontal lines, counted from left to right, starting with index $0$. 
\begin{exmp}
    See Figure~\ref{fig:fence_quiver_ex} for the associated quiver to the plabic fence in Example ~\ref{exem: weighted_fence}. Frozen vertices are square-boxes and mutable vertices are round-boxes, and the corresponding labels are in green.
    \begin{figure}[h!]
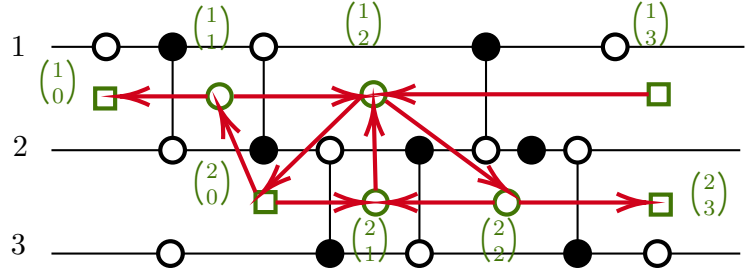

    \centering
    \includestandalone[mode=tex, width=0.6\textwidth]{figures/fence_quiver_example}
    \caption{Associated quiver $Q_\beta$ for $\beta = (-1,1,2,-2,-1,2)$.}
    \label{fig:fence_quiver_ex}
    \end{figure}
\end{exmp}


\section{Double Bott-Samelson Varieties}\label{sec: DBS}

We follow the treatment in \cite{Shen_Weng_2021}, but fix our algebraic group $G = SL_{r+1}(\mathbb{C})$ for some positive integer $r$, with corresponding Weyl group $S_{r+1}$ generated by simple transpositions $s_i.$ We consider $Br_{r+1}^+$, its positive braid monoid, and abuse notation by also denoting its generators with $s_i$. Let $U_+$ denote the subgroup of unipotent upper triangular matrices, and $U_-$ denote the subgroup of unipotent lower triangular matrices. Let $B_+, B_-$ denote the subgroups of upper, lower triangular matrices with determinant $1$, respectively. Let $T$ denote the subgroup of diagonal matrices with determinant $1$. For any $w\in S_{r+1}$, we define two lifts $\overline{w}$ and $\overline{\overline{w}}$ in $G$, where $$\overline{w}(j,i) := \begin{cases}
    (-1)^{\#\{t < i| w(t) > w(i)\}},& j = w(i)\\
    0,&\text{else}\\
\end{cases}\;\;\text{and}\;\;\;\overline{\overline{w}}(j,i):= \begin{cases}
    (-1)^{\#\{t > i| w(t) < w(i)\}},&j = w(i)\\
    0, &\text{else}
\end{cases}.$$

We consider the flag varieties $G/B_+$ and $B_-\backslash G$ and their decorated counterparts $G/U_+$ and $U_-\backslash G.$ An \textbf{undecorated flag} (resp. \textbf{undecorated opposite flag}) is a right coset $xB_+$, for some $x\in G$ (resp. a left coset $B_-y$ for some $y\in G$). A \textbf{decorated flag} (resp. \textbf{decorated opposite flag}) is a right coset $xU_+$, for some $x\in G$ (resp. a left coset $U_-y$ for some $y\in G$). When considering flags as maximal chains of subspaces inside $\mathbb{C}^{r+1},$ one can think of a decoration over a flag as a choice of volume form for each subspace in that flag. Note that we have the left actions of $G$ on flags $(g, xB_+)\mapsto (gx)B_+, (g, B_-y)\mapsto B_-(yg^{-1})$, and analogously for the decorated counterparts. In particular, we can use $yB_-$ to denote the coset $B_-y^{-1}$ (and similarly for $U_-$). Two (undecorated) flags $xB_+$ and $yB_+$ (resp. $B_-x$ and $B_-y$) are in \textbf{relative i-th position} if $x^{-1}y\in B_+\overline{s_i}B_+$ (resp. $xy^{-1}\in B_-\overline{\overline{s_i}}B_-$).

\begin{defn}
We say a decorated flag $yU_+$ is in \textbf{strong relative i-th position} to $xU_+$ if $x^{-1}y\in U_+\overline{s_i}U_+$, denoted as $xU_+\xrightarrow[]{s_i} yU_+$. It is in \textbf{weak relative i-th position} if $x^{-1}y\in B_+\overline{s_i}B_+,$ denoted $xU_+\xdashrightarrow[]{s_i} yU_+$. 

Similarly, a decorated opposite flag $U_-y$ is in \textbf{strong relative i-th position} to $U_-x$ if $xy^{-1}\in U_-\overline{\overline{s_i}}U_-$, denoted as $U_-x\xrightarrow[]{s_i} U_-y$. And in \textbf{weak relative i-th position} if $xy^{-1}\in B_-\overline{\overline{s_i}}B_-$, denoted as $U_-x\xdashrightarrow[]{s_i} U_-y.$
\end{defn}

We set up some matrix notations before we proceed. Let $\phi_i: SL_2\rightarrow SL_{r+1}$ be the inclusion that sends a $2\times 2$ matrix $M$ to a $(r+1)\times (r+1)$ block diagonal matrix, whose $[i,i+1]\times[i,i+1]$ block is $M$, and whose $[1,i-1]\times [1,i-1]$ and $[i+2,r+1]\times [i+2,r+1]$ blocks are identity matrices with corresponding sizes. For $p,q,z\in \mathbb{C}, s\in \mathbb{C}^*$ and $t= diag(t_1,t_2,\cdots,t_{r+1})\in \text{SL}_{r+1}(\mathbb{C})$, we define 
\begin{equation*}
    \begin{split}
        e_i(p):= \phi_i\left(\m{1&p\\0&1}\right)\:\:&\text{and}\:\:e_{-i}(q):= \phi_i\left(\m{1&0\\q&1}\right)\\
        B_i(z):= \phi_i\left(\m{z&-1\\1&0}\right)\:\:&\text{and}\:\:B_{-i}(z):= \phi_i\left(\m{0&1\\-1&z}\right)\\        s^{\alpha_i^\vee}=\phi_i\left(\m{s&0\\0&s^{-1}}\right)\:\:&\text{and}\:\:t^{\alpha_i} = \frac{t_i}{t_{i+1}},\;\;\;\; t^{\omega_i} = t_1t_2\cdots t_i.
    \end{split}
\end{equation*}

\begin{lem}\label{lem: relative_pos_matrix}
    Let $g\in G$.
    \begin{itemize}
        \item If $g\in U_+\overline{s_i}U_+$, we have $gU_+ = B_i(z)U_+$, for some unique $z\in \mathbb{C}$ depending only on $g$;
        \item  Similarly, if $g\in U_-\overline{\overline{s_i}}U_-$, we have $gU_- = B_{-i}(z)U_-$, for some unique $z\in \mathbb{C}$ depending only on $g$.
    \end{itemize}
\end{lem} 
\begin{proof}
    We only prove the first statement. The proof for the second one is analogous.
    
    Let $g= U_1\overline{s_i}U_2$, where $U_1, U_2\in U_+$. Since it suffices to show $g$ is equal to $B_i(z)$ up to right-multiplication by $U_+$, we can choose $g$ such that $U_2 = \text{Id}_{r+1}$. Now, since $U_1$ corresponds to a sequence of bottom-to-top row operations, we know 
   $$\begin{cases}
       g(j,k) = 0, \:\:\forall j > s_i(k)\\
       g(i,i) = U_1(i,i+1)
   \end{cases}$$
   
   On the other hand, we have $$g(s_i(k),k)=\begin{cases}
       -1,\:\:\text{if}\:\: k=i\\
       1,\:\:\text{else}
   \end{cases}$$
   Thus there are left-to-right column operations $U'\in U_+$, where we zero-out the entries to the right of each $g(s_i(k),k)$, that is, $gU' = B_i(U_1(i,i+1))$. Lastly, note that this choice $z:= U_1(i,i+1)$ is unique, since any right action of $U_+$ on $g$ does not change the entry $g(i,i)$, since $g(i,k) = 0,\forall k < i$.
\end{proof}

Being in relative position $s_i$ means the two flags agree for every subspace except the $ith$ one. Given a reduced word $w=s_{i_1}\cdots s_{i_k}\in S_{r+1},$ we can extend this definition and talk about two flags being in relative position $w$. A particularly important case is when the two flags are in relative position $w_0,$ where this is some reduced word for the longest permutation in $S_{r+1}.$ This happens when the flags are of the form $xU_+, U_-y$ with $yx\in B_-B_+$, and similarly for the undecorated versions. We then say the flags are in \textbf{general position}.

Given decorated flags $xU_+, U_-y,$ we call the  upper left $i\times i$ minor of $yx$ their $i$th \textbf{generalized minor}, denoted $\Delta_i(yx)$. Two flags are in general position if and only if all generalized minors don't vanish.

\begin{lem}\label{lem:normalization_on_diagonal}
    Suppose $A^0\in G/U_+,A_0\in U_-\backslash G$ are in general position. Then, there exists $g\in G, t_0,t'_0\in T$, such that $A^0 = gt_0U_+, A_0 =g(t_0')^{-1}U_-$, and $t'_0t_0$ is uniquely determined by $A^0, A_0$.
\end{lem}
\begin{proof}
    Let $A^0 = xU_+, A_0 = yU_- = U_-y^{-1}$. Then, $\Delta_i(y^{-1}x)\neq 0,\forall i\in [r+1]$. Let $$t:= \text{diag}\{\Delta_1(y^{-1}x),\frac{\Delta_2(y^{-1}x)}{\Delta_1(y^{-1}x)},...,\frac{\Delta_{r+1}(y^{-1}x)}{\Delta_{r}(y^{-1}x)}\}$$ Then, there exists $U_1\in U_-, U_2\in U_+$, such that $t = U_1(y^{-1}x)U_2$. Let $g = xU_2$, then $U_1y^{-1} = tg^{-1}$. Thus, $A^0 = xU_+ = gU_+, A_0 = yU_- = (yU^{-1}_1)U_- = gt^{-1}U_-$. 
\end{proof}


\begin{defn}
Let $\underline{b_+} = (s_{i_1},s_{i_2},...,s_{i_m})$ and $\underline{b_-} = (s_{j_1},s_{j_2},...,s_{j_n})$ be two braid words in $Br_{r+1}^+$. The \textbf{decorated double Bott-Samelson variety} $\text{Conf}({\underline{b_+}},{\underline{b_-}})$, is the space of a sequence of decorated flags $\{A^i\}$ and a sequence of opposite decorated flags $\{A_i\}$ satisfying horizontal relative positions and vertical general positions indicated by Figure \ref{fig:dbs_config}, modulo left action of $G$ simultaneously on all flags.

\begin{figure}[h!]
    \centering
    \tikzset{every picture/.style={line width=0.75pt}} 

\begin{tikzpicture}[x=0.75pt,y=0.75pt,yscale=-1,xscale=1]

\draw    (54,110) -- (117,110) ;
\draw [shift={(119,110)}, rotate = 180] [color={rgb, 255:red, 0; green, 0; blue, 0 }  ][line width=0.75]    (10.93,-3.29) .. controls (6.95,-1.4) and (3.31,-0.3) .. (0,0) .. controls (3.31,0.3) and (6.95,1.4) .. (10.93,3.29)   ;
\draw    (150,110) -- (213,110) ;
\draw [shift={(215,110)}, rotate = 180] [color={rgb, 255:red, 0; green, 0; blue, 0 }  ][line width=0.75]    (10.93,-3.29) .. controls (6.95,-1.4) and (3.31,-0.3) .. (0,0) .. controls (3.31,0.3) and (6.95,1.4) .. (10.93,3.29)   ;
\draw    (244,111) -- (307,111) ;
\draw [shift={(309,111)}, rotate = 180] [color={rgb, 255:red, 0; green, 0; blue, 0 }  ][line width=0.75]    (10.93,-3.29) .. controls (6.95,-1.4) and (3.31,-0.3) .. (0,0) .. controls (3.31,0.3) and (6.95,1.4) .. (10.93,3.29)   ;
\draw    (344,112) -- (407,112) ;
\draw [shift={(409,112)}, rotate = 180] [color={rgb, 255:red, 0; green, 0; blue, 0 }  ][line width=0.75]    (10.93,-3.29) .. controls (6.95,-1.4) and (3.31,-0.3) .. (0,0) .. controls (3.31,0.3) and (6.95,1.4) .. (10.93,3.29)   ;
\draw    (54,168) -- (117,168) ;
\draw [shift={(119,168)}, rotate = 180] [color={rgb, 255:red, 0; green, 0; blue, 0 }  ][line width=0.75]    (10.93,-3.29) .. controls (6.95,-1.4) and (3.31,-0.3) .. (0,0) .. controls (3.31,0.3) and (6.95,1.4) .. (10.93,3.29)   ;
\draw    (150,168) -- (213,168) ;
\draw [shift={(215,168)}, rotate = 180] [color={rgb, 255:red, 0; green, 0; blue, 0 }  ][line width=0.75]    (10.93,-3.29) .. controls (6.95,-1.4) and (3.31,-0.3) .. (0,0) .. controls (3.31,0.3) and (6.95,1.4) .. (10.93,3.29)   ;
\draw    (244,169) -- (307,169) ;
\draw [shift={(309,169)}, rotate = 180] [color={rgb, 255:red, 0; green, 0; blue, 0 }  ][line width=0.75]    (10.93,-3.29) .. controls (6.95,-1.4) and (3.31,-0.3) .. (0,0) .. controls (3.31,0.3) and (6.95,1.4) .. (10.93,3.29)   ;
\draw    (344,170) -- (407,170) ;
\draw [shift={(409,170)}, rotate = 180] [color={rgb, 255:red, 0; green, 0; blue, 0 }  ][line width=0.75]    (10.93,-3.29) .. controls (6.95,-1.4) and (3.31,-0.3) .. (0,0) .. controls (3.31,0.3) and (6.95,1.4) .. (10.93,3.29)   ;
\draw    (36.5,118) -- (36.5,154.77) ;
\draw    (419.5,121) -- (419.5,157.77) ;

\draw (29,98.4) node [anchor=north west][inner sep=0.75pt]    {$A^{0}$};
\draw (75,86.4) node [anchor=north west][inner sep=0.75pt]    {$s_{i_{1}}$};
\draw (125,99.4) node [anchor=north west][inner sep=0.75pt]    {$A^{1}$};
\draw (171,86.4) node [anchor=north west][inner sep=0.75pt]    {$s_{i_{2}}$};
\draw (219,99.4) node [anchor=north west][inner sep=0.75pt]    {$A^{2}$};
\draw (265,87.4) node [anchor=north west][inner sep=0.75pt]    {$s_{i_{2}}$};
\draw (318,110) node [anchor=north west][inner sep=0.75pt]    {$...$};
\draw (365,88.4) node [anchor=north west][inner sep=0.75pt]    {$s_{i_{m}}$};
\draw (412,101.4) node [anchor=north west][inner sep=0.75pt]    {$A^{m}$};
\draw (29,156.4) node [anchor=north west][inner sep=0.75pt]    {$A_{0}$};
\draw (75,144.4) node [anchor=north west][inner sep=0.75pt]    {$s_{j_{1}}$};
\draw (125,157.4) node [anchor=north west][inner sep=0.75pt]    {$A_{1}$};
\draw (171,144.4) node [anchor=north west][inner sep=0.75pt]    {$s_{j_{2}}$};
\draw (219,157.4) node [anchor=north west][inner sep=0.75pt]    {$A_{2}$};
\draw (265,145.4) node [anchor=north west][inner sep=0.75pt]    {$s_{j_{2}}$};
\draw (318,168) node [anchor=north west][inner sep=0.75pt]    {$...$};
\draw (365,146.4) node [anchor=north west][inner sep=0.75pt]    {$s_{j_{n}}$};
\draw (412,159.4) node [anchor=north west][inner sep=0.75pt]    {$A_{n}$};

\end{tikzpicture}
    \caption{Configuration of relative positions in $Conf({\underline{b_+}},{\underline{b_-}})$.}
    \label{fig:dbs_config}
\end{figure}
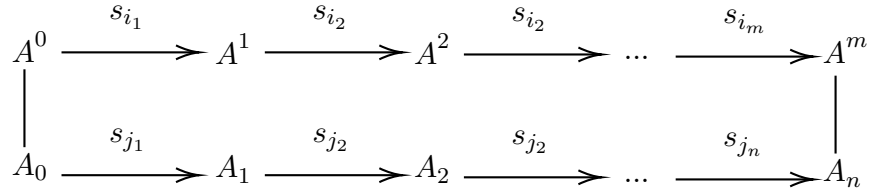
\end{defn}

Under the left action, we can choose a standard representative of a given $G-$orbit in this space as the unique one such that $A_0=U_-$ and $A^0=tU_+$ for some $t\in T.$ This is possible since $A^0$ and $A_0$ are in general position.

Given two words $\underline{b^1_+}, \underline{b^2_+}$ for a positive braid $\beta_+$ (respectively, two words $\underline{b^1_-}$, $\underline{b^2_-}$ for the braid $\beta_-$), they are related by a sequence of braid moves, and \cite[Theorem 2.18]{Shen_Weng_2021} shows that the corresponding double Bott-Samelson varieties $\text{Conf}({\underline{b^1_+}},{\underline{b^1_-}})\sim \text{Conf}({\underline{b^2_+}},{\underline{b^2_-}})$ are related by a canonical isomorphism. Using identifications under these isomorphisms, we also refer to $\text{Conf}({\beta_+},{\beta_-})$ as a decorated double Bott-Samelson variety.

\subsection{Algebraic Tori in Double Bott-Samelson varieties}
To construct the cluster structure on $\mathbb{C}[\text{Conf}(\beta_+,\beta_-)],$ Shen--Weng use triangulations. A \textbf{triangulation} of the configuration diagram of a decorated double Bott-Samelson variety is a maximal collection of noncrossing diagonals, where each diagonal connects a top row flag and a bottom row flag, as can be seen in Figure~\ref{fig:triang_example}. We call a triangle in a triangulation \textbf{upward-facing} (resp. \textbf{downward-facing}) if the arrow indicating the strong relative position appears on the top (resp. bottom) row of the diagram.

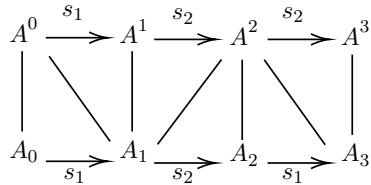
\begin{figure}[h!]
    \centering
    \tikzset{every picture/.style={line width=0.75pt}} 

\begin{tikzpicture}[x=0.75pt,y=0.75pt,yscale=-1,xscale=1]

\draw    (80,83) -- (114.53,83) ;
\draw [shift={(116.53,83)}, rotate = 180] [color={rgb, 255:red, 0; green, 0; blue, 0 }  ][line width=0.75]    (10.93,-3.29) .. controls (6.95,-1.4) and (3.31,-0.3) .. (0,0) .. controls (3.31,0.3) and (6.95,1.4) .. (10.93,3.29)   ;
\draw    (144,84) -- (178.53,84) ;
\draw [shift={(180.53,84)}, rotate = 180] [color={rgb, 255:red, 0; green, 0; blue, 0 }  ][line width=0.75]    (10.93,-3.29) .. controls (6.95,-1.4) and (3.31,-0.3) .. (0,0) .. controls (3.31,0.3) and (6.95,1.4) .. (10.93,3.29)   ;
\draw    (211,84) -- (245.53,84) ;
\draw [shift={(247.53,84)}, rotate = 180] [color={rgb, 255:red, 0; green, 0; blue, 0 }  ][line width=0.75]    (10.93,-3.29) .. controls (6.95,-1.4) and (3.31,-0.3) .. (0,0) .. controls (3.31,0.3) and (6.95,1.4) .. (10.93,3.29)   ;
\draw    (80,156) -- (114.53,156) ;
\draw [shift={(116.53,156)}, rotate = 180] [color={rgb, 255:red, 0; green, 0; blue, 0 }  ][line width=0.75]    (10.93,-3.29) .. controls (6.95,-1.4) and (3.31,-0.3) .. (0,0) .. controls (3.31,0.3) and (6.95,1.4) .. (10.93,3.29)   ;
\draw    (144,157) -- (178.53,157) ;
\draw [shift={(180.53,157)}, rotate = 180] [color={rgb, 255:red, 0; green, 0; blue, 0 }  ][line width=0.75]    (10.93,-3.29) .. controls (6.95,-1.4) and (3.31,-0.3) .. (0,0) .. controls (3.31,0.3) and (6.95,1.4) .. (10.93,3.29)   ;
\draw    (211,157) -- (245.53,157) ;
\draw [shift={(247.53,157)}, rotate = 180] [color={rgb, 255:red, 0; green, 0; blue, 0 }  ][line width=0.75]    (10.93,-3.29) .. controls (6.95,-1.4) and (3.31,-0.3) .. (0,0) .. controls (3.31,0.3) and (6.95,1.4) .. (10.93,3.29)   ;
\draw    (66,91) -- (66,138.43) ;
\draw    (261,93) -- (261,140.43) ;
\draw    (82,94) -- (118,144.26) ;
\draw    (146,146) -- (183.76,95.89) ;
\draw    (131,91) -- (131,138.43) ;
\draw    (196,96) -- (196,143.43) ;
\draw    (209,96) -- (245,146.26) ;

\draw (57,69.4) node [anchor=north west][inner sep=0.75pt]    {$A^{0}$};
\draw (122,70.4) node [anchor=north west][inner sep=0.75pt]    {$A^{1}$};
\draw (187,71.4) node [anchor=north west][inner sep=0.75pt]    {$A^{2}$};
\draw (253,71.4) node [anchor=north west][inner sep=0.75pt]    {$A^{3}$};
\draw (88,62.4) node [anchor=north west][inner sep=0.75pt]    {$s_{1}$};
\draw (153,64.4) node [anchor=north west][inner sep=0.75pt]    {$s_{2}$};
\draw (218,63.4) node [anchor=north west][inner sep=0.75pt]    {$s_{2}$};
\draw (57,142.4) node [anchor=north west][inner sep=0.75pt]    {$A_{0}$};
\draw (122,142.53) node [anchor=north west][inner sep=0.75pt]    {$A_{1}$};
\draw (187,144.4) node [anchor=north west][inner sep=0.75pt]    {$A_{2}$};
\draw (253,144.4) node [anchor=north west][inner sep=0.75pt]    {$A_{3}$};
\draw (89,158.4) node [anchor=north west][inner sep=0.75pt]    {$s_{1}$};
\draw (153,158.4) node [anchor=north west][inner sep=0.75pt]    {$s_{2}$};
\draw (218,158.4) node [anchor=north west][inner sep=0.75pt]    {$s_{1}$};

\end{tikzpicture}
    \caption{A triangulation for $r=2$, $b_+ = (s_1,s_2,s_2)$, and $b_- = (s_1,s_2,s_1)$.}
    \label{fig:triang_example}
\end{figure}
Given a double word $\beta$ (See Definition~\ref{def: double_word}), a left-to-right scanning of $\beta$ gives us two positive braid words $\beta_+, \beta_-$ according to its top and bottom words. We can build a triangulation $C_\beta$ in the double Bott-Samelson variety $\text{Conf}({\beta_+},{\beta_-})$ in the following way. Scan $\beta$ from left to right. For each positive entry $i$ of $\beta$, we append to the right an upward-facing triangle with $s_i$ on its top edge, and for each negative entry $-i$, we append to the right a downward-facing triangle with $s_i$ on its bottom edge. This defines a unique way to connect diagonals of the configuration diagram of $\text{Conf}({\beta_+},{\beta_-})$.

A triangulation $C$ cuts out a subset $T(C)$ of $\text{Conf}({\beta_+},{\beta_-})$ where we require that any two flags connected by a diagonal are in general position. ~\cite{Shen_Weng_2021} showed this is an open algebraic torus inside $\text{Conf}({\beta_+},{\beta_-})$.

\begin{lem}\label{lem:forward_inductive_param}
    Consider one single triangle. See Figure~\ref{fig:single_triangle_param}. Once we fix the choice of $g\in G, t_0,t_0'\in T$ for the first diagonal of flags $A^0,A_0$ as in Lemma~\ref{lem:normalization_on_diagonal}. Then, in the left diagram, there is a unique $p\in \mathbb{C}^*$, such that $A^1 = ge_{-i}(p)t_1U_+$, where $t_1 = t_0(t_0^{\alpha_i}p)^{-\alpha^{\vee}_i}$. Similarly, in the right diagram, there is a unique $q\in \mathbb{C}^*$, such that $A_1 = ge_i(q)(t_1')^{-1}U_-$, where $t'_1= t'_0((t'_0)^{\alpha_i}q)^{-\alpha^{\vee}_i}$.
    \begin{figure}[h!]
    \centering
    \tikzset{every picture/.style={line width=0.75pt}} 

\begin{tikzpicture}[x=0.75pt,y=0.75pt,yscale=-1,xscale=1]

\draw    (87,117) -- (108.94,181.28) ;
\draw    (170,123) -- (148.94,178.28) ;
\draw    (115,107) -- (157.91,107) ;
\draw [shift={(159.91,107)}, rotate = 180] [color={rgb, 255:red, 0; green, 0; blue, 0 }  ][line width=0.75]    (10.93,-3.29) .. controls (6.95,-1.4) and (3.31,-0.3) .. (0,0) .. controls (3.31,0.3) and (6.95,1.4) .. (10.93,3.29)   ;
\draw    (424,120) -- (445.94,184.28) ;
\draw    (372,121) -- (350.94,176.28) ;
\draw    (395,193) -- (436.91,193) ;
\draw [shift={(438.91,193)}, rotate = 180] [color={rgb, 255:red, 0; green, 0; blue, 0 }  ][line width=0.75]    (10.93,-3.29) .. controls (6.95,-1.4) and (3.31,-0.3) .. (0,0) .. controls (3.31,0.3) and (6.95,1.4) .. (10.93,3.29)   ;

\draw (35,93.4) node [anchor=north west][inner sep=0.75pt]    {$A^{0} =gt_{0} U_{+}$};
\draw (77,179.4) node [anchor=north west][inner sep=0.75pt]    {$A_{0} =g( t'_{0})^{-1} U_{-}$};
\draw (168,94.4) node [anchor=north west][inner sep=0.75pt]    {$A^{1}$};
\draw (128,90) node [anchor=north west][inner sep=0.75pt]    {$s_{i}$};
\draw (360,96.4) node [anchor=north west][inner sep=0.75pt]    {$A^{0} =gt_{0} U_{+}$};
\draw (291,181.4) node [anchor=north west][inner sep=0.75pt]    {$A_{0} =g( t'_{0})^{-1} U_{-}$};
\draw (443,183.4) node [anchor=north west][inner sep=0.75pt]    {$A_{1}$};
\draw (407,175) node [anchor=north west][inner sep=0.75pt]    {$s_{i}$};

\end{tikzpicture}
    \caption{Parameterization of a single triangle.}
    \label{fig:single_triangle_param}
\end{figure}
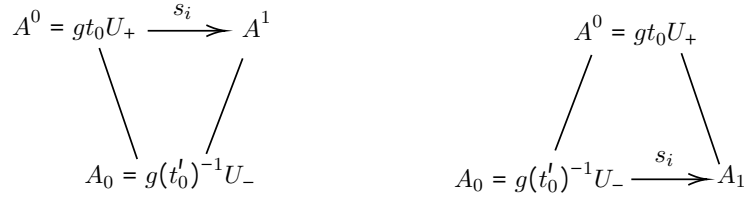
\end{lem}
\begin{proof}
We only prove the case for the left diagram. Let $A^1 = xU_+$. By Lemma~\ref{lem: relative_pos_matrix}, we can choose $x$ such that $(gt_0)^{-1}x = B_i(z)$. Since $A^1, A_0$ are in general position, we have $\Delta_j((t'_0g^{-1})gt_0B_i(z))\neq 0,\forall j\in [r+1]$. Since $t'_0t_0\in T$, we have $\Delta_j(B_i(z))\neq 0, \forall j\in [r+1]$. In particular, $z = \Delta_i(B_i(z))\neq 0$. 

Thus, $B_i(z)U_+ = \phi_i\left(\m{z&0\\1&z^{-1}}\right)U_+ = e_{-i}(z^{-1})\phi_i\left(\m{z&0\\0&z^{-1}}\right)U_+$, and  $$t_0B_i(z)U_+ = e_{-i}(t_0^{-\alpha_i}z^{-1})t_0\phi_i\left(\m{z&0\\0&z^{-1}}\right)U_+$$

Let $p := (t^{\alpha_i}_0z)^{-1}$, then we get $$t_0B_i(z)U_+ = e_{-i}(p)\cdot(t_0(t_0^{\alpha_i}p)^{-\alpha^{\vee}_i})U_+$$ 

Let $t_1 := t_0(t_0^{\alpha_i}p)^{-\alpha^{\vee}_i}$, we have $$A^1 = gt_0B_i(z) U_+= ge_{-i}(p)t_1U_+$$
\end{proof}

\begin{rem}\label{rem: back_propagation_of_decoration}
    Note that $t\mapsto (t(t^{\alpha_i}p)^{-\alpha^{\vee}_i})$ is an involution for fixed parameters $i\in [r], p\in \mathbb{C}^*$. Thus, in the above Lemma~\ref{lem:forward_inductive_param}, $t_0 = (t_1(t_1^{\alpha_i}p)^{-\alpha^{\vee}_i})$.  
\end{rem}

In the following lemma, we adapt Lemma 3.10, Corollary 3.22, Proposition 3.23 in \cite[Section 3.3]{Shen_Weng_2021} to provide our precise way to parameterize this torus.

\begin{lem}\label{lem:torus_param}
    Given any point $(A^0,...,A^m,A_0,...,A_n)\in T(C)$, after applying a left action, normalize to $A^0=t_0U_+, A_0 = U_-$, for some unique diagonal matrix $t_0$. For each upward facing (resp. downward facing) triangle with $s_i$ on the base, one can find a unique $p\in\mathbb{C}^*$ and associate a Chevalley generator $e_{-i}(p)$ (resp. $e_i(p)$), such that each flag has a unique matrix representative which is the product of Chevalley generators coming from the triangles to the left of the first diagonal joining the vertex of the flag in $C$, followed by right-multiplying by a unique diagonal matrix.  
\end{lem}
\begin{proof}
    We use induction on the index of the triangles in $C$. Suppose all flags in $(A^0,...,A^{k}, A_0,...,A_l)$ have a unique parametrization as described in the statement, where $A^k$ and $A_l$ are joined by a diagonal. We prove the case where $A^k\xrightarrow[]{s_i} A^{k+1}$, and $A^{k+1}$ and $A_l$ are joined by a diagonal. The other case is similar. 
    
    Note that the diagonal connecting $A^k, A^l$ is either the first diagonal connecting $A^k$ or the first diagonal connecting $A^l$. See Figure~\ref{fig:toric_param} on the top row. 

    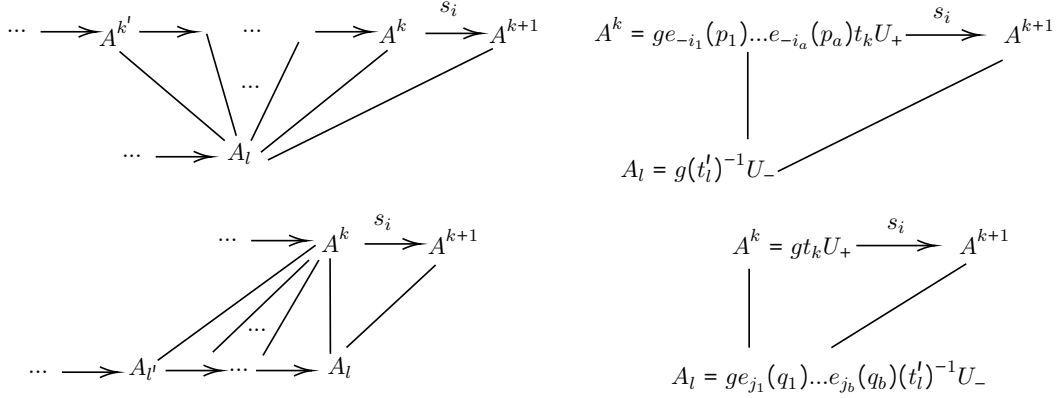
\begin{figure}[h!]
    \centering
    \tikzset{every picture/.style={line width=0.75pt}} 

\begin{tikzpicture}[x=0.75pt,y=0.75pt,yscale=-1,xscale=1]

\draw    (84,68) -- (116,68) ;
\draw [shift={(118,68)}, rotate = 180] [color={rgb, 255:red, 0; green, 0; blue, 0 }  ][line width=0.75]    (10.93,-3.29) .. controls (6.95,-1.4) and (3.31,-0.3) .. (0,0) .. controls (3.31,0.3) and (6.95,1.4) .. (10.93,3.29)   ;
\draw    (191,68) -- (223,68) ;
\draw [shift={(225,68)}, rotate = 180] [color={rgb, 255:red, 0; green, 0; blue, 0 }  ][line width=0.75]    (10.93,-3.29) .. controls (6.95,-1.4) and (3.31,-0.3) .. (0,0) .. controls (3.31,0.3) and (6.95,1.4) .. (10.93,3.29)   ;
\draw    (72,80) -- (135.98,134.24) ;
\draw    (234.98,79.24) -- (157.98,141.24) ;
\draw    (180.98,74.24) -- (151.98,134.24) ;
\draw    (124.98,69.24) -- (142.98,131.24) ;
\draw    (258,69) -- (290,69) ;
\draw [shift={(292,69)}, rotate = 180] [color={rgb, 255:red, 0; green, 0; blue, 0 }  ][line width=0.75]    (10.93,-3.29) .. controls (6.95,-1.4) and (3.31,-0.3) .. (0,0) .. controls (3.31,0.3) and (6.95,1.4) .. (10.93,3.29)   ;
\draw    (298.98,78.24) -- (162.22,145.96) ;
\draw    (94,144) -- (126,144) ;
\draw [shift={(128,144)}, rotate = 180] [color={rgb, 255:red, 0; green, 0; blue, 0 }  ][line width=0.75]    (10.93,-3.29) .. controls (6.95,-1.4) and (3.31,-0.3) .. (0,0) .. controls (3.31,0.3) and (6.95,1.4) .. (10.93,3.29)   ;
\draw    (24,68) -- (56,68) ;
\draw [shift={(58,68)}, rotate = 180] [color={rgb, 255:red, 0; green, 0; blue, 0 }  ][line width=0.75]    (10.93,-3.29) .. controls (6.95,-1.4) and (3.31,-0.3) .. (0,0) .. controls (3.31,0.3) and (6.95,1.4) .. (10.93,3.29)   ;
\draw    (100,274) -- (132,274) ;
\draw [shift={(134,274)}, rotate = 180] [color={rgb, 255:red, 0; green, 0; blue, 0 }  ][line width=0.75]    (10.93,-3.29) .. controls (6.95,-1.4) and (3.31,-0.3) .. (0,0) .. controls (3.31,0.3) and (6.95,1.4) .. (10.93,3.29)   ;
\draw    (193.2,206.81) -- (159.2,264.81) ;
\draw    (187.2,206.81) -- (128.7,262.81) ;
\draw    (191.2,197.81) -- (95.2,267.81) ;
\draw    (221,197) -- (253,197) ;
\draw [shift={(255,197)}, rotate = 180] [color={rgb, 255:red, 0; green, 0; blue, 0 }  ][line width=0.75]    (10.93,-3.29) .. controls (6.95,-1.4) and (3.31,-0.3) .. (0,0) .. controls (3.31,0.3) and (6.95,1.4) .. (10.93,3.29)   ;
\draw    (264.2,209.81) -- (210,260) ;
\draw    (155,274) -- (187,274) ;
\draw [shift={(189,274)}, rotate = 180] [color={rgb, 255:red, 0; green, 0; blue, 0 }  ][line width=0.75]    (10.93,-3.29) .. controls (6.95,-1.4) and (3.31,-0.3) .. (0,0) .. controls (3.31,0.3) and (6.95,1.4) .. (10.93,3.29)   ;
\draw    (152,195) -- (184,195) ;
\draw [shift={(186,195)}, rotate = 180] [color={rgb, 255:red, 0; green, 0; blue, 0 }  ][line width=0.75]    (10.93,-3.29) .. controls (6.95,-1.4) and (3.31,-0.3) .. (0,0) .. controls (3.31,0.3) and (6.95,1.4) .. (10.93,3.29)   ;
\draw    (200,205.76) -- (200.2,263.81) ;
\draw    (37,275) -- (69,275) ;
\draw [shift={(71,275)}, rotate = 180] [color={rgb, 255:red, 0; green, 0; blue, 0 }  ][line width=0.75]    (10.93,-3.29) .. controls (6.95,-1.4) and (3.31,-0.3) .. (0,0) .. controls (3.31,0.3) and (6.95,1.4) .. (10.93,3.29)   ;
\draw    (453.98,80.24) -- (453.98,133.33) ;
\draw    (550,70) -- (593,70) ;
\draw [shift={(595,70)}, rotate = 180] [color={rgb, 255:red, 0; green, 0; blue, 0 }  ][line width=0.75]    (10.93,-3.29) .. controls (6.95,-1.4) and (3.31,-0.3) .. (0,0) .. controls (3.31,0.3) and (6.95,1.4) .. (10.93,3.29)   ;
\draw    (608.98,85.24) -- (472.22,152.96) ;
\draw    (454.82,212.09) -- (454.82,260.09) ;
\draw    (520,198) -- (566,198) ;
\draw [shift={(568,198)}, rotate = 180] [color={rgb, 255:red, 0; green, 0; blue, 0 }  ][line width=0.75]    (10.93,-3.29) .. controls (6.95,-1.4) and (3.31,-0.3) .. (0,0) .. controls (3.31,0.3) and (6.95,1.4) .. (10.93,3.29)   ;
\draw    (586.82,209.09) -- (504.82,260.09) ;

\draw (58,57) node [anchor=north west][inner sep=0.75pt]    {$A^{k'}$};
\draw (228,58.4) node [anchor=north west][inner sep=0.75pt]    {$A^{k}$};
\draw (136,134.4) node [anchor=north west][inner sep=0.75pt]    {$A_{l}$};
\draw (145,65.4) node [anchor=north west][inner sep=0.75pt]    {$...$};
\draw (144,99.4) node [anchor=north west][inner sep=0.75pt]    {$...$};
\draw (295,58.4) node [anchor=north west][inner sep=0.75pt]    {$A^{k+1}$};
\draw (72,141.4) node [anchor=north west][inner sep=0.75pt]    {$...$};
\draw (2,65.4) node [anchor=north west][inner sep=0.75pt]    {$...$};
\draw (76,265.4) node [anchor=north west][inner sep=0.75pt]    {$A_{l'}$};
\draw (193,185.4) node [anchor=north west][inner sep=0.75pt]    {$A^{k}$};
\draw (197,264.4) node [anchor=north west][inner sep=0.75pt]    {$A_{l}$};
\draw (148,247.21) node [anchor=north west][inner sep=0.75pt]    {$...$};
\draw (258,186.4) node [anchor=north west][inner sep=0.75pt]    {$A^{k+1}$};
\draw (137,271.4) node [anchor=north west][inner sep=0.75pt]    {$...$};
\draw (130,192.4) node [anchor=north west][inner sep=0.75pt]    {$...$};
\draw (15,272.4) node [anchor=north west][inner sep=0.75pt]    {$...$};
\draw (359,57.4) node [anchor=north west][inner sep=0.75pt]    {$A^{k} =ge_{-i_{1}}( p_{1}) ...e_{-i_{a}}( p_{a}) t_{k} U_{+}{}$};
\draw (608,58.4) node [anchor=north west][inner sep=0.75pt]    {$A^{k+1}$};
\draw (375,140.4) node [anchor=north west][inner sep=0.75pt]    {$A_{l} =g( t'_{l})^{-1} U_{-}$};
\draw (565,52) node [anchor=north west][inner sep=0.75pt]    {$s_{i}$};
\draw (443,186.4) node [anchor=north west][inner sep=0.75pt]    {$A^{k} =gt_{k} U_{+}{}$};
\draw (581,186.4) node [anchor=north west][inner sep=0.75pt]    {$A^{k+1}$};
\draw (406,266.4) node [anchor=north west][inner sep=0.75pt]    {$A_{l} =ge_{j_{1}}( q_{1}) ...e_{j_{b}}( q_{b})( t'_{l})^{-1} U_{-}$};
\draw (537,180) node [anchor=north west][inner sep=0.75pt]    {$s_{i}$};
\draw (265,48) node [anchor=north west][inner sep=0.75pt]    {$s_{i}$};
\draw (225,178) node [anchor=north west][inner sep=0.75pt]    {$s_{i}$};

\end{tikzpicture}
    \caption{Inductive step for parametrizing a triangulation.}
    \label{fig:toric_param}
    \end{figure}

    Note that
    \begin{equation*}
        \begin{split}
            A_l &= g(t'_l)^{-1}U_-\\
            &= g(e_{-i_1}(p_1)...e_{-i_a}(p_a))(e_{-i_1}(p_1)...e_{-i_a}(p_a))^{-1}(t'_l)^{-1}U_-\\
            &=g(e_{-i_1}(p_1)...e_{-i_a}(p_a))(e_{-i_a}(-p_a)...e_{-i_1}(-p_1))(t'_l)^{-1}U_-\\
            &= g(e_{-i_1}(p_1)...e_{-i_a}(p_a))(t'_l)^{-1}(e_{-i_a}(-(t'_l)^{-\alpha_{i_a}}p_a)...e_{-i_1}(-(t'_l)^{-\alpha_{i_1}}p_1))U_-\\
            &= g(e_{-i_1}(p_1)...e_{-i_a}(p_a))(t'_l)^{-1}U_-\\
        \end{split}
    \end{equation*}
    We are in the shape to apply Lemma~\ref{lem:forward_inductive_param} to get that there is a unique $p_{k+1}\in \mathbb{C}^*$, such that $$A^{k+1} = g(e_{-i_1}(p_1)...e_{-i_a}(p_a))e_{-i}(p_{k+1})t_{k+1}U_+$$ where $t_{k+1} = t_k(t^{\alpha_i}_kp_{k+1})^{-\alpha^\vee_i}$.
\end{proof}

\begin{exmp}\label{exmp:triang_param}
Let $r = 2$, $\beta = (-1,1,2,-2,-1,2)$, so the top word is $\beta_+ = (1,2,2)$, and the bottom word is $\beta_- = (1,2,1)$. Since $\beta_+$ is not reduced, $G_\beta$ is also not reduced. The corresponding plabic fence with right canonical gauge is in Figure~\ref{fig:weighted_fence}, and the point in $Conf(s_1s_2s_2,s_1s_2s_1)$ it parameterizes is in Figure~\ref{fig:param_point}, where $d = diag\{d_1,d_2,d_3\}$, and all the other decorations are uniquely determined via Lemma~\ref{lem:torus_param}.

\begin{figure}[h!]
    \centering
    \tikzset{every picture/.style={line width=0.75pt}} 

\begin{tikzpicture}[x=0.75pt,y=0.75pt,yscale=-1,xscale=1,font=\small]

\draw    (42,126) -- (42,154) -- (42,176) ;
\draw    (55,191) -- (90,191) ;
\draw [shift={(92,191)}, rotate = 180] [color={rgb, 255:red, 0; green, 0; blue, 0 }  ][line width=0.75]    (10.93,-3.29) .. controls (6.95,-1.4) and (3.31,-0.3) .. (0,0) .. controls (3.31,0.3) and (6.95,1.4) .. (10.93,3.29)   ;
\draw    (57,107) -- (81.72,107) ;
\draw [shift={(83.72,107)}, rotate = 180] [color={rgb, 255:red, 0; green, 0; blue, 0 }  ][line width=0.75]    (10.93,-3.29) .. controls (6.95,-1.4) and (3.31,-0.3) .. (0,0) .. controls (3.31,0.3) and (6.95,1.4) .. (10.93,3.29)   ;
\draw    (58,125) -- (111,178) ;
\draw    (185,191) -- (213,191) ;
\draw [shift={(215,191)}, rotate = 180] [color={rgb, 255:red, 0; green, 0; blue, 0 }  ][line width=0.75]    (10.93,-3.29) .. controls (6.95,-1.4) and (3.31,-0.3) .. (0,0) .. controls (3.31,0.3) and (6.95,1.4) .. (10.93,3.29)   ;
\draw    (430,190) -- (461.49,190) ;
\draw [shift={(463.49,190)}, rotate = 180] [color={rgb, 255:red, 0; green, 0; blue, 0 }  ][line width=0.75]    (10.93,-3.29) .. controls (6.95,-1.4) and (3.31,-0.3) .. (0,0) .. controls (3.31,0.3) and (6.95,1.4) .. (10.93,3.29)   ;
\draw    (200,108) -- (235.49,108) ;
\draw [shift={(237.49,108)}, rotate = 180] [color={rgb, 255:red, 0; green, 0; blue, 0 }  ][line width=0.75]    (10.93,-3.29) .. controls (6.95,-1.4) and (3.31,-0.3) .. (0,0) .. controls (3.31,0.3) and (6.95,1.4) .. (10.93,3.29)   ;
\draw    (393,108) -- (445,108) ;
\draw [shift={(447,108)}, rotate = 180] [color={rgb, 255:red, 0; green, 0; blue, 0 }  ][line width=0.75]    (10.93,-3.29) .. controls (6.95,-1.4) and (3.31,-0.3) .. (0,0) .. controls (3.31,0.3) and (6.95,1.4) .. (10.93,3.29)   ;
\draw    (140.25,124) -- (140.25,176.87) ;
\draw    (169,177) -- (265.72,124.32) ;
\draw    (305,124) -- (304.72,178.32) ;
\draw    (366,123) -- (480.72,175.32) ;
\draw    (552,125) -- (552,175.32) ;

\draw (22,98.4) node [anchor=north west][inner sep=0.75pt]    {$t_{0} U_{+}$};
\draw (32,183.4) node [anchor=north west][inner sep=0.75pt]    {$U_{-}$};
\draw (92,99.4) node [anchor=north west][inner sep=0.75pt]    {$\textcolor[rgb]{0.24,0.27,0.89}{e_{1}( q_{1})}\textcolor[rgb]{0.24,0.27,0.89}{e}\textcolor[rgb]{0.24,0.27,0.89}{_{-1}}\textcolor[rgb]{0.24,0.27,0.89}{(}\textcolor[rgb]{0.24,0.27,0.89}{p}\textcolor[rgb]{0.24,0.27,0.89}{_{1}}\textcolor[rgb]{0.24,0.27,0.89}{)} t_{1} U_{+}$};
\draw (95,180.4) node [anchor=north west][inner sep=0.75pt]    {$\textcolor[rgb]{0.24,0.27,0.89}{e_{1}( q_{1})}( t'_{1})^{-1} U_{-}$};
\draw (241,98.4) node [anchor=north west][inner sep=0.75pt]    {$\textcolor[rgb]{0.24,0.27,0.89}{e_{1}( q_{1}) e_{-1}( p_{1}) e_{-2}( p_{2})} t_{2} U_{+}$};
\draw (220,179.4) node [anchor=north west][inner sep=0.75pt]    {$\textcolor[rgb]{0.24,0.27,0.89}{e_{1}( q_{1}) e_{-1}( p_{1}) e_{-2}( p_{2}) e_{2}( q_{2})}( t'_{2})^{-1} U_{-}$};
\draw (466,178.4) node [anchor=north west][inner sep=0.75pt]    {$\textcolor[rgb]{0.24,0.27,0.89}{e_{1}( q_{1}) e_{-1}( p_{1}) e_{-2}( p_{2}) e_{2}( q_{2}) e_{1}( q_{3})}( t'_{3})^{-1} U_{-}$};
\draw (452,97.4) node [anchor=north west][inner sep=0.75pt]    {$\textcolor[rgb]{0.24,0.27,0.89}{e_{1}( q_{1}) e_{-1}( p_{1}) e_{-2}( p_{2}) e_{2}( q_{2}) e_{1}( q_{3}) e_{-2}( q_{3})}\textcolor[rgb]{0.25,0.46,0.02}{d} U_{+}$};
\draw (60,90) node [anchor=north west][inner sep=0.75pt]    {$s_{1}$};
\draw (65,199.4) node [anchor=north west][inner sep=0.75pt]    {$s_{1}$};
\draw (210,90) node [anchor=north west][inner sep=0.75pt]    {$s_{2}$};
\draw (184,198.4) node [anchor=north west][inner sep=0.75pt]    {$s_{2}$};
\draw (405,90) node [anchor=north west][inner sep=0.75pt]    {$s_{2}$};
\draw (435,193.4) node [anchor=north west][inner sep=0.75pt]    {$s_{1}$};

\end{tikzpicture}
    \caption{Parameterized point in $Conf(s_1s_2s_2,s_1s_2s_1).$ The matrix $e_i(q)$ is the identity matrix with a $q$ instead of a 0 in the $(i,i+1)$ position. Similarly, $e_{-i}(p)$ is the identity matrix with a $p$ replacing the 0 in the $(i+1,i)$ position.}
    \label{fig:param_point}
    \end{figure}
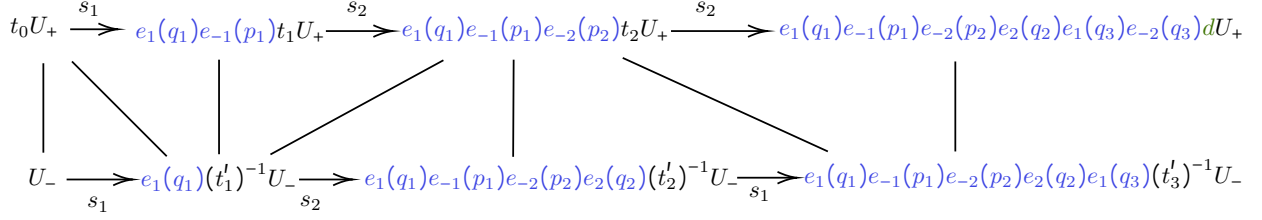    
\end{exmp}

\begin{rem}
This parameterization is analogous to the one in \cite[Section 3]{Shen_Weng_2021}. They are not exactly the same, however, as we always have a single diagonal matrix multiplying the Chevalley matrices at the end, while Shen and Weng's has some diagonal matrices interspersed throughout the product (see the proof of \cite[Proposition 3.23]{Shen_Weng_2021} for details). This amounts to some monomial factors multiplying the $p's$ and $q's$.
\end{rem}

Moreover, for the same triangulation $C$, if a diagonal joining $x U_+$ and $U_- y$ exists in $C$ we consider their generalized minors $\Delta_i(yx)$ for each $i\in [r]$. Taking all such minors and removing duplicates (see \cite[Section 3.3]{Shen_Weng_2021}) gives monomial coordinates for the same torus $T(C).$

Given a double word $\beta$, we have an associated plabic fence $G_\beta$ and an associated triangulation $C_\beta$. We give here a correspondence between the faces of $G_\beta$ and the non-vanishing generalized minors of $C_\beta$ which is an interpretation of the string diagrams of \cite[Section 3]{Shen_Weng_2021} in this setting. Therefore, we can label the corresponding generalized minor of the face $\binom{i}{a}$ as $A_{\binom{i}{a}}$.  

A \textbf{slice} of a plabic fence is any vertical line that does not intersect with any vertical edge of the plabic fence. Each slice corresponds to a diagonal in the corresponding triangulation: the vertical edges to the left of the slice correspond to the triangles to the left of the corresponding diagonal. Let $xU_+$ and $U_-y$ be the two flags joined by this diagonal. Then we associate the face on the $i$-th row intersected by this slice the generalized minor $\Delta_i(yx)$.

Note that a face $\binom{i}{a}$ may be intersected by different slices, so a priori, different slices on this face may give different generalized minors. However, Proposition 3.21 in ~\cite{Shen_Weng_2021} guarantees that if there are no vertical edges in the $i$-th row between the two different slices on the face $\binom{i}{a}$, then the generalized minors given by the two slices are equal.

\begin{exmp}
    See Figure~\ref{fig:fence_face_to_gen_minor}. For $\beta = (-1,1,2,-2,-1,2)$, on the left is the plabic fence $G_\beta$, and on the right is the triangulation $C_\beta$. The blue and orange slices on the left correspond to the blue and orange diagonals on the right. 
    \begin{figure}[h!]
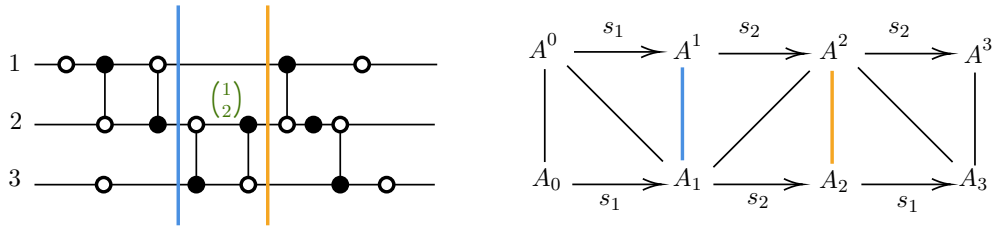

    \centering
    \includestandalone[mode=tex, width=0.8\textwidth]
    {figures/fence_face_to_gen_minor}
    \caption{Illustration of slices on plabic fences and their corresponding diagonals}
    \label{fig:fence_face_to_gen_minor}
    \end{figure}
    Let $A^1 = xU_+, A_1 = U_-y$. Then, by Lemma~\ref{lem: relative_pos_matrix}, $A^2 = xB_2(z)U_+, A_2 = U_-B_{2}(w)y$, for some $z,w\in \mathbb{C}$. With respect to the blue diagonal, the generalized minor for the face $\binom{1}{2}$ is $\Delta_1(yx)$, and with respect to the orange diagonal, the generalized minor for the face $\binom{1}{2}$ is $\Delta_1(B_2(w)yxB_2(z))$. However, since the right action by $B_2(z)$ is a column operation on the second and third columns, leaving the first column unchanged, we have $\Delta_1(B_2(w)yxB_2(z)) = \Delta_1(B_2(w)yx)$. By the same reasoning for the left action by $B_2(w)$, we have $\Delta_1(B_2(w)yx)= \Delta_1(yx)$, illustrating that the two generalized minors are equal. 
\end{exmp}

\begin{thm}[{\cite[Theorems 3.45, 4.13]{Shen_Weng_2021}}]
    Given a triangulation $C_\beta$ in $\text{Conf}({\beta_+},{\beta_-})$ we have associated it a quiver $Q_\beta$ (see \cite[Notation 3.5]{Shen_Weng_2021}). Assign to vertices of $Q_\beta$ the generalized minors $A_{\binom{i}{a}}$ associated to the corresponding faces of $G_\beta$. Then the cluster algebra generated by this seed $\Sigma_\beta:= (\{A_{\binom{i}{a}}\}_{\binom{i}{a}\in F(G_\beta)}, Q_\beta)$ coincides with the coordinate ring of $\text{Conf}({\beta_+},{\beta_-})$. 
\end{thm}


\begin{exmp}
    Both the plabic fence in Figure~\ref{fig:weighted_fence} and the triangulation in Figure~\ref{fig:triang_example} correspond to the double braid word $(-1,1,2,-2,-1,2).$
\end{exmp}

\begin{rem}
    If we specialize \cite[Notation 3.5]{Shen_Weng_2021} to type A, and ignore arrows between boundary faces, we get the same quiver as in Section~\ref{sec: plabic_graphs}. This follows from the discussion of slices in the previous page together with \cite[Remark 3.6]{Shen_Weng_2021}.
    
    Moreover, the connection between plabic fences and double Bott-Samelson varieties has been noted before \cite{grid_plabic,filling_seed}. Often, the varieties considered assume $\beta_-= (e),$ so that the fence only has one type of column. Any double Bott-Samelson variety is isomorphic to one of this type \cite{Shen_Weng_2021}. However, as our goal is to generalize a minimal matching construction to generic, (possibly) non-reduced plabic fences, we keep the two different types of flags. This does create some additional technical difficulties (see Remark~\ref{rem: technical_difficulty_of_contructions}).
\end{rem}

\begin{defn}\label{def: canonical_parametrization}
    Given a right-canonically weighted fence $G_\beta$, we can also associate to it a point in $T(C_\beta)$. A left-to-right scanning of the vertical edges' weights allows us to form the Chevalley generator factorization of the undecorated flags in the configuration space (see Figure~\ref{fig:param_point}). Moreover, if we require that the product of the rightmost horizontal edges' weights be 1, then we can use these weights to fix the decoration of the top-right flag, $A^m = g\cdot diag\{d_1,...,d_{r+1}\}U_+$, where $gB_+$ is the undecorated counterpart of $A^m$, and $d^{-1}_i$ is the weight of the rightmost horizontal edge on the $i$-th row, for $1\leq i\leq r+1$. Importantly, by Remark~\ref{rem: back_propagation_of_decoration}, this allows us to recover a unique parametrization in the form of Lemma~\ref{lem:torus_param}. We call this the \textbf{right canonical parametrization of $T(C_\beta)$} through weights of $G_\beta$.
\end{defn}

\begin{exmp}\label{exmp: computing_decorations}
    Continuing with the parametrization in Example~\ref{exmp:triang_param}. If we fix the parameters $p_i, q_i, i\in\{1,2,3\}$ and $d = \text{diag}\{d_1,...,d_{r+1}\}$ which come from the column and rightmost horizontal-edge weights of the canonically weighted plabic fence in Example~\ref{exem: weighted_fence}, we can recover all the torus elements following Remark~\ref{rem: back_propagation_of_decoration} and Lemma~\ref{lem:forward_inductive_param} as follows:
    
    $$\begin{cases}
        t_2 = d(d^{\alpha_2}p_3)^{-\alpha^\vee_2} = diag\{d_1,d_3p^{-1}_3,d_2p_3\}\\
        t_1 = t_2(t^{\alpha_2}_2p_2)^{-\alpha^\vee_2} = diag\{d_1,d_2p_3p^{-1}_2, d_3p^{-1}_3p_2\}\\
        t_0 = t_1(t^{\alpha_1}_1p_1)^{-\alpha^\vee_1} = diag\{d_2p_3p^{-1}_2p^{-1}_1, d_1p_1,d_3p^{-1}_3p_2\}\\
        t'_1 = \text{Id}_3(\text{Id}_3^{-\alpha_1}q_1)^{-\alpha^\vee_1} = diag\{q^{-1}_1,q_1,1\}\\
        t'_2 = t'_1(t'^{\alpha_2}_1q_2)^{-\alpha^\vee_2} = diag\{q^{-1}_1, q^{-1}_2, q_1q_2\}\\
        t'_3 = t'_2(t'^{\alpha_1}_2q_3)^{-\alpha^\vee_1} = diag\{q^{-1}_2q^{-1}_3, q^{-1}_1q_3, q_1q_2\}
    \end{cases}$$
\end{exmp}

\subsection{Comments on our parametrization}

In this subsection, we make the point that our parametrization in Lemma~\ref{lem:torus_param} is compatible with moves on right-canonically weighted plabic fences, further strengthening the connections with plabic graph theory.

First, note that when we perform a braid move, the column weights change in Figure~\ref{fig:braid_move} which we obtained from general moves on plabic graphs, actually corresponds to the following matrix identity

\begin{equation}\label{eqn: Lusztig_relation}
    e_{-i}(p_{k})e_{-(i+1)}(p_{k+1})e_{-i}(p_{k+2}) = e_{-(i+1)}\left(\frac{p_{k+1}p_{k+2}}{p_k+p_{k+2}}\right)e_{-i}(p_k+p_{k+2})e_{-(i+1)}\left(\frac{p_kp_{k+1}}{p_k+p_{k+2}}\right).
\end{equation}

This is saying that, a braid move $(...i,i+1,i...)\leftrightarrow(...i+1,i,i+1,...)$ gives a change of coordinates in our parametrization (see also \cite[Equation 2.10]{DoubleBruhatCells}).

On the other hand, if we perform a same-row column swap, the change of column and row weights in Figure~\ref{fig:consecutive_col_swap} gives us the transition map on the intersection of the two corresponding tori. The key matrix identity appearing here is, whenever $pq+1\neq 0$:

\begin{equation}\label{eqn: column_swap_matrix_id}
    e_{-i}(p)e_{i}(q) = e_{i}\left(\frac{q}{1+pq}\right)e_{-i}(p(1+pq))\cdot (1+pq)^{-\alpha^\vee_i}
\end{equation}
which is essentially (2.11) in ~\cite{DoubleBruhatCells}, except that we pushed the diagonal matrix to the right end to satisfy our convention of parameterizing the unique representative of each flag. We now give a derivation of this transition map in our framework of double Bott-Samelson varieties.

We start with a triangulation as the solid lines in Figure~\ref{fig:change_of_coords_square_move}, and impose that $A^{k-1}, A_l$ are in general position. 

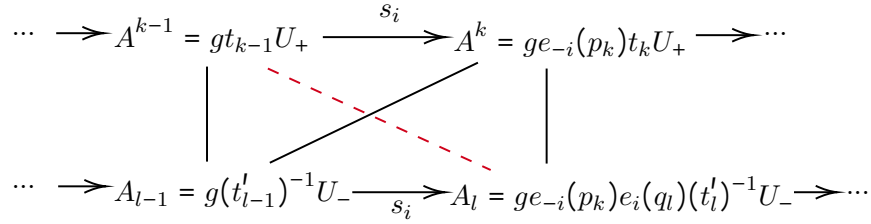
\begin{figure}[h!]
    \centering
    \tikzset{every picture/.style={line width=0.75pt}} 

\begin{tikzpicture}[x=0.75pt,y=0.75pt,yscale=-1,xscale=1]

\draw    (141,83) -- (141,131.4) ;
\draw    (314,84) -- (314,132.4) ;
\draw    (200,68) -- (260.45,68) ;
\draw [shift={(262.45,68)}, rotate = 180] [color={rgb, 255:red, 0; green, 0; blue, 0 }  ][line width=0.75]    (10.93,-3.29) .. controls (6.95,-1.4) and (3.31,-0.3) .. (0,0) .. controls (3.31,0.3) and (6.95,1.4) .. (10.93,3.29)   ;
\draw    (64.18,66) -- (86.45,66) ;
\draw [shift={(88.45,66)}, rotate = 180] [color={rgb, 255:red, 0; green, 0; blue, 0 }  ][line width=0.75]    (10.93,-3.29) .. controls (6.95,-1.4) and (3.31,-0.3) .. (0,0) .. controls (3.31,0.3) and (6.95,1.4) .. (10.93,3.29)   ;
\draw    (65.18,144) -- (87.45,144) ;
\draw [shift={(89.45,144)}, rotate = 180] [color={rgb, 255:red, 0; green, 0; blue, 0 }  ][line width=0.75]    (10.93,-3.29) .. controls (6.95,-1.4) and (3.31,-0.3) .. (0,0) .. controls (3.31,0.3) and (6.95,1.4) .. (10.93,3.29)   ;
\draw    (390,67) -- (417.45,67) ;
\draw [shift={(419.45,67)}, rotate = 180] [color={rgb, 255:red, 0; green, 0; blue, 0 }  ][line width=0.75]    (10.93,-3.29) .. controls (6.95,-1.4) and (3.31,-0.3) .. (0,0) .. controls (3.31,0.3) and (6.95,1.4) .. (10.93,3.29)   ;
\draw    (440,147) -- (460,147) ;
\draw [shift={(462,147)}, rotate = 180] [color={rgb, 255:red, 0; green, 0; blue, 0 }  ][line width=0.75]    (10.93,-3.29) .. controls (6.95,-1.4) and (3.31,-0.3) .. (0,0) .. controls (3.31,0.3) and (6.95,1.4) .. (10.93,3.29)   ;
\draw    (215,147) -- (260,147) ;
\draw [shift={(262.18,147)}, rotate = 180] [color={rgb, 255:red, 0; green, 0; blue, 0 }  ][line width=0.75]    (10.93,-3.29) .. controls (6.95,-1.4) and (3.31,-0.3) .. (0,0) .. controls (3.31,0.3) and (6.95,1.4) .. (10.93,3.29)   ;
\draw    (279,81) -- (173,132) ;
\draw [color={rgb, 255:red, 208; green, 2; blue, 27 }  ,draw opacity=1 ] [dash pattern={on 4.5pt off 4.5pt}]  (172,84) -- (285.5,135.68) ;

\draw (92,56.4) node [anchor=north west][inner sep=0.75pt]    {$A^{k-1} =gt_{k-1} U_{+}$};
\draw (92,135.4) node [anchor=north west][inner sep=0.75pt]    {$A_{l-1} =g( t'_{l-1})^{-1} U_{-}$};
\draw (265,58.4) node [anchor=north west][inner sep=0.75pt]    {$A^{k} =ge_{-i}( p_{k}) t_{k} U_{+}$};
\draw (263,137.4) node [anchor=north west][inner sep=0.75pt]    {$A_{l} =ge_{-i}( p_{k}) e_{i}( q_{l})( t'_{l})^{-1} U_{-}$};
\draw (227,50) node [anchor=north west][inner sep=0.75pt]    {$s_{i}$};
\draw (233,150) node [anchor=north west][inner sep=0.75pt]    {$s_{i}$};
\draw (40,64) node [anchor=north west][inner sep=0.75pt]    {$...$};
\draw (40,142) node [anchor=north west][inner sep=0.75pt]    {$...$};
\draw (423,64) node [anchor=north west][inner sep=0.75pt]    {$...$};
\draw (465,145) node [anchor=north west][inner sep=0.75pt]    {$...$};

\end{tikzpicture}
    \caption{The local picture of a triangulation corresponding to a square face.}
    \label{fig:change_of_coords_square_move}
    \end{figure}

Then, we get $\Delta_j(t'_le_i(-q_l)e_{-i}(-p_k)t_{k-1})\neq 0,\forall j\in [r+1]$, so $\Delta_j(e_i(-q_l)e_{-i}(-p_k))\neq 0$. Note that $$e_i(-q_l)e_{-i}(-p_k) = \phi_i\m{1&-q_l\\0&1\\}\phi_i\m{1&0\\-p_k&1} = \phi_i\m{1+p_kq_l&-q_l\\-p_k&1\\}$$

Thus the minors are non-vanishing if and only if $p_kq_l+1\neq 0$ which is exactly the condition we need to apply the same-row column swap.

Moreover, once $p_kq_l+1\neq 0$, we find a $B_+B_-$ decomposition and push any diagonal matrix to the right end to rewrite the coset representative of $A_l$ as 
\begin{equation}
    \begin{split}
        A_l &= ge_{-i}(p_k)e_i(q_l)(t'_l)^{-1}U_-\\
        &= ge_{i}(\frac{q_l}{1+p_kq_l})e_{-i}(p_k(1+p_kq_l))(1+p_kq_l)^{-\alpha^\vee_i}(t'_l)^{-1}U_-\\
        &= ge_{i}(\frac{q_l}{1+p_kq_l})(1+p_kq_l)^{-\alpha^\vee_i}(t'_l)^{-1}U_-\\
    \end{split}
\end{equation}

On the other hand, the coset representative for $A^k$ can be rewritten as

\begin{equation}
    \begin{split}
        A^k &= ge_{-i}(p_k)t_kU_+\\
        &= ge_{-i}(p_k)e_i(q_l)e_i(-q_l)t_kU_+\\
        &= ge_{i}(\frac{q_l}{1+p_kq_l})e_{-i}(p_k(1+p_kq_l))(1+p_kq_l)^{-\alpha^\vee_i}t_ke_i(-t^{\alpha_i}_kq_l)U_+\\
        &= ge_{i}(\frac{q_l}{1+p_kq_l})e_{-i}(p_k(1+p_kq_l))[(1+p_kq_l)^{-\alpha^\vee_i}t_k]U_+\\
    \end{split}
\end{equation}

Note that the parameters appearing in the Chevalley generators are exactly the same as the new column weights obtained from same-row column swap in Figure~\ref{fig:same_row_swap}.

For any subsequent flag, it has a unique right-parametrized coset representative, $$\cdots e_{-i}(p_k)e_{i}(q_l)e_{\pm i_{1}}(c_{i_1})\cdots e_{\pm i_{N}}(c_{i_N})\cdots t_{j}.$$ Writing $1 + p_kq_l=y$, the matrix commutation corresponding to the above local transition map gives

\begin{equation*}
    \begin{split}
        &\cdots e_{-i}(p_k)e_{i}(q_l)e_{\pm i_{1}}(c_{i_1})\cdots e_{\pm i_{N}}(c_{i_N})\cdots t_{j}\\
        &=\cdots e_{i}(\frac{q_l}{y})e_{-i}(p_k y)(y)^{-\alpha^\vee_i}e_{\pm i_{1}}(c_{i_1})\cdots e_{\pm i_{N}}(c_{i_N})t_{j}\\
        &=\cdots e_{i}(\frac{q_l}{y})e_{-i}(p_k y)e_{\pm i_{1}}((y)^{\mp \langle\alpha^{\vee}_i, \alpha_{i_1}\rangle}c_{i_1})\cdots e_{\pm i_{N}}((y)^{\mp \langle\alpha^{\vee}_i, \alpha_{i_N}\rangle}c_{i_N})[(y)^{-\alpha^\vee_i}t_{j}]\\
    \end{split}
\end{equation*}


where $$\langle \alpha^\vee_i, \alpha_j\rangle =\begin{cases}
    2,\:\:\text{if}\:\:i=j\\
    -1,\:\:\text{if}\:\:j=i-1\:\:\text{or}\:\:j=i+1\\
    0,\:\:\text{otherwise}\\
\end{cases}$$

Thus, the above computation shows the decoration of $A^{m}$ changes by $d\mapsto (1+p_kq_l)^{-\alpha^\vee_i}d$, matching up the change in weights of the last horizontal edges on rows $i,i+1$ in Figure~\ref{fig:same_row_swap}.

Lastly, note that the change in the parameter inside each subsequent Chevalley generator, $c\mapsto (1+p_kq_l)^{\mp \langle\alpha^{\vee}_i, \alpha_{i_1}\rangle}c$ also matches the change in the weight of the corresponding vertical edge under gauge transformations. See Figure~\ref{fig:col_swap_change_of_params}. 

\begin{figure}[h!]
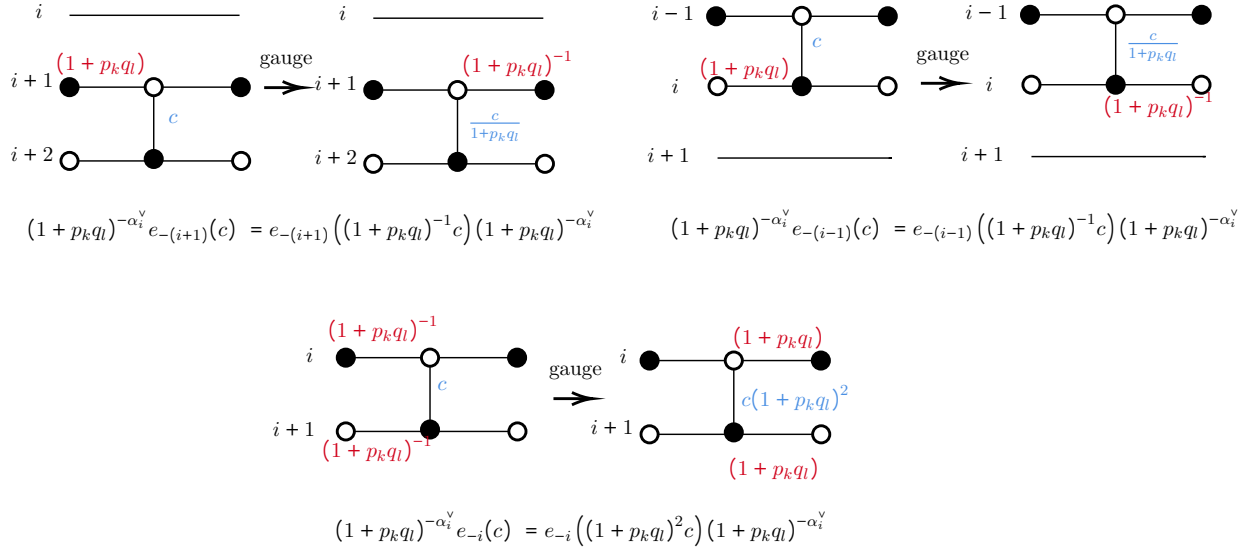

    \centering
    \includestandalone[mode=tex, width=\textwidth]
    {figures/col_swap_change_of_params}
    \caption{Examples of the correspondence between the change in edge weights and in Chevalley parameters.}
    \label{fig:col_swap_change_of_params}
    \end{figure}


\section{Generalized Matchings}\label{sec: gen_matching}

We now present a generalization of \cite{Muller_Speyer_2017}'s minimal matching construction in the setting of plabic fences. In an upcoming work of the second author, this construction is further generalized to the setting of any plabic graphs whose trips have no closed loops.
\begin{lem}\label{rem:monotonicity}
    In a plabic fence $G_\beta$, any strand $\gamma^+_e$ (or $\gamma^-_e$) that travels left along a horizontal $e$ or turns left at the end point of a vertical $e$ gives a strand in the braid presentation of the top-word $\beta^+$ in its reverse direction. Similarly, any strand $\gamma^+_e$ (or $\gamma^-_e$) that travels right along a horizontal $e$ or turns right at the end point of a vertical $e$ gives a strand in the braid presentation of the bottom-word $\beta^-$ in its original direction. Hence, each strand travels only to the right or only to the left in the horizontal direction (but it also travels vertically).
\end{lem}
\begin{proof}
    Consider a strand traveling left on a horizontal edge on the $i$-th line. If it visits a degree $2$ vertex, it stays on the $i$-th line and travels left. If it visits a degree $3$ white vertex (i.e. an endpoint of a vertical edge $e$), there are two cases. Either the vertical edge $e$ corresponds to a letter $s_i$ in the top word $\beta^+$, in which case the strand turns left and travels through $e$ and turns right on the black endpoint of $e$, traveling leftward horizontally on the $(i+1)$-th line, or the vertical edge corresponds to $-s_{i-1}$ in $\beta^-$, in which case the strand turns left and stays traveling horizontally on the $i$-th line. Thus, the index of the horizontal line that the strand travels left on undergoes the change $i\mapsto s_j(i)$ if and only if it passes through a white-over-black vertical edge, hence the collection of all strands traveling left traces out the braid representation of $\beta^+$. Similar discussion can be made to show that strands traveling right trace out the braid representations of $\beta^-$. 
\end{proof}

\begin{rem}
    Plabic fences encode the similar information as the double wiring diagrams from ~\cite{DoubleBruhatCells}. The strands determined above follow the ``wires'' that make a crossing in some particular direction. If the plabic fence is reduced it will correspond to such a double wiring diagram. In this paper, since the braid words $\beta^+,\beta^-$ may not be reduced, $\gamma^+_e,\gamma^-_e$ may cross each other multiple times, making the plabic fence $G_\beta$ non-reduced. 
\end{rem}

\begin{defn}
    For each edge $e$, let the \textbf{positive strand} $\gamma^+_e$ and the \textbf{negative strand} $\gamma^-_e$ be the two trips starting from its middle point, traveling to its white and black endpoints, respectively, and ending at the boundary of the disk. We assign integer values $U(e,f)$ to faces in the following propagating way:

    \begin{enumerate}
        \item[(i)] $U(e,f) := 1$ for the face $f$ that is the right side neighbor of the edge $e$ oriented from white to black.
        \item[(ii)] If $f_1,f_2$ are neighboring faces separated by an edge $e'$, draw a (generic) arrow $(e')^*$ perpendicular to $e'$ from $f_1$ to $f_2$ that does not pass through the middle point of $e'$. Then:
        \begin{enumerate}
            \item[(a)] $U(e,f_1)=U(e,f_2)$ if $(e')^*$ does not cross $\gamma^+_e$ nor $\gamma^-_e$.
            \item[(b)] $U(e,f_1) - U(e,f_2) = 1$ (resp. $-1$) if $(e')^*$ only intersects one of $\gamma^+_e, \gamma^-_e$, and it crosses from the left to the right of $\gamma^-_e$ or from the right to the left of $\gamma^+_e$ (resp. from the right to the left of $\gamma^-_e$ or from the left to the right of $\gamma^+_e$).
            \item[(c)] $U(e,f_1) - U(e,f_2) = 2$ (resp. $-2$) if $(e')^*$ intersects both of $\gamma^+_e, \gamma^-_e$, and it crosses from the left to the right of $\gamma^-_e$ and from the right to the left of $\gamma^+_e$ (resp. from the right to the left of $\gamma^-_e$ and from the left to the right of $\gamma^+_e$).
        \end{enumerate}
        Note that because of Lemma~\ref{rem:monotonicity} these are the only possible cases.
    \end{enumerate}

\end{defn}

\begin{figure}[h!]
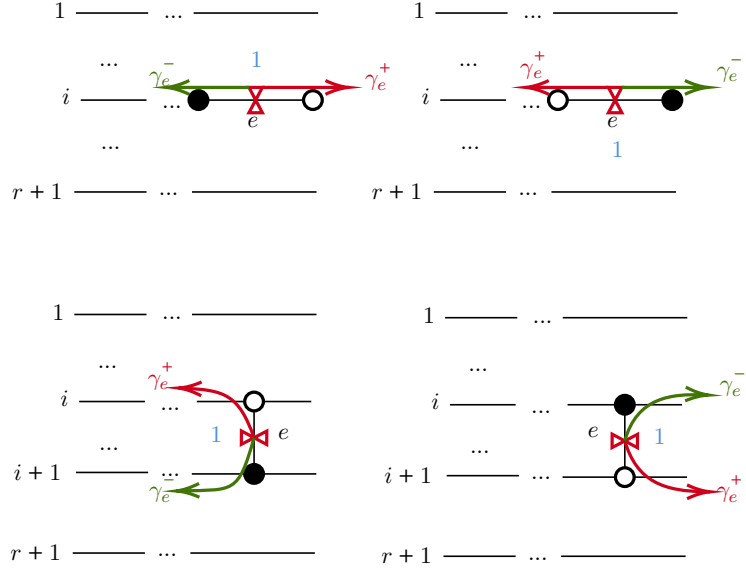

    \centering
    \includestandalone[mode=tex, width=0.6\textwidth]{figures/top_face_of_fence}
    \caption{Local picture of positive and negative strands near a horizontal (see top row) or a vertical (see bottom row) edge.}
    \label{fig:top_face}
    \end{figure}
\begin{lem}\label{lem:weight_on_initial_face}
    Let $f_0$ be the top face of the plabic fence $G_\beta$. Then
    $$U(e,f_0) = \begin{cases}
    1, \text{ if } e \text{ is horizontal and its white endpoint is on the right}\\
    0, \text{ otherwise}\\
    \end{cases}.$$
\end{lem}

\begin{proof}
    See Figure~\ref{fig:top_face}. The neighboring face on the top of the white endpoint gets assigned a weight of $1$ if and only if $e$ is a horizontal edge with white endpoint on its right. By Lemma~\ref{rem:monotonicity}, the strands $\gamma^+_e$ and $\gamma^{-}_e$ are monotone in the horizontal direction. Therefore, neither of $\gamma^+_e$ or $\gamma^{-}_e$ will visit the edge on the rows $1,2,...,i-1$ that is right on top of the middle point of $e$. Thus the top face has the same weight as the neighboring face on the top of the white endpoint of $e$.
\end{proof}

\begin{cor}
    For any edge $e$ and any face $f$ of a plabic fence $G_\beta$, we have $U(e,f)\in \{-1,0,1\}$.
\end{cor}
\begin{proof}
    Consider a straight vertical line path $\alpha$ from the face $f$ to the top face $f_0$. By monotonicity in Lemma~\ref{rem:monotonicity}, $\alpha$ crosses $\gamma^+_e$ at most once and crosses $\gamma^-_e$ at most once. Thus, $U(e,f_0) - U(e,f) \in \{1,0\}$ if $e$ is a horizontal edge with white endpoint on its right, and $U(e,f_0) - U(e,f) \in \{1,0,-1\}$ otherwise. 
\end{proof}

\begin{defn}
    Let $G_\beta$ be a plabic fence with edge weights $\{w(e)\}_{e\in E}$, where $w(e)\in \mathbb{C}^*$. The \textbf{generalized matching variable} associated to each face $f$ is defined as $$M_f:= \prod_{e\in E}(w(e))^{-U(e,f)}.$$
\end{defn}

\begin{prop}
    For any fixed internal vertex $v$ and any fixed face $f$ of the plabic fence $G_\beta$, we have $$\sum_{e\sim v}U(e,f) = 1.$$
    In particular, gauge transformation by $\lambda\in\mathbb{C}^*$ multiplies all the $M_f$'s simultaneously by $\lambda^{-1}$.
\end{prop}
\begin{figure}[h!]
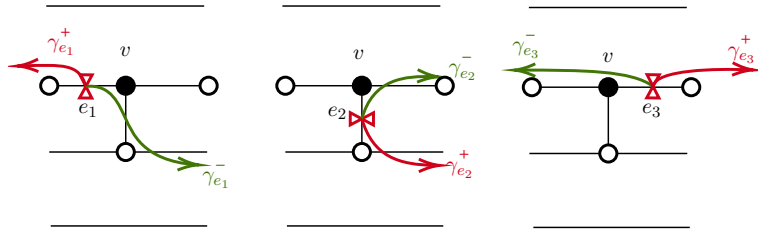

    \centering
    \includestandalone[mode=tex, width=0.6\textwidth]{figures/gauge_lemma_deg_3}
    \caption{Local picture of positive and negative strands on edges incident to a degree $3$ vertex $v$}
    \label{fig:gauge_deg_3}
    \end{figure}
\begin{proof}
    First, note that the conclusion is true for the top face $f = f_0$, because each internal vertex is incident to exactly one horizontal edge $e$ with white endpoint on its right.

    Now take $f$ to be any face, and $\alpha$ to be a vertical straight line path from $f_0$ to $f$, not going through vertical edges of the plabic fence. We argue $$\sum_{e\sim v}U(e,f) = \sum_{e\sim v}U(e,f_0).$$
    
    See Figure~\ref{fig:gauge_deg_3} for an example of the neighboring edges of a degree $3$ vertex. Label the edges incident to $v$ as $e_1,...,e_{d}$ counterclockwise around $v$, where $d = \text{deg}(v)\geq 2$. Then, the strands $\gamma^+_{e_i}$ and $\gamma^-_{e_{i-1}}$ overlap after both of them exit the edges $e_i,e_{i-1}$, respectively. Thus $\alpha$ crosses $\gamma^+_{e_i}$ if and only if it crosses $\gamma^-_{e_{i-1}}$ in the same way. Notice that $\gamma^+_{e_i}$ and $\gamma^-_{e_{i-1}}$ have opposite effects on face weights when we cross from the left side to the right. Thus, the sum $\sum_{e\sim v}U(e,f)$ remains the same when we cross from one side of $\gamma^+_{e_i}$ (and same side of $\gamma^-_{e_{i-1}}$) to the other. 
\end{proof}

\begin{lem}\label{lem: face_monomial_top_bottom}
    Let $G_\beta$ be a right-canonically weighted plabic fence. Suppose that the edge weights of the rightmost horizontal edges multiply to one. Let $f,f'$ denote the top and bottom faces of the fence, respectively. Then $M_f = M_{f'} = 1.$
\end{lem}
\begin{proof}
    Since all the rightmost horizontal edges have their white endpoints on the left, by the proof of Lemma~\ref{lem:weight_on_initial_face}, we have $U(e_i,f) = 0$, $U(e_i,f') = 1$ for $e_i$ the rightmost horizontal edge on the $i$-th row, and $U(e,f) = U(e,f') = 0$ for any vertical edge $e$. Thus, $M_f = \prod_{e\in E'}(w(e))^{0}=1$, where $E'$ denotes the edges with non-trivial weights. And $M_{f'} = w(e_1)^{-1}w(e_2)^{-1}\cdots w(e_{r+1})^{-1} = 1$. 
\end{proof}
\begin{exmp}\label{exmp: computing_gen_matching}
    Continuing from Example~\ref{exem: weighted_fence}, we now assume $d_1d_2d_3 = 1$. We carry out the trip construction for each weighted edge as in Figures~\ref{fig:face_weights_1} and \ref{fig:face_weights_2}. One should note that for most of the edges the strands still form a well-defined downstream region, with the only exception being the vertical edge with weight $p_3$, where the two strands make a bad double crossing at the edge $p_2$.
    \begin{figure}[h!]
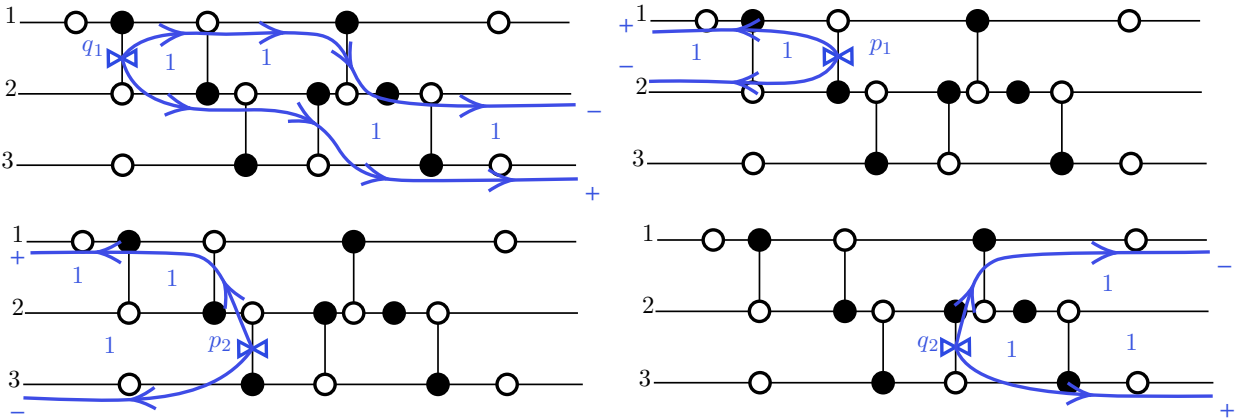

    \centering
    \includestandalone[mode=tex, width=1.0\textwidth]{figures/face_weights_1}
    \caption{Face weights for the edges $q_1,q_2,p_1,p_2$. Unlabeled faces have weight $0$}
    \label{fig:face_weights_1}
    \end{figure}
    
    \begin{figure}[h!]
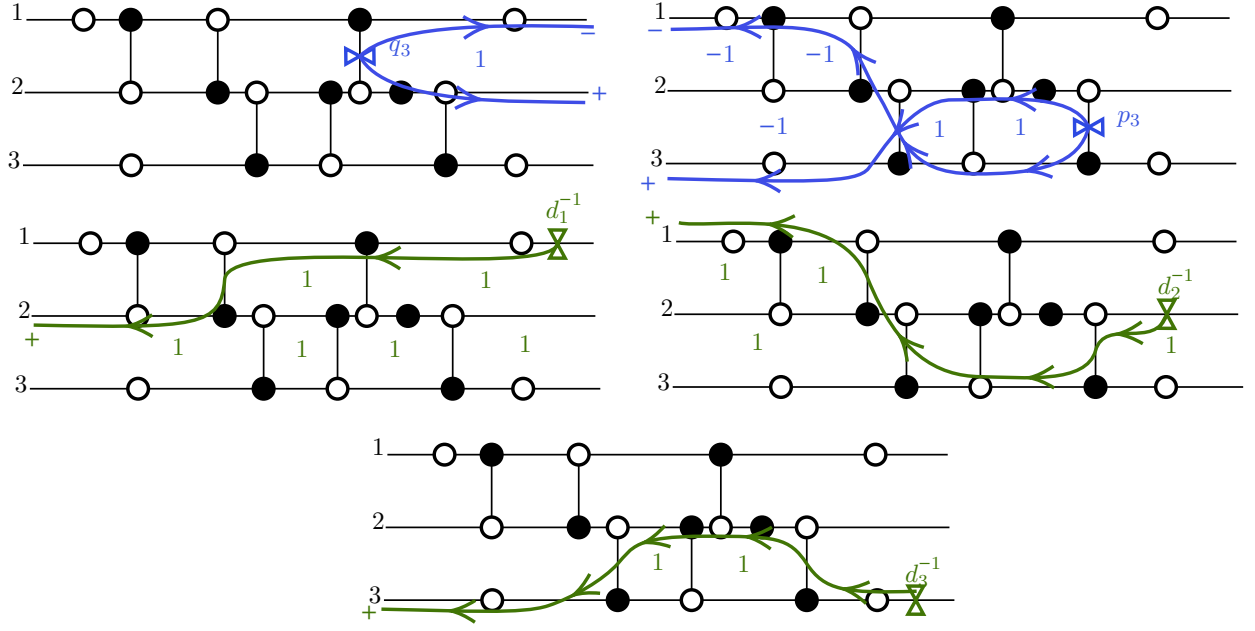

    \centering
    \includestandalone[mode=tex, width=1.0\textwidth]{figures/face_weights_2}
    \caption{Face weights for edges $q_3,p_3,d_1,d_2,d_3$. Unlabeled faces have weight $0$.}
    \label{fig:face_weights_2}
    \end{figure}

    Thus, we obtain the matrix $U = (U(e,f))_{e\in E', f\in F}$ as follows: 
    $$U = \m{0&1&1&0&0&0&1&1\\1&1&0&0&0&0&0&0\\1&1&0&0&1&0&0&0\\0&0&0&1&0&0&1&1\\0&0&0&1&0&0&0&0\\-1&-1&0&0&-1&1&1&0\\0&0&1&1&1&1&1&1\\1&1&0&0&1&0&0&1\\0&0&0&0&0&1&1&0\\}$$
    where the rows are indexed by edges $q_1,p_1,p_2,q_2,q_3,p_3,d_1,d_2,d_3$ from top to bottom, and the columns are indexed by the faces $\binom{1}{0},\binom{1}{1},\binom{1}{2},\binom{1}{3}, \binom{2}{0},\binom{2}{1},\binom{2}{2},\binom{2}{3}$ from left to right, where the notation $\binom{i}{a}$ means the $a$-th face indexed from left to right, on the row in between the $i$-th and $(i+1)$-th horizontal path. Thus, for example, for $f = \binom{1}{1}$, reading the 2nd column gives us $M_{\binom{1}{1}} = q^{-1}_1p^{-1}_1p^{-1}_2p_3d_2$. Taking each column to compute the corresponding generalized matching variable gives us the monomials in Figure~\ref{fig:gen_monomials}.
    
    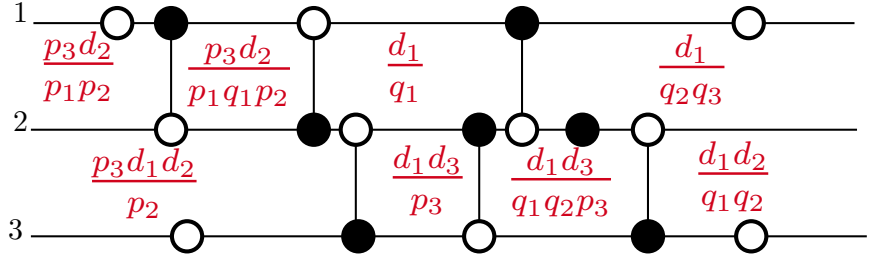
\begin{figure}[h!]
    \centering
    \tikzset{every picture/.style={line width=0.75pt}} 

\begin{tikzpicture}[x=0.75pt,y=0.75pt,yscale=-1,xscale=1]

\draw    (26.01,96.14) -- (100.65,96.14) -- (162.98,96.14) -- (264.7,96.14) -- (427.11,96.14) ;
\draw    (24.78,148.21) -- (100.65,148.21) -- (190.88,148.21) -- (243.79,148.21) -- (294.23,148.21) -- (326.22,148.21) -- (428.34,148.21) ;
\draw    (24.78,200.28) -- (192.11,200.28) -- (377.89,200.28) -- (433.12,200.28) ;
\draw  [fill={rgb, 255:red, 255; green, 255; blue, 255 }  ,fill opacity=1 ][line width=1.5]  (93.68,200.28) .. controls (93.68,196.2) and (96.99,192.9) .. (101.07,192.9) .. controls (105.14,192.9) and (108.45,196.2) .. (108.45,200.28) .. controls (108.45,204.36) and (105.14,207.66) .. (101.07,207.66) .. controls (96.99,207.66) and (93.68,204.36) .. (93.68,200.28) -- cycle ;
\draw  [fill={rgb, 255:red, 0; green, 0; blue, 0 }  ,fill opacity=1 ][line width=1.5]  (85.89,96.14) .. controls (85.89,92.07) and (89.19,88.76) .. (93.27,88.76) .. controls (97.35,88.76) and (100.65,92.07) .. (100.65,96.14) .. controls (100.65,100.22) and (97.35,103.52) .. (93.27,103.52) .. controls (89.19,103.52) and (85.89,100.22) .. (85.89,96.14) -- cycle ;
\draw    (93.27,96.14) -- (93.27,148.21) ;
\draw  [fill={rgb, 255:red, 255; green, 255; blue, 255 }  ,fill opacity=1 ][line width=1.5]  (85.89,148.21) .. controls (85.89,144.13) and (89.19,140.83) .. (93.27,140.83) .. controls (97.35,140.83) and (100.65,144.13) .. (100.65,148.21) .. controls (100.65,152.29) and (97.35,155.59) .. (93.27,155.59) .. controls (89.19,155.59) and (85.89,152.29) .. (85.89,148.21) -- cycle ;
\draw  [fill={rgb, 255:red, 0; green, 0; blue, 0 }  ,fill opacity=1 ][line width=1.5]  (155.6,148.21) .. controls (155.6,144.13) and (158.9,140.83) .. (162.98,140.83) .. controls (167.06,140.83) and (170.36,144.13) .. (170.36,148.21) .. controls (170.36,152.29) and (167.06,155.59) .. (162.98,155.59) .. controls (158.9,155.59) and (155.6,152.29) .. (155.6,148.21) -- cycle ;
\draw    (162.98,96.14) -- (162.98,148.21) ;
\draw  [fill={rgb, 255:red, 255; green, 255; blue, 255 }  ,fill opacity=1 ][line width=1.5]  (155.6,96.14) .. controls (155.6,92.07) and (158.9,88.76) .. (162.98,88.76) .. controls (167.06,88.76) and (170.36,92.07) .. (170.36,96.14) .. controls (170.36,100.22) and (167.06,103.52) .. (162.98,103.52) .. controls (158.9,103.52) and (155.6,100.22) .. (155.6,96.14) -- cycle ;
\draw  [fill={rgb, 255:red, 0; green, 0; blue, 0 }  ,fill opacity=1 ][line width=1.5]  (177.35,200.28) .. controls (177.35,196.2) and (180.65,192.9) .. (184.73,192.9) .. controls (188.81,192.9) and (192.11,196.2) .. (192.11,200.28) .. controls (192.11,204.36) and (188.81,207.66) .. (184.73,207.66) .. controls (180.65,207.66) and (177.35,204.36) .. (177.35,200.28) -- cycle ;
\draw    (183.5,148.21) -- (183.5,200.28) ;
\draw  [fill={rgb, 255:red, 255; green, 255; blue, 255 }  ,fill opacity=1 ][line width=1.5]  (176.12,148.21) .. controls (176.12,144.13) and (179.42,140.83) .. (183.5,140.83) .. controls (187.58,140.83) and (190.88,144.13) .. (190.88,148.21) .. controls (190.88,152.29) and (187.58,155.59) .. (183.5,155.59) .. controls (179.42,155.59) and (176.12,152.29) .. (176.12,148.21) -- cycle ;
\draw  [fill={rgb, 255:red, 0; green, 0; blue, 0 }  ,fill opacity=1 ][line width=1.5]  (236.4,148.21) .. controls (236.4,144.13) and (239.71,140.83) .. (243.79,140.83) .. controls (247.86,140.83) and (251.17,144.13) .. (251.17,148.21) .. controls (251.17,152.29) and (247.86,155.59) .. (243.79,155.59) .. controls (239.71,155.59) and (236.4,152.29) .. (236.4,148.21) -- cycle ;
\draw    (243.79,148.21) -- (243.79,200.28) ;
\draw  [fill={rgb, 255:red, 255; green, 255; blue, 255 }  ,fill opacity=1 ][line width=1.5]  (236.4,200.28) .. controls (236.4,196.2) and (239.71,192.9) .. (243.79,192.9) .. controls (247.86,192.9) and (251.17,196.2) .. (251.17,200.28) .. controls (251.17,204.36) and (247.86,207.66) .. (243.79,207.66) .. controls (239.71,207.66) and (236.4,204.36) .. (236.4,200.28) -- cycle ;
\draw  [fill={rgb, 255:red, 0; green, 0; blue, 0 }  ,fill opacity=1 ][line width=1.5]  (257.32,96.14) .. controls (257.32,92.07) and (260.62,88.76) .. (264.7,88.76) .. controls (268.78,88.76) and (272.08,92.07) .. (272.08,96.14) .. controls (272.08,100.22) and (268.78,103.52) .. (264.7,103.52) .. controls (260.62,103.52) and (257.32,100.22) .. (257.32,96.14) -- cycle ;
\draw    (264.7,96.14) -- (264.7,148.21) ;
\draw  [fill={rgb, 255:red, 255; green, 255; blue, 255 }  ,fill opacity=1 ][line width=1.5]  (257.32,148.21) .. controls (257.32,144.13) and (260.62,140.83) .. (264.7,140.83) .. controls (268.78,140.83) and (272.08,144.13) .. (272.08,148.21) .. controls (272.08,152.29) and (268.78,155.59) .. (264.7,155.59) .. controls (260.62,155.59) and (257.32,152.29) .. (257.32,148.21) -- cycle ;
\draw  [fill={rgb, 255:red, 0; green, 0; blue, 0 }  ,fill opacity=1 ][line width=1.5]  (318.84,200.28) .. controls (318.84,196.2) and (322.14,192.9) .. (326.22,192.9) .. controls (330.3,192.9) and (333.6,196.2) .. (333.6,200.28) .. controls (333.6,204.36) and (330.3,207.66) .. (326.22,207.66) .. controls (322.14,207.66) and (318.84,204.36) .. (318.84,200.28) -- cycle ;
\draw    (326.22,148.21) -- (326.22,200.28) ;
\draw  [fill={rgb, 255:red, 255; green, 255; blue, 255 }  ,fill opacity=1 ][line width=1.5]  (318.84,148.21) .. controls (318.84,144.13) and (322.14,140.83) .. (326.22,140.83) .. controls (330.3,140.83) and (333.6,144.13) .. (333.6,148.21) .. controls (333.6,152.29) and (330.3,155.59) .. (326.22,155.59) .. controls (322.14,155.59) and (318.84,152.29) .. (318.84,148.21) -- cycle ;
\draw  [fill={rgb, 255:red, 0; green, 0; blue, 0 }  ,fill opacity=1 ][line width=1.5]  (286.85,148.21) .. controls (286.85,144.13) and (290.15,140.83) .. (294.23,140.83) .. controls (298.31,140.83) and (301.61,144.13) .. (301.61,148.21) .. controls (301.61,152.29) and (298.31,155.59) .. (294.23,155.59) .. controls (290.15,155.59) and (286.85,152.29) .. (286.85,148.21) -- cycle ;
\draw  [fill={rgb, 255:red, 255; green, 255; blue, 255 }  ,fill opacity=1 ][line width=1.5]  (59.56,96.14) .. controls (59.56,92.07) and (62.86,88.76) .. (66.94,88.76) .. controls (71.02,88.76) and (74.32,92.07) .. (74.32,96.14) .. controls (74.32,100.22) and (71.02,103.52) .. (66.94,103.52) .. controls (62.86,103.52) and (59.56,100.22) .. (59.56,96.14) -- cycle ;
\draw  [fill={rgb, 255:red, 255; green, 255; blue, 255 }  ,fill opacity=1 ][line width=1.5]  (368.05,96.14) .. controls (368.05,92.07) and (371.36,88.76) .. (375.43,88.76) .. controls (379.51,88.76) and (382.82,92.07) .. (382.82,96.14) .. controls (382.82,100.22) and (379.51,103.52) .. (375.43,103.52) .. controls (371.36,103.52) and (368.05,100.22) .. (368.05,96.14) -- cycle ;
\draw  [fill={rgb, 255:red, 255; green, 255; blue, 255 }  ,fill opacity=1 ][line width=1.5]  (369.28,200.28) .. controls (369.28,196.2) and (372.59,192.9) .. (376.66,192.9) .. controls (380.74,192.9) and (384.05,196.2) .. (384.05,200.28) .. controls (384.05,204.36) and (380.74,207.66) .. (376.66,207.66) .. controls (372.59,207.66) and (369.28,204.36) .. (369.28,200.28) -- cycle ;

\draw (26.47,97.63) node [anchor=north west][inner sep=0.75pt,scale=1.5]    {$\textcolor[rgb]{0.82,0.01,0.11}{\frac{p_{3} d_{2}}{p_{1} p_{2}}}$};
\draw (97,100) node [anchor=north west][inner sep=0.75pt, scale=1.5]    {$\textcolor[rgb]{0.82,0.01,0.11}{\frac{p_{3} d_{2}}{p_{1} q_{1} p_{2}}}$};
\draw (14.68,84.85) node [anchor=north west][inner sep=0.75pt]    {$1$};
\draw (14.68,137.76) node [anchor=north west][inner sep=0.75pt]    {$2$};
\draw (12.22,190.66) node [anchor=north west][inner sep=0.75pt]    {$3$};
\draw (194.47,98.55) node [anchor=north west][inner sep=0.75pt,scale=1.5]    {$\textcolor[rgb]{0.82,0.01,0.11}{\frac{d_{1}}{q_{1}}}$};
\draw (327.09,99.48) node [anchor=north west][inner sep=0.75pt,scale=1.5]    {$\textcolor[rgb]{0.82,0.01,0.11}{\frac{d_{1}}{q_{2} q_{3}}}$};
\draw (50,155) node [anchor=north west][inner sep=0.75pt,scale=1.5]    {$\textcolor[rgb]{0.82,0.01,0.11}{\frac{p_{3} d_{1} d_{2}}{p_{2}}}$};
\draw (196.89,153.39) node [anchor=north west][inner sep=0.75pt, scale=1.5]    {$\textcolor[rgb]{0.82,0.01,0.11}{\frac{d_{1} d_{3}}{p_{3}}}$};
\draw (254.78,154.32) node [anchor=north west][inner sep=0.75pt, scale=1.5]    {$\textcolor[rgb]{0.82,0.01,0.11}{\frac{d_{1} d_{3}}{q_{1} q_{2} p_{3}}}$};
\draw (346.17,151.54) node [anchor=north west][inner sep=0.75pt, scale=1.5]    {$\textcolor[rgb]{0.82,0.01,0.11}{\frac{d_{1} d_{2}}{q_{1} q_{2}}}$};

\end{tikzpicture}
    \caption{Generalized matching variables associated to faces.}
    \label{fig:gen_monomials}
    \end{figure}
\end{exmp}

We now state our main theorem. 

\begin{thm}\label{thm:matching_equals_minor}
Given a double braid word $\beta$, consider the right-canonically weighted associated plabic fence $G_\beta$. If we require that the rightmost horizontal-edge weights multiply to one, then these weights parameterize a torus in $\text{Conf}(\beta_+,\beta_-)$. Moreover, the generalized matching variables for the faces of $G_\beta$ recover the corresponding generalized minors in the torus $T(C_\beta)$.
\end{thm}

\begin{exmp}
    Following Example~\ref{exmp: computing_decorations}, we compute the generalized minors corresponding to each diagonal of the triangulation, and compare with the generalized matchings we computed in Example~\ref{exmp: computing_gen_matching} as follows:
    $$\begin{cases}
        A_{\binom{1}{0}} = \Delta_1(t_0) = d_2p_3p^{-1}_2p^{-1}_1 = M_{\binom{1}{0}}\\
        A_{\binom{2}{0}} = \Delta_2(t_0) = d_1d_2p_3p^{-1}_2 = M_{\binom{2}{0}}\\
        A_{\binom{1}{1}} = \Delta_1(t'_1e_1(-q_1)t_0) = \Delta_1(t'_1t_0e_1(-t^{-\alpha_1}_0q_1)) = \Delta_1(t'_1t_0) = d_2p_3p^{-1}_2p^{-1}_1q^{-1}_1 = M_{\binom{1}{1}}\\
        A_{\binom{1}{2}} = \Delta_1(t'_1e_{-1}(p_1)t_1) = \Delta_1(t'_1t_1e_{-1}(t^{\alpha_1}_1p_1)) = \Delta_1(t'_1t_1) = d_1q^{-1}_1 = M_{\binom{1}{2}}\\
        A_{\binom{2}{1}} = \Delta_2(t'_1e_{-1}(p_1)e_{-2}(p_2)t_2) = \Delta_2(t'_1t_2) = d_1d_3p^{-1}_3 = M_{\binom{2}{1}}\\
        A_{\binom{2}{2}} = \Delta_2(t'_2e_{2}(-q_2)t_2) = \Delta_2(t'_2t_2) = d_1d_3p^{-1}_3q^{-1}_1q^{-1}_2 = M_{\binom{2}{2}}\\
        A_{\binom{1}{3}} = \Delta_1(t'_3e_{1}(-q_3)e_{2}(-q_2)t_2) = \Delta_1(t'_3t_2) = d_1q^{-1}_2q^{-1}_3 = M_{\binom{1}{3}}\\
        A_{\binom{2}{3}} = \Delta_2(t'_3e_{-1}(p_3)d) = \Delta_2(t'_3d) = d_1d_2q^{-1}_1q^{-1}_2 = M_{\binom{2}{3}}\\
    \end{cases}$$
    This shows that they are equal.
\end{exmp}
Before the proof of the theorem, we introduce the following useful technical lemma. 

\begin{lem}\label{lem:gluing_fences}
Consider two right-canonically weighted $r$-row plabic fences $G_1$ and $G_2$, where the rightmost horizontal-edge weights of each graph multiply to one. Suppose that for each $i\in [r]$, the generalized matching variable on rightmost face of $G_1$ on the $i$-th row is equal to that on the leftmost face of $G_2$ on the $i$-th row. Let $G$ be the $r$-row plabic fence obtained by gluing the right end of $G_1$ to the left end of $G_2$. Then, there is a choice of edge weights on $G$, such that the generalized matching variables on the faces of $G$ on the left subgraph are equal to those on $G_1$, and similarly, the generalized matching variables on the faces of $G$ on the right subgraph are equal to those on $G_2$.
\end{lem}
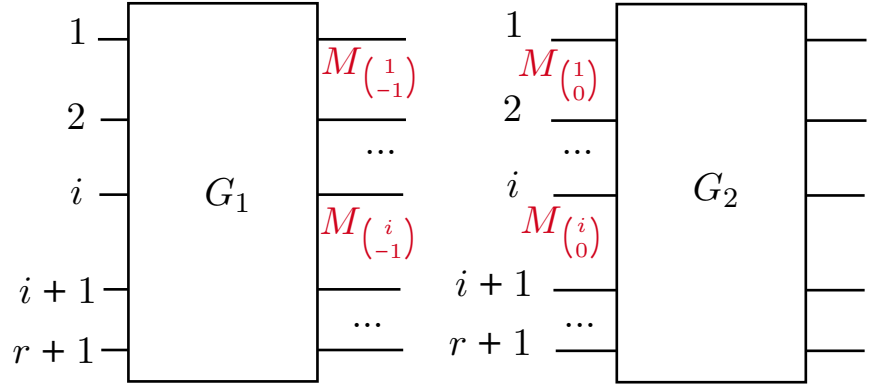
\begin{figure}[h!]
    \centering
    \tikzset{every picture/.style={line width=0.75pt}} 

\begin{tikzpicture}[x=0.75pt,y=0.75pt,yscale=-1,xscale=1]

\draw    (226.21,35.24) -- (344.88,35.24) ;
\draw    (226.21,65.53) -- (344.88,65.53) ;
\draw    (227.07,93.7) -- (326.33,93.7) -- (344.18,93.7) ;
\draw    (227.07,129.34) -- (325.92,129.34) -- (344.18,129.34) ;
\draw    (227.93,152.16) -- (344.18,152.16) ;
\draw    (55.65,34.94) -- (171.57,34.94) ;
\draw    (56.67,65.22) -- (171.57,65.22) ;
\draw    (56.2,93.39) -- (131.47,93.39) -- (170.53,93.39) ;
\draw    (57.76,129.03) -- (131.92,129.03) -- (169.5,129.03) ;
\draw    (56.91,151.85) -- (170.53,151.85) ;
\draw  [fill={rgb, 255:red, 255; green, 255; blue, 255 }  ,fill opacity=1 ] (67.12,21.39) -- (138.26,21.39) -- (138.26,163.67) -- (67.12,163.67) -- cycle ;
\draw  [fill={rgb, 255:red, 255; green, 255; blue, 255 }  ,fill opacity=1 ] (250.91,21.7) -- (322.04,21.7) -- (322.04,163.97) -- (250.91,163.97) -- cycle ;

\draw (42.99,25.93) node [anchor=north west][inner sep=0.75pt]    {$1$};
\draw (42.28,57.62) node [anchor=north west][inner sep=0.75pt]    {$2$};
\draw (43.69,85.79) node [anchor=north west][inner sep=0.75pt]    {$i$};
\draw (24.39,123.12) node [anchor=north west][inner sep=0.75pt]    {$i+1$};
\draw (21.98,145.66) node [anchor=north west][inner sep=0.75pt]    {$r+1$};
\draw (94.23,85.83) node [anchor=north west][inner sep=0.75pt]    {$G_{1}$};
\draw (228.55,75) node [anchor=north west][inner sep=0.75pt]    {$...$};
\draw (229.41,140) node [anchor=north west][inner sep=0.75pt]    {$...$};
\draw (149.52,140) node [anchor=north west][inner sep=0.75pt]    {$...$};
\draw (154.52,75) node [anchor=north west][inner sep=0.75pt]    {$...$};
\draw (278.02,84.07) node [anchor=north west][inner sep=0.75pt]    {$G_{2}$};
\draw (207.18,23.14) node [anchor=north west][inner sep=0.75pt]    {$1$};
\draw (206.47,54.83) node [anchor=north west][inner sep=0.75pt]    {$2$};
\draw (207.88,83.01) node [anchor=north west][inner sep=0.75pt]    {$i$};
\draw (188.58,120.33) node [anchor=north west][inner sep=0.75pt]    {$i+1$};
\draw (186.17,142.87) node [anchor=north west][inner sep=0.75pt]    {$r+1$};
\draw (212,37.4) node [anchor=north west][inner sep=0.75pt]    {$\textcolor[rgb]{0.82,0.01,0.11}{M}\textcolor[rgb]{0.82,0.01,0.11}{_{\binom{1}{0}}}$};
\draw (213,96.4) node [anchor=north west][inner sep=0.75pt]    {$\textcolor[rgb]{0.82,0.01,0.11}{M}\textcolor[rgb]{0.82,0.01,0.11}{_{\binom{i}{0}}}$};
\draw (138,36.4) node [anchor=north west][inner sep=0.75pt]    {$\textcolor[rgb]{0.82,0.01,0.11}{M}\textcolor[rgb]{0.82,0.01,0.11}{_{\binom{1}{-1}}}$};
\draw (137.47,96.79) node [anchor=north west][inner sep=0.75pt]    {$\textcolor[rgb]{0.82,0.01,0.11}{M}\textcolor[rgb]{0.82,0.01,0.11}{_{\binom{i}{-1}}}$};

\end{tikzpicture}
    \caption{Two $r$-row fences. For Lemma~\ref{lem:gluing_fences}, we need $M_{\binom{i}{-1}}=M_{\binom{i}{0}}$ for every $i$.}
    \label{fig:two_fences_with_prescribed_weights}
\end{figure}
\begin{figure}[h!]
    \centering
    \tikzset{every picture/.style={line width=0.75pt}} 

\begin{tikzpicture}[x=0.75pt,y=0.75pt,yscale=-1,xscale=1]

\draw    (185.47,76.26) -- (394,76.26) ;
\draw    (185.33,107.59) -- (393,107.59) ;
\draw    (184.82,134.65) -- (341.49,134.65) -- (392,134.65) ;
\draw    (184.82,170.25) -- (341.07,170.25) -- (391,170.25) ;
\draw    (180.4,193.9) -- (392,193.9) ;
\draw    (69.69,76.51) -- (185.47,76.51) ;
\draw    (70.71,106.76) -- (185.47,106.76) ;
\draw    (70.24,134.9) -- (145.43,134.9) -- (184.44,134.9) ;
\draw    (71.8,170.49) -- (145.87,170.49) -- (183.41,170.49) ;
\draw    (70.94,193.29) -- (184.44,193.29) ;
\draw  [fill={rgb, 255:red, 255; green, 255; blue, 255 }  ,fill opacity=1 ] (81.15,62.98) -- (152.2,62.98) -- (152.2,205.09) -- (81.15,205.09) -- cycle ;
\draw  [fill={rgb, 255:red, 255; green, 255; blue, 255 }  ,fill opacity=1 ] (283.96,61.6) -- (355.01,61.6) -- (355.01,203.71) -- (283.96,203.71) -- cycle ;
\draw  [fill={rgb, 255:red, 0; green, 0; blue, 0 }  ,fill opacity=1 ][line width=1.5]  (181.49,171.16) .. controls (181.49,169.33) and (182.98,167.84) .. (184.82,167.84) .. controls (186.65,167.84) and (188.14,169.33) .. (188.14,171.16) .. controls (188.14,173) and (186.65,174.48) .. (184.82,174.48) .. controls (182.98,174.48) and (181.49,173) .. (181.49,171.16) -- cycle ;
\draw  [fill={rgb, 255:red, 0; green, 0; blue, 0 }  ,fill opacity=1 ][line width=1.5]  (180.4,193.9) .. controls (180.4,192.06) and (181.88,190.57) .. (183.72,190.57) .. controls (185.55,190.57) and (187.04,192.06) .. (187.04,193.9) .. controls (187.04,195.73) and (185.55,197.22) .. (183.72,197.22) .. controls (181.88,197.22) and (180.4,195.73) .. (180.4,193.9) -- cycle ;
\draw  [fill={rgb, 255:red, 0; green, 0; blue, 0 }  ,fill opacity=1 ][line width=1.5]  (181.49,135.66) .. controls (181.49,133.82) and (182.98,132.34) .. (184.82,132.34) .. controls (186.65,132.34) and (188.14,133.82) .. (188.14,135.66) .. controls (188.14,137.49) and (186.65,138.98) .. (184.82,138.98) .. controls (182.98,138.98) and (181.49,137.49) .. (181.49,135.66) -- cycle ;
\draw  [fill={rgb, 255:red, 0; green, 0; blue, 0 }  ,fill opacity=1 ][line width=1.5]  (182.01,107.59) .. controls (182.01,105.76) and (183.5,104.27) .. (185.33,104.27) .. controls (187.17,104.27) and (188.65,105.76) .. (188.65,107.59) .. controls (188.65,109.43) and (187.17,110.91) .. (185.33,110.91) .. controls (183.5,110.91) and (182.01,109.43) .. (182.01,107.59) -- cycle ;
\draw  [fill={rgb, 255:red, 0; green, 0; blue, 0 }  ,fill opacity=1 ][line width=1.5]  (182.01,77.42) .. controls (182.01,75.59) and (183.5,74.1) .. (185.33,74.1) .. controls (187.17,74.1) and (188.65,75.59) .. (188.65,77.42) .. controls (188.65,79.26) and (187.17,80.74) .. (185.33,80.74) .. controls (183.5,80.74) and (182.01,79.26) .. (182.01,77.42) -- cycle ;
\draw  [dash pattern={on 4.5pt off 4.5pt}]  (184.99,40.16) -- (184.99,238.44) ;
\draw  [color={rgb, 255:red, 245; green, 166; blue, 35 }  ,draw opacity=1 ][line width=1.5]  (197.42,68.69) -- (205.25,68.69) -- (197.42,81.31) -- (205.25,81.31) -- cycle ;
\draw  [color={rgb, 255:red, 245; green, 166; blue, 35 }  ,draw opacity=1 ][line width=1.5]  (196.42,100.69) -- (204.25,100.69) -- (196.42,113.31) -- (204.25,113.31) -- cycle ;
\draw  [color={rgb, 255:red, 245; green, 166; blue, 35 }  ,draw opacity=1 ][line width=1.5]  (195.42,186.69) -- (203.25,186.69) -- (195.42,199.31) -- (203.25,199.31) -- cycle ;
\draw  [color={rgb, 255:red, 245; green, 166; blue, 35 }  ,draw opacity=1 ][line width=1.5]  (195.42,164.69) -- (203.25,164.69) -- (195.42,177.31) -- (203.25,177.31) -- cycle ;

\draw (53.22,64.95) node [anchor=north west][inner sep=0.75pt]    {$1$};
\draw (56.33,99.16) node [anchor=north west][inner sep=0.75pt]    {$2$};
\draw (57.74,127.3) node [anchor=north west][inner sep=0.75pt]    {$i$};
\draw (38.45,164.58) node [anchor=north west][inner sep=0.75pt]    {$i+1$};
\draw (36.04,187.09) node [anchor=north west][inner sep=0.75pt]    {$r+1$};
\draw (108.21,127.33) node [anchor=north west][inner sep=0.75pt]    {$G_{1}$};
\draw (248.48,178) node [anchor=north west][inner sep=0.75pt]    {$...$};
\draw (159.62,178) node [anchor=north west][inner sep=0.75pt]    {$...$};
\draw (160.8,117) node [anchor=north west][inner sep=0.75pt]    {$...$};
\draw (309.21,126.02) node [anchor=north west][inner sep=0.75pt]    {$G_{2}$};
\draw (247.32,117) node [anchor=north west][inner sep=0.75pt]    {$...$};
\draw (192,43.4) node [anchor=north west][inner sep=0.75pt]    {$\textcolor[rgb]{0.96,0.65,0.14}{M_{\binom{1}{0}}^{-1} *}$};
\draw (190.65,79) node [anchor=north west][inner sep=0.75pt]    {$\textcolor[rgb]{0.96,0.65,0.14}{M}\textcolor[rgb]{0.96,0.65,0.14}{_{\binom{1}{0}}}\textcolor[rgb]{0.96,0.65,0.14}{M}\textcolor[rgb]{0.96,0.65,0.14}{_{\binom{2}{0}}^{^{-1}} *}$};
\draw (186.82,138.05) node [anchor=north west][inner sep=0.75pt]    {$\textcolor[rgb]{0.96,0.65,0.14}{M}\textcolor[rgb]{0.96,0.65,0.14}{_{\binom{i}{0}}}\textcolor[rgb]{0.96,0.65,0.14}{M}\textcolor[rgb]{0.96,0.65,0.14}{_{\binom{i+1}{0}}^{^{-1}} *}$};
\draw (189.04,199) node [anchor=north west][inner sep=0.75pt]    {$\textcolor[rgb]{0.96,0.65,0.14}{M}\textcolor[rgb]{0.96,0.65,0.14}{_{\binom{r}{0}}}\textcolor[rgb]{0.96,0.65,0.14}{*}$};

\end{tikzpicture}
    \caption{Gluing of $G_1$ and $G_2$.}
    \label{fig:gluing_fences}
\end{figure}
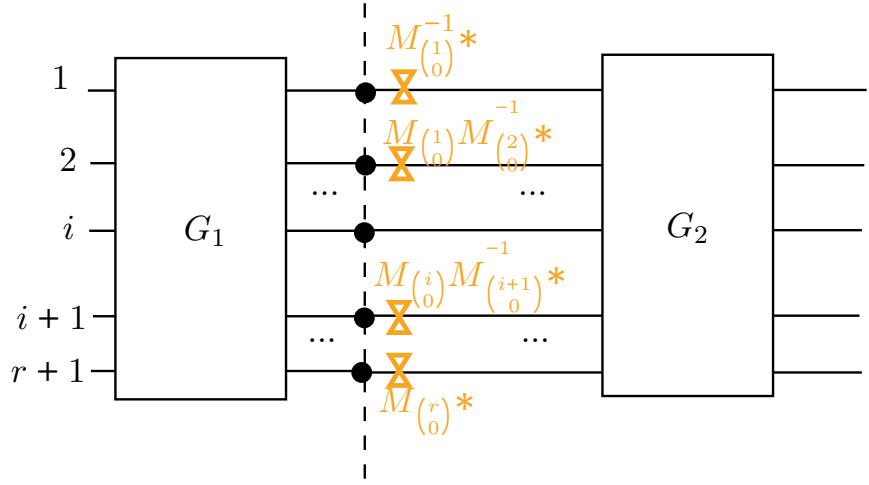
\begin{proof}
    Denote the generalized matching variables on the right end faces of $G_1$ as $M_{\binom{i}{-1}}$, and those on the left-end faces of $G_2$ as $M_{\binom{i}{0}}$ (see Figure~\ref{fig:two_fences_with_prescribed_weights}). Then, we have $M_{\binom{i}{-1}} = M_{\binom{i}{0}}, \forall i\in [r]$. 
    
    We construct weights on the glued graph $G$ as follows. For edges on the $G_1$-side of $G$, keep their original weights from $G_1$, and similarly, for edges on the $G_2$-side of $G$, keep the original weights from $G_2$. For the leftmost horizontal edge of $G_2$ on the $i$-th level, multiply the original weight by $M_{\binom{i-1}{0}}M^{-1}_{\binom{i}{0}}$, where $M_{\binom{-1}{0}}:= 1, M_{\binom{r+1}{0}}:= 1$. Notice that this construction ensures that $G$ is gauge equivalent to a right canonical gauge, where the rightmost edge weights multiply to one. We claim this is the desired weighting.

    First, we demonstrate how to compute the contribution to face weights on $G_2$ from edges of $G_1$. Note that by construction of the generalized matching, when we cross from the left to the right side of the positive strand of an edge $e\in E(G_1)$, the face weights change by $w(e)$, and similarly, when we cross from the left to the right of the negative strand of $e$, face weights change by $w(e)^{-1}$. By Lemma~\ref{lem: face_monomial_top_bottom}, the bottom face of $G$ has weight $1$. Thus, the change of face weights when we cross all the strands traveling rightwards along the $r+1$-th level to go from the bottom face to the right end face on the $r$-th row is equal to $M_{\binom{r}{-1}}$. Similarly, when we cross all the strands traveling rightwards along the $i$-th level to go from the right end face on the $i$-th row to that on the $(i-1)$-th row, face weights change by $M^{-1}_{\binom{i}{-1}}M_{\binom{i-1}{-1}}$ (see Figure~\ref{fig:gluing_fences}).

    Note that the strands coming from $G_1$ that are traveling rightward along the $i$-th level of $G$ will all go along the same direction as the positive strand of the left-most edge $e_{L,i}$ on the $i$-th level of $G_2$. Thus, when we cross all these strands together from the right to the left, face weights will change by $(M^{-1}_{\binom{i}{-1}}M_{\binom{i-1}{-1}})\cdot(M_{\binom{i-1}{0}}M^{-1}_{\binom{i}{0}}w(e_{L,i}))^{-1} = w(e_{L,i})^{-1}$, which is the same change as when $G_1$ was not present. Therefore, the face weights of $G$ on the $G_2$-side are all the same as the face weights of $G_2$.

    To compute the contribution of $G_2$ on the $G_1$-side of $G$, we can first do a gauge transformation on the black vertices of the gluing interface (which does not change the face weights of $G$) to get the following in Figure~\ref{fig:shifted_weights}, and apply similar reasoning to show that the face weights of $G$ on the $G_1$-side are all the same as the face weights of $G_1$.
    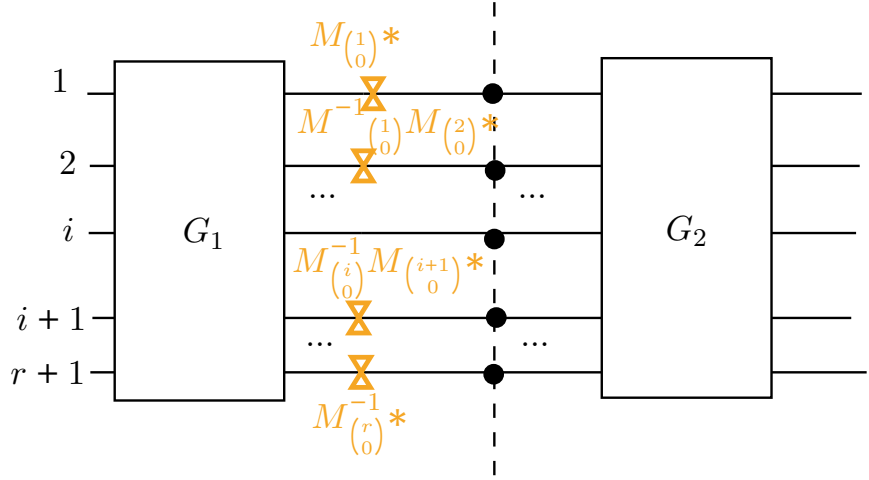
\begin{figure}[h!]
    \centering
    \tikzset{every picture/.style={line width=0.75pt}} 

\begin{tikzpicture}[x=0.75pt,y=0.75pt,yscale=-1,xscale=1]

\draw    (184.25,76.48) -- (392.78,76.48) ;
\draw    (184.25,106.69) -- (391.92,106.69) ;
\draw    (183.22,134.79) -- (339.89,134.79) -- (390.41,134.79) ;
\draw    (182.19,170.33) -- (338.45,170.33) -- (388.38,170.33) ;
\draw    (183.22,193.09) -- (394.83,193.09) ;
\draw    (68.64,76.48) -- (184.25,76.48) ;
\draw    (69.66,106.69) -- (184.25,106.69) ;
\draw    (69.19,134.79) -- (144.26,134.79) -- (183.22,134.79) ;
\draw    (70.75,170.33) -- (144.71,170.33) -- (182.19,170.33) ;
\draw    (69.89,193.09) -- (183.22,193.09) ;
\draw  [fill={rgb, 255:red, 255; green, 255; blue, 255 }  ,fill opacity=1 ] (80.08,62.98) -- (151.03,62.98) -- (151.03,204.88) -- (80.08,204.88) -- cycle ;
\draw  [fill={rgb, 255:red, 255; green, 255; blue, 255 }  ,fill opacity=1 ] (283.96,61.6) -- (355.01,61.6) -- (355.01,203.71) -- (283.96,203.71) -- cycle ;
\draw  [fill={rgb, 255:red, 0; green, 0; blue, 0 }  ,fill opacity=1 ][line width=1.5]  (236.49,170.16) .. controls (236.49,168.33) and (237.98,166.84) .. (239.82,166.84) .. controls (241.65,166.84) and (243.14,168.33) .. (243.14,170.16) .. controls (243.14,172) and (241.65,173.48) .. (239.82,173.48) .. controls (237.98,173.48) and (236.49,172) .. (236.49,170.16) -- cycle ;
\draw  [fill={rgb, 255:red, 0; green, 0; blue, 0 }  ,fill opacity=1 ][line width=1.5]  (235.4,193.9) .. controls (235.4,192.06) and (236.88,190.57) .. (238.72,190.57) .. controls (240.55,190.57) and (242.04,192.06) .. (242.04,193.9) .. controls (242.04,195.73) and (240.55,197.22) .. (238.72,197.22) .. controls (236.88,197.22) and (235.4,195.73) .. (235.4,193.9) -- cycle ;
\draw  [fill={rgb, 255:red, 0; green, 0; blue, 0 }  ,fill opacity=1 ][line width=1.5]  (235.67,137.3) .. controls (235.67,135.47) and (237.16,133.98) .. (238.99,133.98) .. controls (240.83,133.98) and (242.31,135.47) .. (242.31,137.3) .. controls (242.31,139.14) and (240.83,140.62) .. (238.99,140.62) .. controls (237.16,140.62) and (235.67,139.14) .. (235.67,137.3) -- cycle ;
\draw  [fill={rgb, 255:red, 0; green, 0; blue, 0 }  ,fill opacity=1 ][line width=1.5]  (236.01,108.59) .. controls (236.01,106.76) and (237.5,105.27) .. (239.33,105.27) .. controls (241.17,105.27) and (242.65,106.76) .. (242.65,108.59) .. controls (242.65,110.43) and (241.17,111.91) .. (239.33,111.91) .. controls (237.5,111.91) and (236.01,110.43) .. (236.01,108.59) -- cycle ;
\draw  [fill={rgb, 255:red, 0; green, 0; blue, 0 }  ,fill opacity=1 ][line width=1.5]  (235.01,76.42) .. controls (235.01,74.59) and (236.5,73.1) .. (238.33,73.1) .. controls (240.17,73.1) and (241.65,74.59) .. (241.65,76.42) .. controls (241.65,78.26) and (240.17,79.74) .. (238.33,79.74) .. controls (236.5,79.74) and (235.01,78.26) .. (235.01,76.42) -- cycle ;
\draw  [dash pattern={on 4.5pt off 4.5pt}]  (238.99,38.16) -- (238.99,236.44) ;
\draw  [color={rgb, 255:red, 245; green, 166; blue, 35 }  ,draw opacity=1 ][line width=1.5]  (184.34,70.18) -- (192.16,70.18) -- (184.34,82.79) -- (192.16,82.79) -- cycle ;
\draw  [color={rgb, 255:red, 245; green, 166; blue, 35 }  ,draw opacity=1 ][line width=1.5]  (180.34,100.38) -- (188.16,100.38) -- (180.34,113) -- (188.16,113) -- cycle ;
\draw  [color={rgb, 255:red, 245; green, 166; blue, 35 }  ,draw opacity=1 ][line width=1.5]  (178.28,164.03) -- (186.11,164.03) -- (178.28,176.64) -- (186.11,176.64) -- cycle ;
\draw  [color={rgb, 255:red, 245; green, 166; blue, 35 }  ,draw opacity=1 ][line width=1.5]  (179.31,186.79) -- (187.14,186.79) -- (179.31,199.4) -- (187.14,199.4) -- cycle ;

\draw (52.18,64.94) node [anchor=north west][inner sep=0.75pt]    {$1$};
\draw (55.29,99.09) node [anchor=north west][inner sep=0.75pt]    {$2$};
\draw (56.7,127.19) node [anchor=north west][inner sep=0.75pt]    {$i$};
\draw (37.42,164.42) node [anchor=north west][inner sep=0.75pt]    {$i+1$};
\draw (35.01,186.9) node [anchor=north west][inner sep=0.75pt]    {$r+1$};
\draw (107.09,127.22) node [anchor=north west][inner sep=0.75pt]    {$G_{1}$};
\draw (248.48,179) node [anchor=north west][inner sep=0.75pt]    {$...$};
\draw (158.43,179) node [anchor=north west][inner sep=0.75pt]    {$...$};
\draw (159.6,117) node [anchor=north west][inner sep=0.75pt]    {$...$};
\draw (309.21,126.02) node [anchor=north west][inner sep=0.75pt]    {$G_{2}$};
\draw (247.32,117) node [anchor=north west][inner sep=0.75pt]    {$...$};
\draw (160,45) node [anchor=north west][inner sep=0.75pt]    {$\textcolor[rgb]{0.96,0.65,0.14}{M_{\binom{1}{0}} *}$};
\draw (154.65,76.82) node [anchor=north west][inner sep=0.75pt]    {$\textcolor[rgb]{0.96,0.65,0.14}{M}\textcolor[rgb]{0.96,0.65,0.14}{^{-1}{}_{\binom{1}{0}}}\textcolor[rgb]{0.96,0.65,0.14}{M}\textcolor[rgb]{0.96,0.65,0.14}{_{\binom{2}{0}} *}$};
\draw (152.82,135.05) node [anchor=north west][inner sep=0.75pt]    {$\textcolor[rgb]{0.96,0.65,0.14}{M}\textcolor[rgb]{0.96,0.65,0.14}{_{\binom{i}{0}}^{-1}}\textcolor[rgb]{0.96,0.65,0.14}{M}\textcolor[rgb]{0.96,0.65,0.14}{_{\binom{i+1}{0}} *}$};
\draw (161,200.4) node [anchor=north west][inner sep=0.75pt]    {$\textcolor[rgb]{0.96,0.65,0.14}{M}\textcolor[rgb]{0.96,0.65,0.14}{_{\binom{r}{0}}^{-1}}\textcolor[rgb]{0.96,0.65,0.14}{*}$};

\end{tikzpicture}
    \caption{Shifted weights.}
    \label{fig:shifted_weights}
    \end{figure}
\end{proof}

\begin{proof}[Proof of Theorem~\ref{thm:matching_equals_minor}]
We use induction on the number of columns of the plabic fence. 

\noindent \textbf{Base Case I:}
Suppose there is only one ``white-over-black" column.

The point in $\text{Conf}(s_i, e)$ to which this fence corresponds is a single upward triangle with the left most side normalized to be of the form $tU_+\mathdash U_-$, where the upper-rightmost flag is fixed as $A^1 = e_{-i}(p)dU_+$ with the given edge weights, and $t = d(d^{\alpha_i}p)^{-\alpha^{\vee}_i}$ by Remark~\ref{rem: back_propagation_of_decoration} (see Figure~\ref{fig:base_case_1}). Thus, when we compute the generalized minors of this point, we get $A_{\binom{l}{1}} = d^{w_l} = d_1...d_l, l\in [r]$, and $$A_{\binom{l}{0}} = t^{w_l} = \begin{cases}
        d_1...d_{i-1}(d_{i+1}p^{-1}), l=i\\
        d_1...d_l, l\in [r]\setminus\{i\}\\
    \end{cases} $$

\begin{figure}[h!]
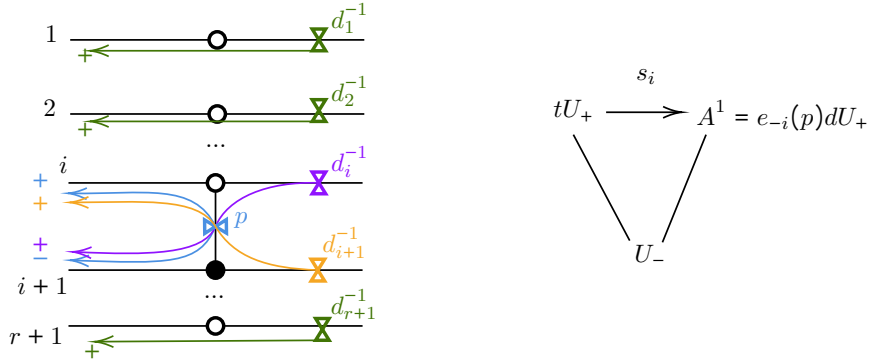

    \centering
    \includestandalone[mode=tex, width=0.7\textwidth]{figures/base_case_1}
    \caption{Base case I.}
    \label{fig:base_case_1}
\end{figure}

On the other hand, we can compute the generalized matchings. Note first that faces on the right end of the $l$-th level are on the left side of the positive strands coming out of horizontal edges with weights $d^{-1}_1,...,d^{-1}_l$, and since they are on the right side of the positive strand from the vertical edge to the direct downstream face $\binom{i}{0}$, we get $M_{\binom{l}{0}} = (d^{-1}_1...d^{-1}_l)^{-1} = d_1...d_l, l\neq i$. Lastly, since the face $\binom{i}{0}$ is in the direct downstream region of the vertical edge, and on the left side of the positive strand with weight $d^{-1}_{i+1}$, but on the right side of the positive strand with weight $d^{-1}_i$, we have $M_{\binom{i}{0}} = (d^{-1}_1...d^{-1}_{i-1}d^{-1}_{i+1}p)^{-1} = d_1...d_{i-1}d_{i+1}p^{-1} = A_{\binom{i}{0}}$. Thus, the generalized matching variables on faces recover the corresponding minors.

\noindent \textbf{Base Case II:} Suppose there is only one ``black-over-white" column. First we compute the generalized minors of the point given by Definition~\ref{def: canonical_parametrization}. See Figure~\ref{fig:base_case_2}. We have $A^0 = dU_+$, $A_0 = U_-$, and $A_1 = e_{i}(q)(t_1')^{-1}U_-$, where $t'_1 = \text{Id}_{r+1}(\text{Id}^{\alpha_i}_{r+1}q)^{-\alpha^{\vee}_{i}} = \phi_i\m{q^{-1}&0\\0&q}$, and $q$ is the column weight, and $d$ is given by the weights on the rightmost horizontal edges.

Thus, we have $A_{\binom{l}{0}} = d^{w_l} = d_1...d_l= (d^{-1}_1...d^{-1}_l)^{-1}= M_{\binom{l}{0}}$ since positive strands from horizontal edges do not cross each other. In addition, $A_{\binom{i}{1}} = t^{w_i} = d_1...d_{i-1}(d_{i}q^{-1})$. On the other hand, since the right end face on the $i$-th level is in the direct downstream region of the vertical edge, we have $M_{\binom{i}{1}} = (d^{-1}_1...d^{-1}_iq)^{-1} = d_1...d_iq^{-1}$.

\begin{figure}[h!]
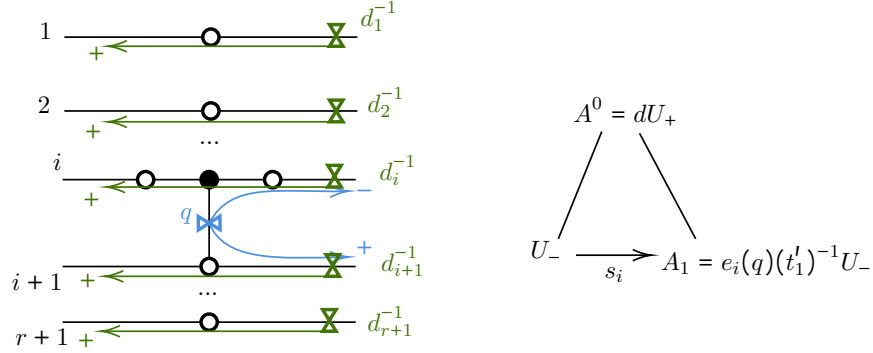

    \centering
    \includestandalone[mode=tex, width=0.7\textwidth]{figures/base_case_2}
    \caption{Base case II.}
    \label{fig:base_case_2}
\end{figure}

\noindent \textbf{Inductive Case I:} Suppose there are $(n+m+1)$ columns in $G_\beta$, and the $(n+m+1)$-th column is of type ``white-over-black." Take the corresponding point parametrized by Definition~\ref{def: canonical_parametrization}, See Figure~\ref{fig:inductive_case_1} on the right. Let $g$ denote the product of Chevalley generators of the first $n+m$ triangles, with the column weights of the plabic fence as their parameters. By Remark~\ref{rem: back_propagation_of_decoration}, we know $t_n = d(d^{\alpha_i}p_{n+1})^{-\alpha^\vee_i}$ is uniquely determined by the edge weights. By inductive hypothesis, the left subgraph $G$ with rightmost horizontal edge weights $(t^{-1}_n)_1,...,(t^{-1}_n)_{r+1}$ has face weights equal to the generalized minors of the point consisting of flags in the vertices of the first $(n+m)$ triangles in $C_\beta$.

\begin{figure}[h!]
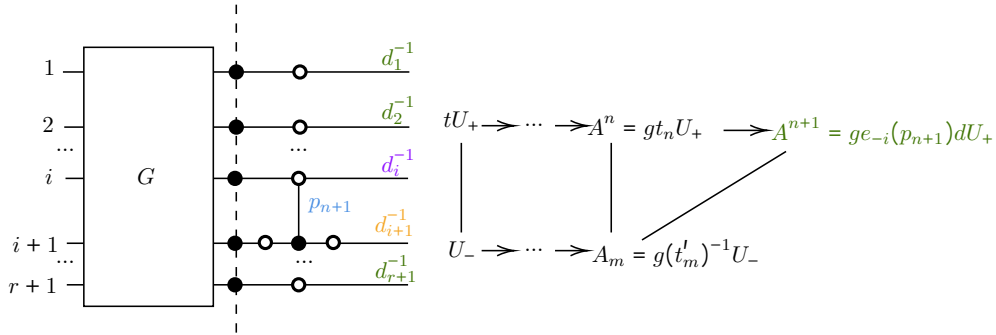

    \centering
    \includestandalone[mode=tex, width=0.8\textwidth]{figures/inductive_case_1}
    \caption{Inductive case I.}
    \label{fig:inductive_case_1}
\end{figure}

Moreover, by Base Case I, we have the following correspondence of a plabic fence with a triangulation that is left-action equivalent to the $(n+m+1)$-th triangle, see Figure~\ref{fig:inductive_case_1_rightside_triangle}:

\begin{figure}[h!]
    \centering
    \tikzset{every picture/.style={line width=0.75pt}} 

\begin{tikzpicture}[x=0.75pt,y=0.75pt,yscale=-1,xscale=1]

\draw    (270.93,117) -- (296.93,117) ;
\draw [shift={(298.93,117)}, rotate = 180] [color={rgb, 255:red, 0; green, 0; blue, 0 }  ][line width=0.75]    (10.93,-3.29) .. controls (6.95,-1.4) and (3.31,-0.3) .. (0,0) .. controls (3.31,0.3) and (6.95,1.4) .. (10.93,3.29)   ;
\draw    (226,125) -- (226,182.38) ;
\draw    (332.93,133.5) -- (240.69,190.17) ;
\draw    (21.72,88.6) -- (132.72,88.6) ;
\draw  [fill={rgb, 255:red, 255; green, 255; blue, 255 }  ,fill opacity=1 ][line width=1.5]  (58.78,88.6) .. controls (58.78,86.45) and (60.52,84.71) .. (62.67,84.71) .. controls (64.81,84.71) and (66.55,86.45) .. (66.55,88.6) .. controls (66.55,90.75) and (64.81,92.49) .. (62.67,92.49) .. controls (60.52,92.49) and (58.78,90.75) .. (58.78,88.6) -- cycle ;
\draw    (61.85,156.74) -- (61.85,198.28) ;
\draw    (22.54,123.9) -- (132.72,123.9) ;
\draw  [fill={rgb, 255:red, 255; green, 255; blue, 255 }  ,fill opacity=1 ][line width=1.5]  (58.78,123.9) .. controls (58.78,121.75) and (60.52,120.01) .. (62.67,120.01) .. controls (64.81,120.01) and (66.55,121.75) .. (66.55,123.9) .. controls (66.55,126.05) and (64.81,127.79) .. (62.67,127.79) .. controls (60.52,127.79) and (58.78,126.05) .. (58.78,123.9) -- cycle ;
\draw    (23.37,156.74) -- (111.1,156.74) -- (131.9,156.74) ;
\draw  [fill={rgb, 255:red, 255; green, 255; blue, 255 }  ,fill opacity=1 ][line width=1.5]  (57.96,156.74) .. controls (57.96,154.59) and (59.7,152.85) .. (61.85,152.85) .. controls (63.99,152.85) and (65.73,154.59) .. (65.73,156.74) .. controls (65.73,158.89) and (63.99,160.63) .. (61.85,160.63) .. controls (59.7,160.63) and (57.96,158.89) .. (57.96,156.74) -- cycle ;
\draw    (24.19,198.28) -- (110.62,198.28) -- (131.9,198.28) ;
\draw  [fill={rgb, 255:red, 0; green, 0; blue, 0 }  ,fill opacity=1 ][line width=1.5]  (57.96,198.28) .. controls (57.96,196.13) and (59.7,194.39) .. (61.85,194.39) .. controls (63.99,194.39) and (65.73,196.13) .. (65.73,198.28) .. controls (65.73,200.42) and (63.99,202.17) .. (61.85,202.17) .. controls (59.7,202.17) and (57.96,200.42) .. (57.96,198.28) -- cycle ;
\draw    (24.19,224.88) -- (131.9,224.88) ;
\draw  [fill={rgb, 255:red, 255; green, 255; blue, 255 }  ,fill opacity=1 ][line width=1.5]  (57.96,224.88) .. controls (57.96,222.73) and (59.7,220.99) .. (61.85,220.99) .. controls (63.99,220.99) and (65.73,222.73) .. (65.73,224.88) .. controls (65.73,227.02) and (63.99,228.77) .. (61.85,228.77) .. controls (59.7,228.77) and (57.96,227.02) .. (57.96,224.88) -- cycle ;
\draw  [fill={rgb, 255:red, 255; green, 255; blue, 255 }  ,fill opacity=1 ][line width=1.5]  (36.61,197.79) .. controls (36.61,195.64) and (38.35,193.9) .. (40.5,193.9) .. controls (42.65,193.9) and (44.39,195.64) .. (44.39,197.79) .. controls (44.39,199.93) and (42.65,201.67) .. (40.5,201.67) .. controls (38.35,201.67) and (36.61,199.93) .. (36.61,197.79) -- cycle ;
\draw  [fill={rgb, 255:red, 255; green, 255; blue, 255 }  ,fill opacity=1 ][line width=1.5]  (80.12,197.79) .. controls (80.12,195.64) and (81.86,193.9) .. (84.01,193.9) .. controls (86.16,193.9) and (87.9,195.64) .. (87.9,197.79) .. controls (87.9,199.93) and (86.16,201.67) .. (84.01,201.67) .. controls (81.86,201.67) and (80.12,199.93) .. (80.12,197.79) -- cycle ;
\draw    (515.93,116) -- (541.93,116) ;
\draw [shift={(543.93,116)}, rotate = 180] [color={rgb, 255:red, 0; green, 0; blue, 0 }  ][line width=0.75]    (10.93,-3.29) .. controls (6.95,-1.4) and (3.31,-0.3) .. (0,0) .. controls (3.31,0.3) and (6.95,1.4) .. (10.93,3.29)   ;
\draw    (498,124) -- (498,181.38) ;
\draw    (564,133) -- (512.69,189.17) ;
\draw    (348,163) .. controls (349.67,161.33) and (351.33,161.33) .. (353,163) .. controls (354.67,164.67) and (356.33,164.67) .. (358,163) .. controls (359.67,161.33) and (361.33,161.33) .. (363,163) .. controls (364.67,164.67) and (366.33,164.67) .. (368,163) .. controls (369.67,161.33) and (371.33,161.33) .. (373,163) .. controls (374.67,164.67) and (376.33,164.67) .. (378,163) .. controls (379.67,161.33) and (381.33,161.33) .. (383,163) .. controls (384.67,164.67) and (386.33,164.67) .. (388,163) .. controls (389.67,161.33) and (391.33,161.33) .. (393,163) .. controls (394.67,164.67) and (396.33,164.67) .. (398,163) .. controls (399.67,161.33) and (401.33,161.33) .. (403,163) .. controls (404.67,164.67) and (406.33,164.67) .. (408,163) .. controls (409.67,161.33) and (411.33,161.33) .. (413,163) .. controls (414.67,164.67) and (416.33,164.67) .. (418,163) .. controls (419.67,161.33) and (421.33,161.33) .. (423,163) .. controls (424.67,164.67) and (426.33,164.67) .. (428,163) .. controls (429.67,161.33) and (431.33,161.33) .. (433,163) .. controls (434.67,164.67) and (436.33,164.67) .. (438,163) -- (441,163) -- (449,163) ;
\draw [shift={(451,163)}, rotate = 180] [color={rgb, 255:red, 0; green, 0; blue, 0 }  ][line width=0.75]    (10.93,-3.29) .. controls (6.95,-1.4) and (3.31,-0.3) .. (0,0) .. controls (3.31,0.3) and (6.95,1.4) .. (10.93,3.29)   ;
\draw    (161,160) -- (209,160) ;
\draw [shift={(211,160)}, rotate = 180] [color={rgb, 255:red, 0; green, 0; blue, 0 }  ][line width=0.75]    (10.93,-3.29) .. controls (6.95,-1.4) and (3.31,-0.3) .. (0,0) .. controls (3.31,0.3) and (6.95,1.4) .. (10.93,3.29)   ;
\draw [shift={(159,160)}, rotate = 0] [color={rgb, 255:red, 0; green, 0; blue, 0 }  ][line width=0.75]    (10.93,-3.29) .. controls (6.95,-1.4) and (3.31,-0.3) .. (0,0) .. controls (3.31,0.3) and (6.95,1.4) .. (10.93,3.29)   ;

\draw (209,107.4) node [anchor=north west][inner sep=0.75pt]    {$t'_{m} t_{n} U_{+}$};
\draw (221,192.4) node [anchor=north west][inner sep=0.75pt]    {$U_{-}$};
\draw (307,105.4) node [anchor=north west][inner sep=0.75pt]    {$\textcolor[rgb]{0.25,0.46,0.02}{e_{-i}\left(( t'_{m})^{-\alpha _{i}} p_{n+1}\right) t'_{m} dU_{+}}$};
\draw (275,93.4) node [anchor=north west][inner sep=0.75pt]    {$s_{i}$};
\draw (63.85,164.03) node [anchor=north west][inner sep=0.75pt]    {$\textcolor[rgb]{0.29,0.56,0.89}{( t'_{m})^{-\alpha _{i}} p}\textcolor[rgb]{0.29,0.56,0.89}{_{n+1}}$};
\draw (113.43,68.24) node [anchor=north west][inner sep=0.75pt]    {$\textcolor[rgb]{0.25,0.46,0.02}{( t'_{m} d)_{1}^{-1}}$};
\draw (55,135) node [anchor=north west][inner sep=0.75pt]    {$...$};
\draw (55,208) node [anchor=north west][inner sep=0.75pt]    {$...$};
\draw (488,103.4) node [anchor=north west][inner sep=0.75pt]    {$A^{n}$};
\draw (488,185.4) node [anchor=north west][inner sep=0.75pt]    {$A_{m}$};
\draw (556,107.4) node [anchor=north west][inner sep=0.75pt]    {$A^{n+1}$};
\draw (524,93.4) node [anchor=north west][inner sep=0.75pt]    {$s_{i}$};
\draw (372,171.4) node [anchor=north west][inner sep=0.75pt]    {$g( t'_{m})^{-1} *$};
\draw (114.43,103.24) node [anchor=north west][inner sep=0.75pt]    {$\textcolor[rgb]{0.25,0.46,0.02}{(}\textcolor[rgb]{0.25,0.46,0.02}{t'}\textcolor[rgb]{0.25,0.46,0.02}{_{m}}\textcolor[rgb]{0.25,0.46,0.02}{d}\textcolor[rgb]{0.25,0.46,0.02}{)}\textcolor[rgb]{0.25,0.46,0.02}{_{2}^{-1}}$};
\draw (114.43,132.24) node [anchor=north west][inner sep=0.75pt]    {$\textcolor[rgb]{0.25,0.46,0.02}{(}\textcolor[rgb]{0.25,0.46,0.02}{t'}\textcolor[rgb]{0.25,0.46,0.02}{_{m}}\textcolor[rgb]{0.25,0.46,0.02}{d}\textcolor[rgb]{0.25,0.46,0.02}{)}\textcolor[rgb]{0.25,0.46,0.02}{_{i}^{-1}}$};
\draw (115.43,185.24) node [anchor=north west][inner sep=0.75pt]    {$\textcolor[rgb]{0.25,0.46,0.02}{(}\textcolor[rgb]{0.25,0.46,0.02}{t'}\textcolor[rgb]{0.25,0.46,0.02}{_{m}}\textcolor[rgb]{0.25,0.46,0.02}{d}\textcolor[rgb]{0.25,0.46,0.02}{)}\textcolor[rgb]{0.25,0.46,0.02}{_{i+1}^{-1}}$};
\draw (115.43,211.24) node [anchor=north west][inner sep=0.75pt]    {$\textcolor[rgb]{0.25,0.46,0.02}{(}\textcolor[rgb]{0.25,0.46,0.02}{t'}\textcolor[rgb]{0.25,0.46,0.02}{_{m}}\textcolor[rgb]{0.25,0.46,0.02}{d}\textcolor[rgb]{0.25,0.46,0.02}{)}\textcolor[rgb]{0.25,0.46,0.02}{_{r+1}^{-1}}$};

\end{tikzpicture}
    \caption{The correspondence between the last column in inductive case I and the last triangle of flags.}
    \label{fig:inductive_case_1_rightside_triangle}
\end{figure}
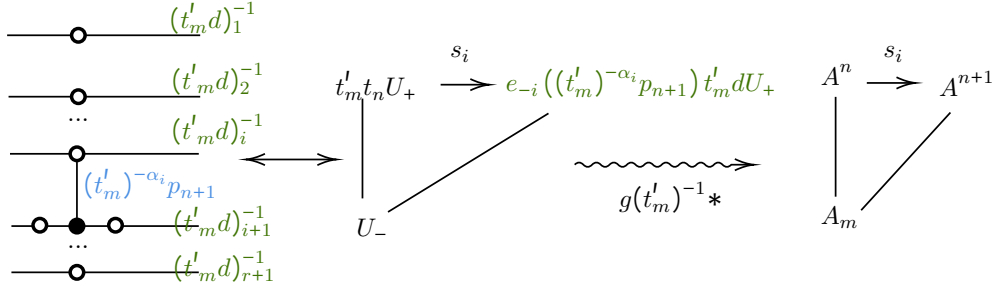

By Lemma~\ref{lem:gluing_fences}, with the additional weights $(t'_mt_n)^{w_{i-1}}(t'_mt_n)^{-w_{i}} = (t'_mt_n)^{-1}_i$ on the leftmost edge of the $i$-th level of right-hand graph, the following glued plabic fence in Figure~\ref{fig:inductive_case_1_glue} has the same face weights on the left subgraph (separated from the dashed line) as $G$ and the same face weights on the right subgraph as the second plabic fence corresponding to the above single triangle.

\begin{figure}[h!]
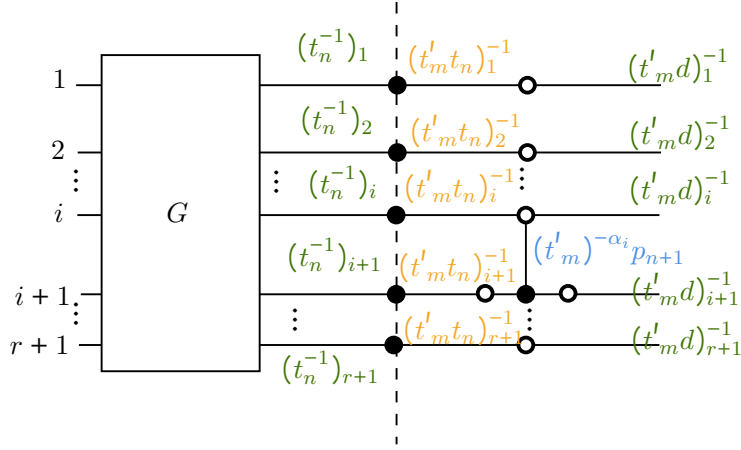

    \centering
    \includestandalone[mode=tex, width=0.6\textwidth]{figures/inductive_case_1_glue}
    \caption{The glued plabic fence in inductive case I.}
    \label{fig:inductive_case_1_glue}
\end{figure}
    
    Using gauge transformations to move the weights on the horizontal edges on two sides of the dashed line towards to the right end, we would obtain the same plabic fence we started with, which shows its face weights are the minors of the merged triangulation.

\noindent \textbf{Inductive Case II:} Suppose the $(n+m+1)$-th column is of type ``black-over-white." See Figure~\ref{fig:inductive_case_2} left. As before, let $g$ denote the product of all Chevalley generators coming from the first $(n+m)$ triangles. Then, the flags in the last triangle are $A^n = gdU_+, A_m = g(t'_m)^{-1}U_-, A_{m+1} = ge_{i}(q_{m+1})(t'_{m+1})^{-1}U_-$, where $d, q_{m+1}$ come from weights on the plabic fence, and the other decorations are determined by Remark~\ref{rem: back_propagation_of_decoration} and Lemma~\ref{lem:forward_inductive_param}. 

By the inductive hypothesis, the induced subgraph $G$ on the first $(n+m)$ columns, with rightmost horizontal edge weights $d^{-1}_1,...,d^{-1}_{r+1}$ has generalized matchings computing the generalized minors of the point corresponding to the first $(n+m)$ triangles.  

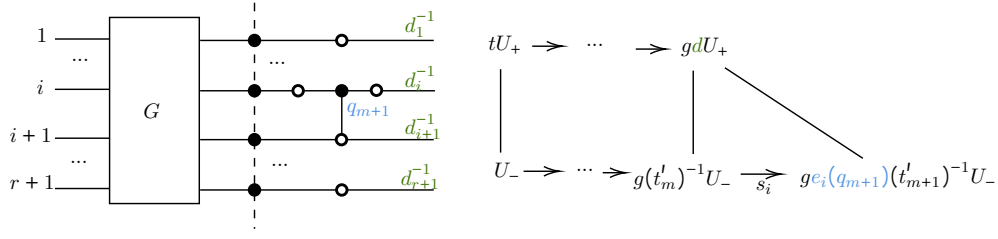
\begin{figure}[h!]
    \centering
    \tikzset{every picture/.style={line width=0.75pt}} 

\begin{tikzpicture}[x=0.75pt,y=0.75pt,yscale=-1,xscale=1]

\draw    (37,71) -- (74.81,71) ;
\draw   (75,56) -- (137.41,56) -- (137.41,185.46) -- (75,185.46) -- cycle ;
\draw    (37,106) -- (74.81,106) ;
\draw    (37,140) -- (74.81,140) ;
\draw    (37,175) -- (74.81,175) ;
\draw    (138,72) -- (175.81,72) ;
\draw    (138,107) -- (175.81,107) ;
\draw    (138,141) -- (175.81,141) ;
\draw    (138,176) -- (175.81,176) ;
\draw  [fill={rgb, 255:red, 0; green, 0; blue, 0 }  ,fill opacity=1 ][line width=1.5]  (172.16,72) .. controls (172.16,69.99) and (173.8,68.35) .. (175.81,68.35) .. controls (177.83,68.35) and (179.46,69.99) .. (179.46,72) .. controls (179.46,74.01) and (177.83,75.65) .. (175.81,75.65) .. controls (173.8,75.65) and (172.16,74.01) .. (172.16,72) -- cycle ;
\draw  [fill={rgb, 255:red, 0; green, 0; blue, 0 }  ,fill opacity=1 ][line width=1.5]  (172.16,107) .. controls (172.16,104.99) and (173.8,103.35) .. (175.81,103.35) .. controls (177.83,103.35) and (179.46,104.99) .. (179.46,107) .. controls (179.46,109.01) and (177.83,110.65) .. (175.81,110.65) .. controls (173.8,110.65) and (172.16,109.01) .. (172.16,107) -- cycle ;
\draw  [fill={rgb, 255:red, 0; green, 0; blue, 0 }  ,fill opacity=1 ][line width=1.5]  (172.16,141) .. controls (172.16,138.99) and (173.8,137.35) .. (175.81,137.35) .. controls (177.83,137.35) and (179.46,138.99) .. (179.46,141) .. controls (179.46,143.01) and (177.83,144.65) .. (175.81,144.65) .. controls (173.8,144.65) and (172.16,143.01) .. (172.16,141) -- cycle ;
\draw  [fill={rgb, 255:red, 0; green, 0; blue, 0 }  ,fill opacity=1 ][line width=1.5]  (172.16,176) .. controls (172.16,173.99) and (173.8,172.35) .. (175.81,172.35) .. controls (177.83,172.35) and (179.46,173.99) .. (179.46,176) .. controls (179.46,178.01) and (177.83,179.65) .. (175.81,179.65) .. controls (173.8,179.65) and (172.16,178.01) .. (172.16,176) -- cycle ;
\draw  [dash pattern={on 4.5pt off 4.5pt}]  (175.81,45.75) -- (175.81,203) ;
\draw    (175.81,72) -- (300.5,72) ;
\draw  [fill={rgb, 255:red, 0; green, 0; blue, 0 }  ,fill opacity=1 ][line width=1.5]  (232.88,107) .. controls (232.88,104.99) and (234.51,103.35) .. (236.52,103.35) .. controls (238.54,103.35) and (240.17,104.99) .. (240.17,107) .. controls (240.17,109.01) and (238.54,110.65) .. (236.52,110.65) .. controls (234.51,110.65) and (232.88,109.01) .. (232.88,107) -- cycle ;
\draw  [fill={rgb, 255:red, 255; green, 255; blue, 255 }  ,fill opacity=1 ][line width=1.5]  (232.88,72) .. controls (232.88,69.99) and (234.51,68.35) .. (236.52,68.35) .. controls (238.54,68.35) and (240.17,69.99) .. (240.17,72) .. controls (240.17,74.01) and (238.54,75.65) .. (236.52,75.65) .. controls (234.51,75.65) and (232.88,74.01) .. (232.88,72) -- cycle ;
\draw    (175.81,107) -- (301.33,107) ;
\draw    (175.81,141) -- (300.5,141) ;
\draw    (175.81,176) -- (300.5,176) ;
\draw  [fill={rgb, 255:red, 255; green, 255; blue, 255 }  ,fill opacity=1 ][line width=1.5]  (202.52,107) .. controls (202.52,104.99) and (204.15,103.35) .. (206.17,103.35) .. controls (208.18,103.35) and (209.82,104.99) .. (209.82,107) .. controls (209.82,109.01) and (208.18,110.65) .. (206.17,110.65) .. controls (204.15,110.65) and (202.52,109.01) .. (202.52,107) -- cycle ;
\draw  [fill={rgb, 255:red, 255; green, 255; blue, 255 }  ,fill opacity=1 ][line width=1.5]  (232.88,176) .. controls (232.88,173.99) and (234.51,172.35) .. (236.52,172.35) .. controls (238.54,172.35) and (240.17,173.99) .. (240.17,176) .. controls (240.17,178.01) and (238.54,179.65) .. (236.52,179.65) .. controls (234.51,179.65) and (232.88,178.01) .. (232.88,176) -- cycle ;
\draw  [fill={rgb, 255:red, 255; green, 255; blue, 255 }  ,fill opacity=1 ][line width=1.5]  (257,106.65) .. controls (257,104.63) and (258.63,103) .. (260.65,103) .. controls (262.66,103) and (264.3,104.63) .. (264.3,106.65) .. controls (264.3,108.66) and (262.66,110.3) .. (260.65,110.3) .. controls (258.63,110.3) and (257,108.66) .. (257,106.65) -- cycle ;
\draw    (236.52,107) -- (236.52,141) ;
\draw  [fill={rgb, 255:red, 255; green, 255; blue, 255 }  ,fill opacity=1 ][line width=1.5]  (232.67,141) .. controls (232.67,138.99) and (234.3,137.35) .. (236.32,137.35) .. controls (238.33,137.35) and (239.96,138.99) .. (239.96,141) .. controls (239.96,143.01) and (238.33,144.65) .. (236.32,144.65) .. controls (234.3,144.65) and (232.67,143.01) .. (232.67,141) -- cycle ;
\draw    (347,92) -- (347,150) ;
\draw    (481,93) -- (481,151) ;
\draw    (369,76) -- (388.03,76) ;
\draw [shift={(390.03,76)}, rotate = 180] [color={rgb, 255:red, 0; green, 0; blue, 0 }  ][line width=0.75]    (10.93,-3.29) .. controls (6.95,-1.4) and (3.31,-0.3) .. (0,0) .. controls (3.31,0.3) and (6.95,1.4) .. (10.93,3.29)   ;
\draw    (441.03,78) -- (460.03,78) ;
\draw [shift={(462.03,78)}, rotate = 180] [color={rgb, 255:red, 0; green, 0; blue, 0 }  ][line width=0.75]    (10.93,-3.29) .. controls (6.95,-1.4) and (3.31,-0.3) .. (0,0) .. controls (3.31,0.3) and (6.95,1.4) .. (10.93,3.29)   ;
\draw    (363,163) -- (386.03,163) ;
\draw [shift={(388.03,163)}, rotate = 180] [color={rgb, 255:red, 0; green, 0; blue, 0 }  ][line width=0.75]    (10.93,-3.29) .. controls (6.95,-1.4) and (3.31,-0.3) .. (0,0) .. controls (3.31,0.3) and (6.95,1.4) .. (10.93,3.29)   ;
\draw    (418.03,164) -- (432.03,164) ;
\draw [shift={(434.03,164)}, rotate = 180] [color={rgb, 255:red, 0; green, 0; blue, 0 }  ][line width=0.75]    (10.93,-3.29) .. controls (6.95,-1.4) and (3.31,-0.3) .. (0,0) .. controls (3.31,0.3) and (6.95,1.4) .. (10.93,3.29)   ;
\draw    (514.03,165) -- (540.32,165) ;
\draw [shift={(542.32,165)}, rotate = 180] [color={rgb, 255:red, 0; green, 0; blue, 0 }  ][line width=0.75]    (10.93,-3.29) .. controls (6.95,-1.4) and (3.31,-0.3) .. (0,0) .. controls (3.31,0.3) and (6.95,1.4) .. (10.93,3.29)   ;
\draw    (503,91) -- (599,154) ;

\draw (97,114.4) node [anchor=north west][inner sep=0.75pt]    {$G$};
\draw (279,50.4) node [anchor=north west][inner sep=0.75pt]    {$\textcolor[rgb]{0.25,0.46,0.02}{d_{1}^{-1}}$};
\draw (238.52,114.05) node [anchor=north west][inner sep=0.75pt]    {$\textcolor[rgb]{0.29,0.56,0.89}{q_{m+1}}$};
\draw (279,88.4) node [anchor=north west][inner sep=0.75pt]    {$\textcolor[rgb]{0.25,0.46,0.02}{d}\textcolor[rgb]{0.25,0.46,0.02}{_{i}^{-1}}$};
\draw (279,121.4) node [anchor=north west][inner sep=0.75pt]    {$\textcolor[rgb]{0.25,0.46,0.02}{d}\textcolor[rgb]{0.25,0.46,0.02}{_{i+1}^{-1}}$};
\draw (276,156.4) node [anchor=north west][inner sep=0.75pt]    {$\textcolor[rgb]{0.25,0.46,0.02}{d}\textcolor[rgb]{0.25,0.46,0.02}{_{r+1}^{-1}}$};
\draw (23,62.4) node [anchor=north west][inner sep=0.75pt]    {$1$};
\draw (23,98.4) node [anchor=north west][inner sep=0.75pt]    {$i$};
\draw (4,131.4) node [anchor=north west][inner sep=0.75pt]    {$i+1$};
\draw (4,164.4) node [anchor=north west][inner sep=0.75pt]    {$r+1$};
\draw (47,83.4) node [anchor=north west][inner sep=0.75pt]    {$...$};
\draw (46,154.4) node [anchor=north west][inner sep=0.75pt]    {$...$};
\draw (184,85.4) node [anchor=north west][inner sep=0.75pt]    {$...$};
\draw (186,156.4) node [anchor=north west][inner sep=0.75pt]    {$...$};
\draw (337,67.4) node [anchor=north west][inner sep=0.75pt]    {$tU_{+}$};
\draw (341,155.4) node [anchor=north west][inner sep=0.75pt]    {$U_{-}$};
\draw (471,68.4) node [anchor=north west][inner sep=0.75pt]    {$g\textcolor[rgb]{0.25,0.46,0.02}{d} U_{+}$};
\draw (438,156.4) node [anchor=north west][inner sep=0.75pt]    {$g( t'_{m})^{-1} U_{-}$};
\draw (405,73) node [anchor=north west][inner sep=0.75pt]    {$...$};
\draw (398,160) node [anchor=north west][inner sep=0.75pt]    {$...$};
\draw (554,155.4) node [anchor=north west][inner sep=0.75pt]    {$g\textcolor[rgb]{0.29,0.56,0.89}{e_{i}( q_{m+1})}( t'_{m+1})^{-1} U_{-}$};
\draw (523,168.4) node [anchor=north west][inner sep=0.75pt]    {$s_{i}$};

\end{tikzpicture}
    \caption{Inductive case II.}
    \label{fig:inductive_case_2}
\end{figure}

By Base Case II, after we apply a left action by $(g(t'_m)^{-1})^{-1}\cdot$, one can see the point corresponding to the last triangle has its generalized minors computed by the plabic fence in Figure~\ref{fig:inductive_case_2_last_triangle}. 

\begin{figure}[h!]
    \centering
    \tikzset{every picture/.style={line width=0.75pt}} 

\begin{tikzpicture}[x=0.75pt,y=0.75pt,yscale=-1,xscale=1]

\draw    (49,39) -- (49,97) ;
\draw    (82.03,111) -- (108.32,111) ;
\draw [shift={(110.32,111)}, rotate = 180] [color={rgb, 255:red, 0; green, 0; blue, 0 }  ][line width=0.75]    (10.93,-3.29) .. controls (6.95,-1.4) and (3.31,-0.3) .. (0,0) .. controls (3.31,0.3) and (6.95,1.4) .. (10.93,3.29)   ;
\draw    (71,37) -- (167,100) ;
\draw    (57,180) -- (57,238) ;
\draw    (74.03,251) -- (100.32,251) ;
\draw [shift={(102.32,251)}, rotate = 180] [color={rgb, 255:red, 0; green, 0; blue, 0 }  ][line width=0.75]    (10.93,-3.29) .. controls (6.95,-1.4) and (3.31,-0.3) .. (0,0) .. controls (3.31,0.3) and (6.95,1.4) .. (10.93,3.29)   ;
\draw    (79,178) -- (175,241) ;
\draw [line width=1.5]    (118,129) .. controls (119.67,130.67) and (119.67,132.33) .. (118,134) .. controls (116.33,135.67) and (116.33,137.33) .. (118,139) .. controls (119.67,140.67) and (119.67,142.33) .. (118,144) .. controls (116.33,145.67) and (116.33,147.33) .. (118,149) .. controls (119.67,150.67) and (119.67,152.33) .. (118,154) .. controls (116.33,155.67) and (116.33,157.33) .. (118,159) .. controls (119.67,160.67) and (119.67,162.33) .. (118,164) -- (118,168) -- (118,176) ;
\draw [shift={(118,179)}, rotate = 270] [color={rgb, 255:red, 0; green, 0; blue, 0 }  ][line width=1.5]    (14.21,-4.28) .. controls (9.04,-1.82) and (4.3,-0.39) .. (0,0) .. controls (4.3,0.39) and (9.04,1.82) .. (14.21,4.28)   ;
\draw    (446.81,169) -- (571.5,169) ;
\draw  [fill={rgb, 255:red, 0; green, 0; blue, 0 }  ,fill opacity=1 ][line width=1.5]  (503.88,204) .. controls (503.88,201.99) and (505.51,200.35) .. (507.52,200.35) .. controls (509.54,200.35) and (511.17,201.99) .. (511.17,204) .. controls (511.17,206.01) and (509.54,207.65) .. (507.52,207.65) .. controls (505.51,207.65) and (503.88,206.01) .. (503.88,204) -- cycle ;
\draw  [fill={rgb, 255:red, 255; green, 255; blue, 255 }  ,fill opacity=1 ][line width=1.5]  (503.88,169) .. controls (503.88,166.99) and (505.51,165.35) .. (507.52,165.35) .. controls (509.54,165.35) and (511.17,166.99) .. (511.17,169) .. controls (511.17,171.01) and (509.54,172.65) .. (507.52,172.65) .. controls (505.51,172.65) and (503.88,171.01) .. (503.88,169) -- cycle ;
\draw    (446.81,204) -- (572.33,204) ;
\draw    (446.81,238) -- (571.5,238) ;
\draw    (446.81,273) -- (571.5,273) ;
\draw  [fill={rgb, 255:red, 255; green, 255; blue, 255 }  ,fill opacity=1 ][line width=1.5]  (473.52,204) .. controls (473.52,201.99) and (475.15,200.35) .. (477.17,200.35) .. controls (479.18,200.35) and (480.82,201.99) .. (480.82,204) .. controls (480.82,206.01) and (479.18,207.65) .. (477.17,207.65) .. controls (475.15,207.65) and (473.52,206.01) .. (473.52,204) -- cycle ;
\draw  [fill={rgb, 255:red, 255; green, 255; blue, 255 }  ,fill opacity=1 ][line width=1.5]  (503.88,273) .. controls (503.88,270.99) and (505.51,269.35) .. (507.52,269.35) .. controls (509.54,269.35) and (511.17,270.99) .. (511.17,273) .. controls (511.17,275.01) and (509.54,276.65) .. (507.52,276.65) .. controls (505.51,276.65) and (503.88,275.01) .. (503.88,273) -- cycle ;
\draw  [fill={rgb, 255:red, 255; green, 255; blue, 255 }  ,fill opacity=1 ][line width=1.5]  (528,203.65) .. controls (528,201.63) and (529.63,200) .. (531.65,200) .. controls (533.66,200) and (535.3,201.63) .. (535.3,203.65) .. controls (535.3,205.66) and (533.66,207.3) .. (531.65,207.3) .. controls (529.63,207.3) and (528,205.66) .. (528,203.65) -- cycle ;
\draw    (507.52,204) -- (507.52,238) ;
\draw  [fill={rgb, 255:red, 255; green, 255; blue, 255 }  ,fill opacity=1 ][line width=1.5]  (503.67,238) .. controls (503.67,235.99) and (505.3,234.35) .. (507.32,234.35) .. controls (509.33,234.35) and (510.96,235.99) .. (510.96,238) .. controls (510.96,240.01) and (509.33,241.65) .. (507.32,241.65) .. controls (505.3,241.65) and (503.67,240.01) .. (503.67,238) -- cycle ;
\draw [line width=1.5]    (339,219) -- (391,219) ;
\draw [shift={(394,219)}, rotate = 180] [color={rgb, 255:red, 0; green, 0; blue, 0 }  ][line width=1.5]    (14.21,-4.28) .. controls (9.04,-1.82) and (4.3,-0.39) .. (0,0) .. controls (4.3,0.39) and (9.04,1.82) .. (14.21,4.28)   ;
\draw [shift={(336,219)}, rotate = 0] [color={rgb, 255:red, 0; green, 0; blue, 0 }  ][line width=1.5]    (14.21,-4.28) .. controls (9.04,-1.82) and (4.3,-0.39) .. (0,0) .. controls (4.3,0.39) and (9.04,1.82) .. (14.21,4.28)   ;

\draw (39,14.4) node [anchor=north west][inner sep=0.75pt]    {$g\textcolor[rgb]{0.25,0.46,0.02}{d} U_{+}$};
\draw (6,102.4) node [anchor=north west][inner sep=0.75pt]    {$g( t'_{m})^{-1} U_{-}$};
\draw (122,101.4) node [anchor=north west][inner sep=0.75pt]    {$g\textcolor[rgb]{0.29,0.56,0.89}{e}\textcolor[rgb]{0.29,0.56,0.89}{_{i}}\textcolor[rgb]{0.29,0.56,0.89}{(}\textcolor[rgb]{0.29,0.56,0.89}{q}\textcolor[rgb]{0.29,0.56,0.89}{_{m+1}}\textcolor[rgb]{0.29,0.56,0.89}{)}( t'_{m+1})^{-1} U_{-}$};
\draw (91,114.4) node [anchor=north west][inner sep=0.75pt]    {$s_{i}$};
\draw (47,155.4) node [anchor=north west][inner sep=0.75pt]    {$t'_{m}\textcolor[rgb]{0.25,0.46,0.02}{d} U_{+}$};
\draw (49,242.4) node [anchor=north west][inner sep=0.75pt]    {$U_{-}$};
\draw (109,242.4) node [anchor=north west][inner sep=0.75pt]    {$(t'_{m})^{-1}\textcolor[rgb]{0.29,0.56,0.89}{e}\textcolor[rgb]{0.29,0.56,0.89}{_{i}}\textcolor[rgb]{0.29,0.56,0.89}{(}\textcolor[rgb]{0.29,0.56,0.89}{q}\textcolor[rgb]{0.29,0.56,0.89}{_{m+1}}\textcolor[rgb]{0.29,0.56,0.89}{)}( t'_{m+1})^{-1} U_{-}$};
\draw (76.03,254.4) node [anchor=north west][inner sep=0.75pt]    {$s_{i}$};
\draw (110,266.4) node [anchor=north west][inner sep=0.75pt]    {$=\textcolor[rgb]{0.29,0.56,0.89}{e}\textcolor[rgb]{0.29,0.56,0.89}{_{i}}\textcolor[rgb]{0.29,0.56,0.89}{(}\textcolor[rgb]{0.29,0.56,0.89}{(}\textcolor[rgb]{0.29,0.56,0.89}{t'}\textcolor[rgb]{0.29,0.56,0.89}{_{m}}\textcolor[rgb]{0.29,0.56,0.89}{)}\textcolor[rgb]{0.29,0.56,0.89}{^{-\alpha _{i}}}\textcolor[rgb]{0.29,0.56,0.89}{q}\textcolor[rgb]{0.29,0.56,0.89}{_{m+1}}\textcolor[rgb]{0.29,0.56,0.89}{)}( t'_{m})^{-1}( t'_{m+1})^{-1} U_{-}$};
\draw (550,147.4) node [anchor=north west][inner sep=0.75pt]    {$\textcolor[rgb]{0.25,0.46,0.02}{d}\textcolor[rgb]{0.25,0.46,0.02}{_{1}^{-1}}$};
\draw (455,212.05) node [anchor=north west][inner sep=0.75pt]    {$\textcolor[rgb]{0.29,0.56,0.89}{(}\textcolor[rgb]{0.29,0.56,0.89}{t'}\textcolor[rgb]{0.29,0.56,0.89}{_{m}}\textcolor[rgb]{0.29,0.56,0.89}{)}\textcolor[rgb]{0.29,0.56,0.89}{^{-\alpha _{i}}}\textcolor[rgb]{0.29,0.56,0.89}{q}\textcolor[rgb]{0.29,0.56,0.89}{_{m+1}}$};
\draw (550,185.4) node [anchor=north west][inner sep=0.75pt]    {$\textcolor[rgb]{0.25,0.46,0.02}{d}\textcolor[rgb]{0.25,0.46,0.02}{_{i}^{-1}}$};
\draw (550,218.4) node [anchor=north west][inner sep=0.75pt]    {$\textcolor[rgb]{0.25,0.46,0.02}{d}\textcolor[rgb]{0.25,0.46,0.02}{_{i+1}^{-1}}$};
\draw (547,253.4) node [anchor=north west][inner sep=0.75pt]    {$\textcolor[rgb]{0.25,0.46,0.02}{d}\textcolor[rgb]{0.25,0.46,0.02}{_{r+1}^{-1}}$};
\draw (453,253.4) node [anchor=north west][inner sep=0.75pt]    {$...$};
\draw (449,181.4) node [anchor=north west][inner sep=0.75pt]    {$...$};
\draw (428,159.4) node [anchor=north west][inner sep=0.75pt]    {$1$};
\draw (428,195.4) node [anchor=north west][inner sep=0.75pt]    {$i$};
\draw (409,228.4) node [anchor=north west][inner sep=0.75pt]    {$i+1$};
\draw (409,261.4) node [anchor=north west][inner sep=0.75pt]    {$r+1$};

\end{tikzpicture}
    \caption{The correspondence between the last column in the plabic fence in inductive case II and the last triangle of flags.}
    \label{fig:inductive_case_2_last_triangle}
\end{figure}
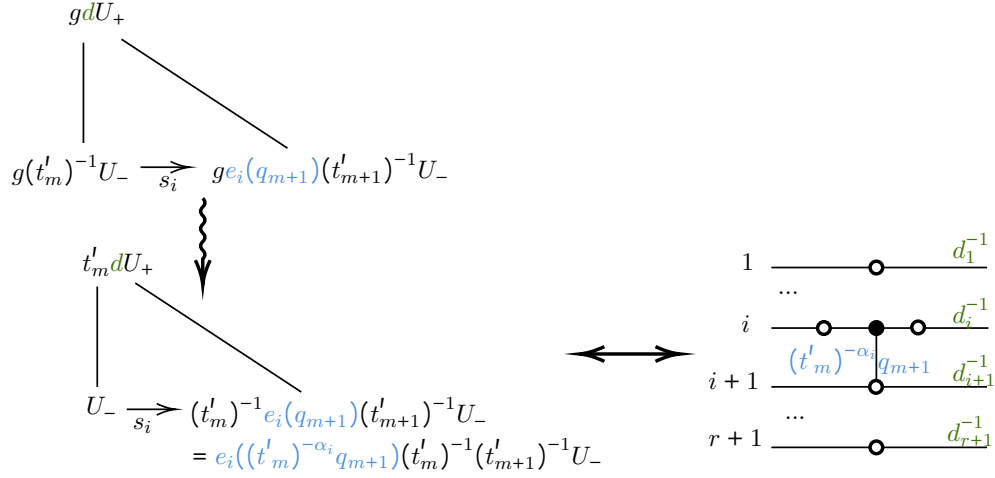

Applying Lemma~\ref{lem:gluing_fences}, we get the following plabic fence in Figure~\ref{fig:inductive_case_2_glue}, whose generalized matchings compute exactly all the generalized minors of the point given by the $(n+m+1)$ triangles. After applying a gauge transformation to make the glued plabic fence right-canonically weighted, we get back plabic fence in Figure~\ref{fig:inductive_case_2}. 
\begin{figure}[h!]
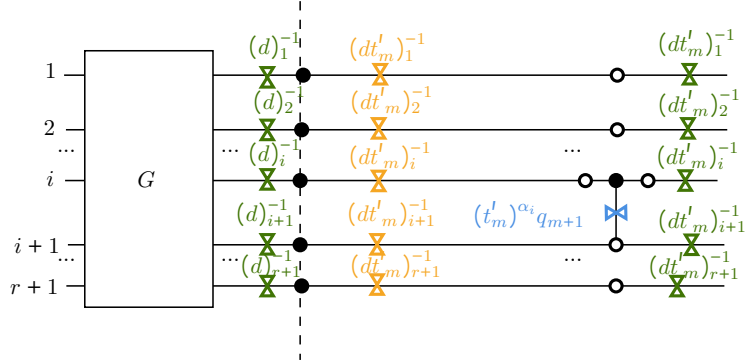

    \centering
    \includestandalone[mode=tex, width=0.6\textwidth]{figures/inductive_case_2_glue}
    \caption{The glued plabic fence in inductive case II.}
    \label{fig:inductive_case_2_glue}
\end{figure}
\end{proof}

\begin{rem}\label{rem: technical_difficulty_of_contructions}
    As a concluding remark for this section, we explain some subtleties and technical difficulties of our construction. In the standard parameterization given by~\cite{Shen_Weng_2021}, we fix some diagonal matrix $t$ and normalize the first flags as $A^0 = tU_+, A_0 = U_-$. The decorations are then propagated to the other flags, albeit several diagonal matrices appear in the parametrization of each flag. To keep a single diagonal matrix component per decorated flag, we can naturally use Lemma~\ref{lem:torus_param}. Following this convention, it is more natural to use gauge transformations to normalize the horizontal edge weights of $G_\beta$ so that all horizontal edges but the leftmost ones have weights $1$, and use the weights on the leftmost horizontal edges as entries in $t$. This is enough when there are only ``white-over-black" columns in $G_\beta$. However, when there are two types of columns, in order for the generalized matching variables to compute the corresponding minors in $T(C_\beta)$, one is forced to assign edge weights to the rightmost horizontal edges which are non-trivial monomials on $t$ and the vertical edge weights. Thus, it is not straightforward to simply start with a weighted plabic fence as an input to build a cluster in the double Bott-Samelson variety as the edge weights must satisfy non-trivial relations. To fix this, we defined our right canonical gauge, and used the rightmost horizontal edge weights to fix the decoration $d$ on the top-right flag $A^m$ and propagate from right to left, while still being able to normalize the bottom-left flag to be $A_0 = U_-$. The advantage of using the right canonical gauge is that the only relation between non-trivial edge weights (i.e. the vertical edge weights and the rightmost horizontal edge weights) is that the rightmost horizontal edge weights multiply to 1. 
\end{rem}

\section{Type $A$ Chamber Ansatz on Double Bott-Samelson Varieties}\label{sec: chamber_ansatz}

In this section, we derive an inverse formula to reconstruct the edge weights of a right canonically weighted plabic fence $G_\beta$ from its associated generalized minors $A_{\binom{i}{a}}$. This parallels the Chamber Ansatz from \cite{DoubleBruhatCells} in the reduced case. In \cite{Shen_Weng_2021}, they derive a similar result for a canonical triangle in their parametrization. This section extends that to our modified parametrization, and provides a closed formula for the reconstruction of the weights for an arbitrary triangulation of the fence. 

We start by developing some dimer theory background adapted to the context of plabic fences. For now, we take any weighted plabic fence $G_\beta$ (not necessarily right canonically weighted), with edge weights $w: E(G_\beta)\rightarrow \mathbb{C}^*$. We use $\partial f$ to denote the set of boundary edges of $f$. Let $\mathcal{F}$ be any statement. We define $[\mathcal{F}]$ to be $1$ if $\mathcal{F}$ is true, and $0$ otherwise. In particular, we are interested in $[e\in \gamma$], where the statement means that path $\gamma$ visits edge $e$.

\begin{prop}\label{prop: interior_face_alternating_prod}
    Let $f$ be an interior face of $G_\beta$. Label the edges of $\partial f$ as $e_1,e_2,...,e_{2k-1}, e_{2k}$ clockwise, such that each $e_{2i-1}$ is oriented black-to-white clockwise, and each $e_{2i}$ is oriented white-to-black clockwise. Let $f_j$ denote the face neighboring $f$ along the edge $e_j, j=1,...,2k$, where we allow duplicates. Then we have
    \begin{equation}\label{eqn: monodromy_of_interior_face}
        \frac{\prod^k_{i=1}w(e_{2i})}{\prod^k_{i=1}w(e_{2i-1})} = \frac{\prod^k_{i=1}M_{f_{2i-1}}}{\prod^k_{i=1}M_{f_{2i}}}.
    \end{equation}
\end{prop}
\begin{proof}
    Recall $M_f = \prod_{e\in E}w(e)^{-U(e,f)}$. Thus, we need to compute $\sum^k_{i=1}U(e,f_{2i}) - \sum^k_{i=1}U(e,f_{2i-1})$.

    It follows from Lemma~\ref{rem:monotonicity} that any $\gamma^+_e$ (or $\gamma^-_e$) visits a disjoint union of the boundary edges of $f$, exactly once for each such edge, and all these edges are visited by $\gamma^+_e$ (or $\gamma^-_e$) in the same direction: either they all follow the clockwise orientation of $\partial f$ or they all follow the counterclockwise orientation of $\partial f$. Define $\text{sgn}(\gamma^+_e) = 1$ (or $\text{sgn}(\gamma^-_e) = -1$, resp.) if $\gamma^+_e$ (or $\gamma^-_e$, resp.) travels clockwise through $\partial f$, and $\text{sgn}(\gamma^+_e) = -1$ (or $\text{sgn}(\gamma^-_e) = 1$, resp.) if $\gamma^+_e$ (or $\gamma^-_e$, resp.) travels counterclockwise through $\partial f$.
    
    For any $e$ that is not a boundary edge of $f$, we have: 
    \begin{equation}\label{eqn: face_weight_change_for_non_midpoint_edge}
        U(e, f_j) - U(e,f) = \text{sgn}(\gamma^+_e)[e_j \in \gamma^+_e]+\text{sgn}(\gamma^-_e)[e_j \in \gamma^-_e]
    \end{equation}

    Thus, we have 

    \begin{equation}
        \begin{aligned}
        \sum_{i=1}^k U(e,f_{2i})
         -\sum_{i=1}^k U(e,f_{2i-1})
         &=
         \sum_{i=1}^k\bigl(U(e,f_{2i})-U(e,f)\bigr) -
         \sum_{i=1}^k\bigl(U(e,f_{2i-1})-U(e,f)\bigr) \\
         &=
         \operatorname{sgn}(\gamma_e^+)
         \sum_{j=1}^{2k}(-1)^j[e_j\in\gamma_e^+] +
         \operatorname{sgn}(\gamma_e^-)
         \sum_{j=1}^{2k}(-1)^j[e_j\in\gamma_e^-].
        \end{aligned}
    \end{equation}\label{eq:interior-face-sum}
    
    Note that, because $e\notin \partial f$, we note that $\gamma^+_e$ (or $\gamma^-_e$) visits equal number of white-to-black edges of $\partial f$ as that of black-to-white edges. See Figure~\ref{fig:strands_traveling_on_interior_face} left. Thus, $\sum^{2k}_{j=1}(-1)^j[e_j \in \gamma^+_e] = 0 = \sum^{2k}_{j=1}(-1)^j[e_j \in \gamma^-_e]$.

    Now consider $e = e_{2i_0}, 1\leq i_0\leq k$. We have $U(e,f_{2i_0}) - U(e,f) = -1$, and Equation~\ref{eqn: face_weight_change_for_non_midpoint_edge} holds for all $f_j, j\neq 2i_0$. Moreover, in this case we know $\text{sgn}(\gamma^+_e) = -1$ and $\text{sgn}(\gamma^-_e) = -1$. Note also that $\gamma^+_e$ (or $\gamma^-_e$) visits exactly one more black-to-white boundary edge than it visits white-to-black edges, not counting the edge $e_{2i_0}$. See Figure~\ref{fig:strands_traveling_on_interior_face} middle. This is because the first vertex $\gamma^+_e$ visits is white, and $\gamma^+_e$ departs away from $\partial f$ and returns to $\partial f$ at a black vertex of $\partial f$.
    
    Thus, we have

    \begin{equation}
        \begin{split}
            \sum^k_{i=1}U(e,f_{2i}) - \sum^k_{i=1}U(e,f_{2i-1})
            &= -1 -(\sum_{j\neq 2i_0}(-1)^j([e_j \in \gamma^+_e]+[e_j \in \gamma^-_e]))=-1+2 =1
        \end{split}
    \end{equation}

    If $e = e_{2i_0-1}, 1\leq i_0\leq k$, we have $U(e,f_{2i_0-1}) - U(e,f) = 1$, and Equation~\ref{eqn: face_weight_change_for_non_midpoint_edge} holds for all $f_j, j\neq 2i_0-1$. We also have $\text{sgn}(\gamma^+_e) = +1$ and $\text{sgn}(\gamma^-_e) = +1$. By a similar argument as above, we know $\gamma^+_e$ (or $\gamma^-_e$) visits an equal number of black-to-white-boundary edges as white-to-black boundary edges, not counting $e_{2i_0-1}$. See Figure~\ref{fig:strands_traveling_on_interior_face} right. Hence, in this case we have $\sum^k_{i=1}U(e,f_{2i}) - \sum^k_{i=1}U(e,f_{2i-1}) = -(U(e, f_{2i_0-1})-U(e,f))=-1$. 
\end{proof}

\begin{figure}[h!]
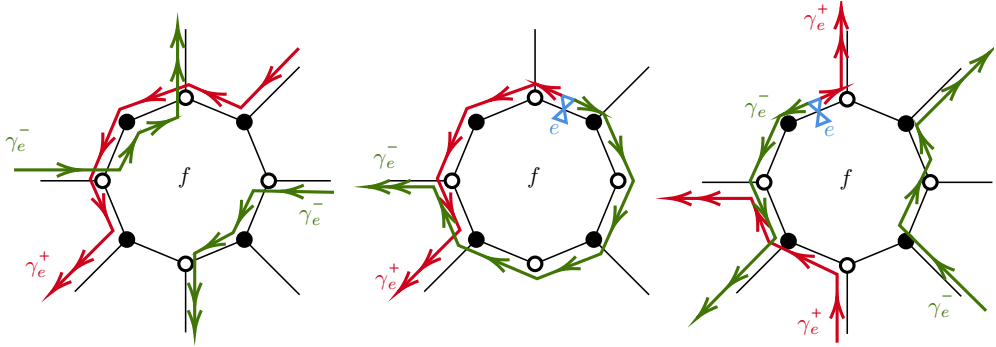

    \centering
    \includestandalone[mode=tex, width=0.8\textwidth]{figures/strands_traveling_on_interior_face}
    \caption{Illustration of how positive and negative strands coming from an edge $e$ behave along $\partial f$; Left: $e\notin \partial f$, Middle: $e$ is a clockwise white-to-black boundary edge of $f$, Right: $e$ is a black-to-white boundary edge of $f$.}
    \label{fig:strands_traveling_on_interior_face}
\end{figure} 

\begin{prop}\label{prop: boundary_face_alternating_prod}
    Let $f$ be a boundary face of $G_\beta$ (that is, a leftmost or rightmost face in some row). Label the edges of $\partial f$ as $e_1,e_2,...,e_{2k-1}, e_{2k}$ clockwise. Note that $k\geq 2$. Each $e_{2i+1}$ is oriented black-to-white clockwise, and each $e_{2i}$ is oriented white-to-black clockwise $i=1,...,k-1$, and $e_1$ ($e_{2k}$, resp.) is the first (last, resp.) edge of $\partial f$ with respect to a clockwise orientation. Let $f_j$ denote the face neighboring $f$ along the edge $e_j, j=1,...,2k$, where we allow duplicates. Then we have
    \begin{equation}\label{eqn: monodromy_of_boundary_face}
        \frac{\prod^k_{i=1}w(e_{2i})}{\prod^k_{i=1}w(e_{2i-1})} = \frac{\prod^k_{i=1}M_{f_{2i-1}}}{M_f\prod^{k-1}_{i=1}M_{f_{2i}}}
    \end{equation}
\end{prop}
\begin{proof}
    We need to compute $U(e,f) + \sum^{k-1}_{i=1} U(e,f_{2i}) - \sum^{k}_{i=1}U(e,f_{2i-1})$ for all edges $e$. This is equivalent to $\sum^{k-1}_{i=1} (U(e,f_{2i})-U(e,f)) - \sum^{k}_{i=1}(U(e,f_{2i-1})-U(e,f))$. 
    
    For any $e\notin \partial f$, we note that $\gamma^+_e$ (or $\gamma^-_e$) visits equal number of white-to-black edges and black-to-white-edges of $\partial f$, not counting $e_{2k}$. Thus, plugging in ~\ref{eqn: face_weight_change_for_non_midpoint_edge}, we see $\sum^{k-1}_{i=1} (U(e,f_{2i})-U(e,f)) - \sum^{k}_{i=1}(U(e,f_{2i-1})-U(e,f)) = 0$. 

    For any $e = e_{2i_0} - 1, 1\leq i_0\leq k$, we note that $\text{sgn}(\gamma^+_e) = 1 = \text{sgn}(\gamma^-_e)$, and moreover, $\gamma^+_e$ (or $\gamma^-_e$) visits exactly the same number of black-to-white edges and white-to-black edges on $\partial f$, not counting $e_{2i_0-1}, e_{2k}$. By a similar calculation as in the previous proof, we get $\sum^{k-1}_{i=1} (U(e,f_{2i})-U(e,f)) - \sum^{k}_{i=1}(U(e,f_{2i-1})-U(e,f)) = -(U(e,f_{2i_0-1})-U(e,f)) = -1$.

    For any $e = e_{2i_0}, 1\leq i_0\leq k$, we note that $\text{sgn}(\gamma^+_e) = -1 = \text{sgn}(\gamma^-_e)$, and $\gamma^+_e$ (or $\gamma^-_e$) visits exactly one more black-to-white edge than white-to-black edges on $\partial f$, not counting $e_{2i_0}, e_{2k}$. Thus, $\sum^{k-1}_{i=1} (U(e,f_{2i})-U(e,f)) - \sum^{k}_{i=1}(U(e,f_{2i-1})-U(e,f)) = -(-1)= 1$.
\end{proof}
\begin{rem}
    Under the right-canonical weighting, the left-hand-side of ~\ref{eqn: monodromy_of_interior_face} and ~\ref{eqn: monodromy_of_boundary_face} simplifies to be $w(e_L)^{\delta(e_L)}w(e_R)^{\delta(e_R)}$ for any $f$ that is not the rightmost face in a row, where $e_L,e_R$ are the left, right vertical edges of the face $f$, and $\delta(e_L) = 1$ (or $\delta(e_R) = 1$) if $e_L$ (or $e_R$) is a black-over-white (or white-over-black) vertical edge, and $\delta(e_L) = -1$ (or $\delta(e_R) = -1$) if $e_L$ (or $e_R$) is a white-over-black (or black-over-white) vertical edge, and $\delta(e_L) = 0$ if $f$ is a leftmost boundary face.
\end{rem}

\begin{rem}
    The left-hand-side quantity in ~\ref{eqn: monodromy_of_interior_face} and ~\ref{eqn: monodromy_of_boundary_face} is known as a \textbf{face weight} and known to be an invariant under gauge transformations (see ~\cite[Section 11]{postnikov_2006}, also ~\cite[Section 2.3]{kenyon2007heightfluctuationshoneycombdimer}); we denote it as $X_f$. The right-hand-side quantity in ~\ref{eqn: monodromy_of_interior_face} is the \textbf{exchange ratio} (or variables in what is called \textbf{$Y$-seed} in ~\cite{fomin2024introductionclusteralgebraschapters}) at a mutable vertex in the quiver $Q_\beta$ (see also ~\cite{fraser2018braidgroupsymmetriesgrassmannian}).
\end{rem}

To derive a clean formula for each column weight $p$ (or $q$), we define a new quantity $B(f)$ for any face $f$, as the following ``alternating sign" product of generalized matching variables. 

\begin{defn}\label{defn: left_side_alt_prod}
    Let $f = \binom{i}{a}, a\geq 0, i\in [r]$ be a face of $G_\beta$ in the $i$-th row, define $B(f):= \displaystyle\prod^a_{a'=0}M^{\epsilon(\binom{i}{a'})}_{\binom{i}{a'}}$ (resp. $\hat{B}(f) := \displaystyle\prod^a_{a'=0}M^{\hat
{\epsilon}(\binom{i}{a'})}_{\binom{i}{a'}}$), where the exponent $\epsilon(\binom{i}{a'})$ (resp. $\hat
{\epsilon}(\binom{i}{a'})$) is defined as:
    \begin{enumerate}
        \item[(i)] $\epsilon(\binom{i}{0}) = -1$ (resp. $\hat{\epsilon}(\binom{i}{0}) = -1$) if the rightmost vertical edge of the boundary face $\binom{i}{0}$ is black-over-white, and $\epsilon(\binom{i}{0}) = 0$ (resp. $\hat{\epsilon}(\binom{i}{0}) = 0$) otherwise.
        \item[(ii)] If $a > 0 $, then $\epsilon(\binom{i}{a}) = -1$ (resp. $\hat{\epsilon}(\binom{i}{a}) = 0$) if the leftmost vertical edge of $\binom{i}{a}$ is white-over-black, and $\epsilon(\binom{i}{a}) = 0$ (resp. $\hat{\epsilon}(\binom{i}{a}) = 1$) otherwise.
        \item[(iii)] For any $0 < a' < a$, there are three cases:
        \begin{enumerate}
            \item[(1)] If the leftmost vertical edge is white-over-black and the rightmost vertical edge is black-over-white, then $\epsilon(\binom{i}{a'}) = -1$ (resp. $\hat{\epsilon}(\binom{i}{a'}) = -1$);
            \item[(2)] If the leftmost vertical edge is black-over-white and the rightmost vertical edge is white-over-black, then $\epsilon(\binom{i}{a'}) = 1$ (resp. $\hat{\epsilon}(\binom{i}{a'}) = 1$)
            \item[(3)] $\epsilon(\binom{i}{a'}) = 0$ (resp. $\hat{\epsilon}(\binom{i}{a'}) = 0$) otherwise. 
        \end{enumerate}   
    \end{enumerate}
    We also set the convention that if $f$ is the top face of $G_\beta$, then $B(f) = 1$ and $\hat{B}(f) = 1$; if $f$ is the bottom face of $G_\beta$, then $B(f) = M^{-1}_{f}$ and $\hat{B}(f) = 1$.
\end{defn}

We obtain the following inverse formula expressed in the flavor of Theorem 2.15 in ~\cite{TwistOnRichardson}. 
\begin{thm}\label{thm:chamber_ansatz}
    Let $G_\beta$ be a right-canonically weighted plabic fence associated with the double word $\beta$. Let $p$ (or $q$, resp.) denote the weight of a white-over-black (or black-over-white, resp.) vertical edge $e$ of $G_\beta$. Then we have
    \begin{equation}\label{eqn: inverse_edge_weights}
        p = \frac{B(f_{\rightarrow})B(f_{\leftarrow})}{B(f_\uparrow)B(f_{\downarrow})},\:\: q = \frac{\hat{B}(f_\uparrow)\hat{B}(f_{\downarrow})}{\hat{B}(f_{\rightarrow})\hat{B}(f_{\leftarrow})},
    \end{equation}
    where $f_{\leftarrow}, f_{\rightarrow}, f_{\uparrow}, f_{\downarrow}$ are the faces on the left, right, top, bottom of $e$, respectively. 
    
    In particular, if $\beta_- = \emptyset$ (or $\beta_+ = \emptyset$), the above formula reduces to 
    \begin{equation}
        p = \frac{M(f_{\rightarrow})M(f_{\leftarrow})}{M(f_\uparrow)M(f_{\downarrow})},\:\: q = \frac{M(f_\uparrow)M(f_{\downarrow})}{M(f_{\rightarrow})M(f_{\leftarrow})}.
    \end{equation}
\end{thm}
\begin{exmp}
    We continue with Example~\ref{exmp: computing_gen_matching}. We demonstrate here how to use Equation~\ref{eqn: inverse_edge_weights} to obtain an expression of $p_3$ in terms of the generalized matching variables $M_f$.  

    An illustration of $f_{\leftarrow}, f_\rightarrow, f_\uparrow, f_\downarrow$ of $p_3$ is provided in the leftmost column of Figure~\ref{fig:inverse_formula_for_column_weights}. We first compute the exponents of $B(f_\uparrow), B(f_\downarrow)$ as in the first row of the second and third columns in Figure~\ref{fig:inverse_formula_for_column_weights}, and $B(f_{\rightarrow}), B(f_\leftarrow)$ are in the second row of the second and third columns in Figure~\ref{fig:inverse_formula_for_column_weights}. Note that because $p_3$ is a white-over-black edge, we set the convention to be $B(f_\downarrow) = M^{-1}_{f_\downarrow}$.

    Thus, the right hand side of Equation~\ref{eqn: inverse_edge_weights} gives us
    \begin{equation}
        \begin{split}
            \frac{B(f_{\rightarrow})B(f_{\leftarrow})}{B(f_\uparrow)B(f_{\downarrow})} &= \frac{M^{-2}_{\binom{2}{1}}M_{\binom{2}{2}}M^{-1}_{\binom{2}{3}}}{M^{-1}_{\binom{1}{0}}M_{\binom{1}{1}}M^{-1}_{\binom{1}{2}}M^{-1}_{f_\downarrow}}\\
            &= \frac{\left(\frac{d_1d_3}{p_3}\right)^{-2}\frac{d_1d_3}{q_1q_2p_3}\left(\frac{d_1d_2}{q_1q_2}\right)^{-1}}{\left(\frac{p_3d_2}{p_1p_2}\right)^{-1}\frac{p_3d_2}{p_1q_1p_2}\left(\frac{d_1}{q_1}\right)^{-1}(d_1d_2d_3)^{-1}}=\frac{(d_1d_3)^{-1}p_3}{(d_1d_3)^{-1}} = p_3
        \end{split}
    \end{equation}
    \begin{figure}[h!]
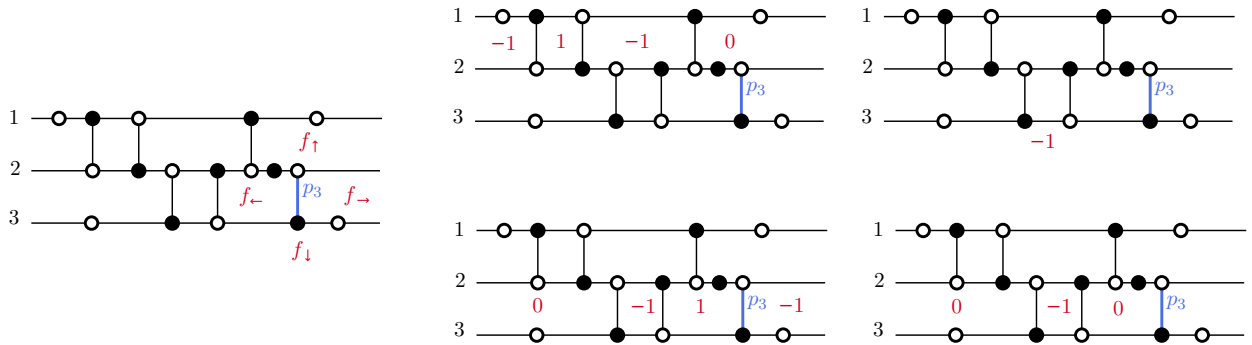

    \centering
    \includestandalone[mode=tex, width=\textwidth]{figures/inverse_formula_for_column_weights}
    \caption{Left: Positions of $f_{\leftarrow}, f_\rightarrow, f_\uparrow, f_\downarrow$ relative to the vertical edge $p_3$; Middle and Right: Exponents of $B(f_\uparrow), B(f_\downarrow), B(f_\rightarrow), B(f_\leftarrow)$, respectively.}
    \label{fig:inverse_formula_for_column_weights}
\end{figure}
\end{exmp}

\begin{proof}[Proof of Theorem~\ref{thm:chamber_ansatz}]
    We only prove the formula for the weight $p$ of a white-over-black vertical edge. The proof for the formula of $q$ should be similar. Let $f_\leftarrow = \binom{i}{a}$. Then, we note that $p$ can be computed by the product of all the face alternating products to its left, $\displaystyle\prod^a_{a'=0} X_{\binom{i}{a'}}$, because the horizontal edge weights of a right-canonically weighted $G_\beta$ are $1$ everywhere except for the rightmost horizontal edge weights. Substituting in the right-hand-sides of formulas in Proposition~\ref{prop: interior_face_alternating_prod} and ~\ref{prop: boundary_face_alternating_prod}, we can figure out the contribution of each $M_f$ as follows.
    
    If $\binom{i}{0}$'s right-side vertical edge is black-over-white, then $M_{\binom{i}{0}}$ appears in the denominator of $X_{\binom{i}{0}}$, and in the denominator of $X_{\binom{i}{1}}$. Thus the degree of $M_{\binom{i}{0}}$ is $-2$. If $\binom{i}{0}$'s right-side vertical edge is white-over-black, then $M_{\binom{i}{0}}$ appears in the denominator of $X_{\binom{i}{0}}$, but appears in the numerator of $X_{\binom{i}{1}}$, so the total degree of $M_{\binom{i}{0}}$ is $0$. This matches the exponents in two cases of Definition~\ref{defn: left_side_alt_prod} (i).

    For any face $\binom{i}{a'}, a > a'\geq 1$ in the $i$th row whose left and right vertical edges are of the same type, $M_{\binom{i}{a'}}$ appears in the denominator, numerator of $X_{\binom{i}{a'-1}}$, $X_{\binom{i}{a'+1}}$ exactly once and they cancel each other out, contributing to a total degree of $0$. This is consistent with Definition~\ref{defn: left_side_alt_prod} (iii)-(3).

    For any face $\binom{i}{a'}, a > a'\geq 1$ in the $i$th row whose left vertical edge is black-over-white and whose right vertical edge is white-over-black, $M_{\binom{i}{a'}}$ appears in the numerator of both $X_{\binom{i}{a'-1}}$ and $X_{\binom{i}{a'+1}}$, so its total degree is $2$, which is consistent with Definition~\ref{defn: left_side_alt_prod} (iii)-(2). Similarly, the total degree of $M_{\binom{i}{a'}}$ for a face whose left vertical edge is white-over-black and whose right vertical edge is black-over-white is $-2$, matching Definition~\ref{defn: left_side_alt_prod} (iii)-(1).

    For the face $\binom{i}{a}$, if its leftmost vertical edge is white-over-black, then $M_{\binom{i}{a}}$ appears in the denominator of $X_{\binom{i}{a-1}}$, otherwise $M_{\binom{i}{a}}$ appears in the numerator of $X_{\binom{i}{a-1}}$. Note that $M_{\binom{i}{a+1}}$ always appears in the denominator of $X_{\binom{i}{a}}$, because the vertical edge for $p$ is white-over-black. This matches the degree of $M_{\binom{i}{a}}, M_{\binom{i}{a+1}}$ in $B(f_{\leftarrow})B(f_{\rightarrow})$ by Definition~\ref{defn: left_side_alt_prod} (ii).

    Thus, so far we have checked that the degrees of $M_{\binom{i}{a'}}, \forall a'$ are the same in both $\displaystyle\prod^a_{a'=0} X_{\binom{i}{a'}}$ and $\frac{B(f_{\rightarrow})B(f_{\leftarrow})}{B(f_\uparrow)B(f_{\downarrow})}$. One can also check the degrees of $M_{\binom{i-1}{b}}, M_{\binom{i+1}{c}}, \forall b,c$ in a similar way.
\end{proof}

\printbibliography

\end{document}